\documentclass[11pt,letterpaper]{amsart}
\usepackage{amsthm}
\usepackage{amsmath}
\usepackage{amssymb}
\usepackage{amsfonts}
\usepackage{relsize}
\usepackage{graphicx}
\usepackage{mathrsfs}
\usepackage{comment}
\usepackage{tikz}
\usetikzlibrary{calc}
\usepackage{enumitem}
\usepackage{amsfonts}
\usepackage{mathtools}
\usepackage{hyperref}
\usepackage{cleveref}
\usepackage{todonotes}
\usepackage[utf8]{inputenc}
\usepackage[top=1.5in,bottom=1.5in,left=1.5in, right=1.5in]{geometry}

\newcommand{\Z}{\mathbb{Z}}
\newcommand{\C}{\mathbb{C}}

\newcommand{\bs}{\boldsymbol}
\newcommand{\D}{\mathcal{D}}
\newcommand{\mc}{\mathcal}
\newcommand{\splt}{\mathsf{split}}
\newcommand{\merge}{\mathsf{merge}}
\renewcommand{\d}{\boldsymbol{d}}
\renewcommand{\tilde}{\widetilde}

\DeclareMathOperator{\T}{\mathcal{T}}
\DeclareMathOperator{\partn}{\mathsf{str}}
\DeclareMathOperator{\toprow}{\mathsf{top}}
\DeclareMathOperator{\colr}{\mathsf{col}}
\DeclareMathOperator{\Fl}{Fl}
\DeclareMathOperator{\gr}{Gr}

\newcommand\precdot{\mathrel{\ooalign{$\prec$\cr
  \hidewidth\raise0.001ex\hbox{$\cdot\mkern0.6mu$}\cr}}}

\newtheorem{theorem}{Theorem}[section]
\newtheorem{def-prop}[theorem]{Definition-Proposition}
\newtheorem{prop}[theorem]{Proposition}
\crefname{prop}{Proposition}{Propositions}
\newtheorem{conj}[theorem]{Conjecture}
\newtheorem{lemma}[theorem]{Lemma}

\theoremstyle{definition}
\newtheorem{ex}[theorem]{Example}

\newtheorem{defin}[theorem]{Definition}

\theoremstyle{remark}
\newtheorem*{remark}{Remark}

\begin{document}
\title{Interlacing triangles, Schubert puzzles, and graph colorings}

\author{Christian Gaetz}
\address[Gaetz]{Department of Mathematics, University of California, Berkeley, CA, USA.}
\email{\href{mailto:gaetz@berkeley.edu}{{\tt gaetz@berkeley.edu}}}

\author{Yibo Gao}
\thanks{Y.G. acknowledges funding from NSFC Grant no. 12471309.}
\address[Gao]{Beijing International Center for Mathematical Research, Peking University, Beijing, China.}
\email{\href{mailto:gaoyibo@bicmr.pku.edu.cn}{{\tt gaoyibo@bicmr.pku.edu.cn}}}

\date{\today}

\begin{abstract}
We show that \emph{interlacing triangular arrays}, introduced by Aggarwal--Borodin--Wheeler to study certain probability measures, can be used to compute structure constants for multiplying Schubert classes in the $K$-theory of Grassmannians, in the cohomology of their cotangent bundles, and in the cohomology of partial flag varieties. Our results are achieved by establishing a splitting lemma, allowing for interlacing triangular arrays of high rank to be decomposed into arrays of lower rank, and by constructing a bijection between interlacing triangular arrays of rank 3 with certain proper vertex colorings of the triangular grid graph that factors through generalizations of Knutson--Tao puzzles. Along the way, we prove one enumerative conjecture of Aggarwal--Borodin--Wheeler and disprove another.
\end{abstract}

\maketitle

\section{Introduction}
\label{sec:intro}

The \textit{LLT polynomials} are a $1$-parameter family of symmetric polynomials introduced by Lascoux, Leclerc, and Thibon \cite{LLT}. They have close connections to Macdonald polynomials \cite{haglund-haiman-loehr,haglund-haiman-loehr-remmel-ulyanov} and Kazhdan--Lusztig theory \cite{grojnowski-haiman}, among other areas of representation theory and geometry.

In recent work, Aggarwal, Borodin, and Wheeler \cite{Aggarwal-Borodin-Wheeler} studied probability measures arising from the Cauchy identity for LLT polynomials. They show that these measures asymptotically split into a continuous part, given by a product of GUE corners processes, and a discrete part, supported on \textit{interlacing triangular arrays}. They conjectured that these arrays are equinumerous with vertex colorings of the triangular grid graph in the ``rank-$3$" case, and with vertex colorings of another grid graph in the rank-$4$ case. We construct a bijection between interlacing arrays and vertex colorings proving the first conjecture, and we disprove the second.

Our bijection factors through intermediate objects which are certain \textit{edge} labelings of the triangular grid graph. We recognize these edge labelings as cryptomorphic to certain Schubert calculus \emph{puzzles} \cite{Knutson-Tao-equivariant, Knutson-Tao-Woodward}. We apply various geometric interpretations of puzzles \cite{halacheva-knutson-zinn-justin,Knutson-Zinn-Justin-2,Vakil, wheeler-zinn-justin} to prove that the corresponding families of interlacing triangular arrays compute structure constants in cohomology $H^*(\gr(d,n))$ and $K$-theory $K(\gr(d,n))$ of Grassmannians, in the (localized) cohomology $H^{*\text{loc}}_{\C^{\times}}(T^*\gr(d,n))$ of their cotangent bundles, and for the multiplication in the cohomology of partial flag varieties of classes pulled back from smaller partial flag varieties.

\subsection{Interlacing triangular arrays and graph colorings}

\begin{figure}
    \centering
\begin{tikzpicture}[scale=1]
\def\sep{1.50000000000000};
\draw[green!80,thin](0.000000000000000,0.000000000000000)--(3.00000000000000,0.000000000000000);
\draw[green!80,thin](3.00000000000000,0)--(1.50000000000000,-2.59800000000000);
\draw[green!80,thin](0.000000000000000,0)--(1.50000000000000,-2.59800000000000);
\draw[green!80,thin](0.500000000000000,-0.866000000000000)--(2.50000000000000,-0.866000000000000);
\draw[green!80,thin](2.00000000000000,0)--(1.00000000000000,-1.73200000000000);
\draw[green!80,thin](1.00000000000000,0)--(2.00000000000000,-1.73200000000000);
\draw[green!80,thin](1.00000000000000,-1.73200000000000)--(2.00000000000000,-1.73200000000000);
\draw[green!80,thin](1.00000000000000,0)--(0.500000000000000,-0.866000000000000);
\draw[green!80,thin](2.00000000000000,0)--(2.50000000000000,-0.866000000000000);
\draw[green!80,thin](1.50000000000000,-2.59800000000000)--(1.50000000000000,-2.59800000000000);
\draw[green!80,thin](0.000000000000000,0)--(0.000000000000000,0.000000000000000);
\draw[green!80,thin](3.00000000000000,0)--(3.00000000000000,0.000000000000000);
\draw[green!80,thin](4.50000000000000,0.000000000000000)--(7.50000000000000,0.000000000000000);
\draw[green!80,thin](7.50000000000000,0)--(6.00000000000000,-2.59800000000000);
\draw[green!80,thin](4.50000000000000,0)--(6.00000000000000,-2.59800000000000);
\draw[green!80,thin](5.00000000000000,-0.866000000000000)--(7.00000000000000,-0.866000000000000);
\draw[green!80,thin](6.50000000000000,0)--(5.50000000000000,-1.73200000000000);
\draw[green!80,thin](5.50000000000000,0)--(6.50000000000000,-1.73200000000000);
\draw[green!80,thin](5.50000000000000,-1.73200000000000)--(6.50000000000000,-1.73200000000000);
\draw[green!80,thin](5.50000000000000,0)--(5.00000000000000,-0.866000000000000);
\draw[green!80,thin](6.50000000000000,0)--(7.00000000000000,-0.866000000000000);
\draw[green!80,thin](6.00000000000000,-2.59800000000000)--(6.00000000000000,-2.59800000000000);
\draw[green!80,thin](4.50000000000000,0)--(4.50000000000000,0.000000000000000);
\draw[green!80,thin](7.50000000000000,0)--(7.50000000000000,0.000000000000000);
\draw[green!80,thin](9.00000000000000,0.000000000000000)--(12.0000000000000,0.000000000000000);
\draw[green!80,thin](12.0000000000000,0)--(10.5000000000000,-2.59800000000000);
\draw[green!80,thin](9.00000000000000,0)--(10.5000000000000,-2.59800000000000);
\draw[green!80,thin](9.50000000000000,-0.866000000000000)--(11.5000000000000,-0.866000000000000);
\draw[green!80,thin](11.0000000000000,0)--(10.0000000000000,-1.73200000000000);
\draw[green!80,thin](10.0000000000000,0)--(11.0000000000000,-1.73200000000000);
\draw[green!80,thin](10.0000000000000,-1.73200000000000)--(11.0000000000000,-1.73200000000000);
\draw[green!80,thin](10.0000000000000,0)--(9.50000000000000,-0.866000000000000);
\draw[green!80,thin](11.0000000000000,0)--(11.5000000000000,-0.866000000000000);
\draw[green!80,thin](10.5000000000000,-2.59800000000000)--(10.5000000000000,-2.59800000000000);
\draw[green!80,thin](9.00000000000000,0)--(9.00000000000000,0.000000000000000);
\draw[green!80,thin](12.0000000000000,0)--(12.0000000000000,0.000000000000000);
\node at (0.000000000000000,-0.000000000000000) {$1$};
\node at (1.00000000000000,-0.000000000000000) {$2$};
\node at (2.00000000000000,-0.000000000000000) {$1$};
\node at (3.00000000000000,-0.000000000000000) {$3$};
\node at (0.500000000000000,-0.866000000000000) {$1$};
\node at (1.50000000000000,-0.866000000000000) {$2$};
\node at (2.50000000000000,-0.866000000000000) {$1$};
\node at (1.00000000000000,-1.73200000000000) {$1$};
\node at (2.00000000000000,-1.73200000000000) {$2$};
\node at (1.50000000000000,-2.59800000000000) {$1$};
\node at (4.50000000000000,-0.000000000000000) {$1$};
\node at (5.50000000000000,-0.000000000000000) {$3$};
\node at (6.50000000000000,-0.000000000000000) {$3$};
\node at (7.50000000000000,-0.000000000000000) {$2$};
\node at (5.00000000000000,-0.866000000000000) {$3$};
\node at (6.00000000000000,-0.866000000000000) {$3$};
\node at (7.00000000000000,-0.866000000000000) {$1$};
\node at (5.50000000000000,-1.73200000000000) {$3$};
\node at (6.50000000000000,-1.73200000000000) {$1$};
\node at (6.00000000000000,-2.59800000000000) {$3$};
\node at (9.00000000000000,-0.000000000000000) {$1$};
\node at (10.0000000000000,-0.000000000000000) {$2$};
\node at (11.0000000000000,-0.000000000000000) {$3$};
\node at (12.0000000000000,-0.000000000000000) {$2$};
\node at (9.50000000000000,-0.866000000000000) {$2$};
\node at (10.5000000000000,-0.866000000000000) {$3$};
\node at (11.5000000000000,-0.866000000000000) {$2$};
\node at (10.0000000000000,-1.73200000000000) {$3$};
\node at (11.0000000000000,-1.73200000000000) {$2$};
\node at (10.5000000000000,-2.59800000000000) {$2$};
\end{tikzpicture}
    \caption{An interlacing triangular array $T$ of rank $3$ and height $4$. }
    \label{fig:interlacing-example}
\end{figure}

Let $[m] \coloneqq \{1,\ldots,m\}$. For integers $m,n \geq 1$ and tuples $\lambda^{(i)}\in [m]^n$ for $i=1,\ldots,m$, we write $\T_{m,n}(\lambda^{(1)},\ldots,\lambda^{(m)})$ for the set of interlacing triangular arrays of rank $m$ and height $n$ with top row $\lambda^{(1)},\ldots,\lambda^{(m)}$ (see \Cref{sec:prelim-interlacing}). We write $\T_{m,n}$ for $\bigsqcup_{\bs{\lambda}} \T_{m,n}(\bs{\lambda})$.

\begin{conj}[Conj.~A.3 of Aggarwal--Borodin--Wheeler \cite{Aggarwal-Borodin-Wheeler}]
\label{conj:3-color-conjecture}   
For $n \geq 1$, we have
\[
|\T_{3,n}|=\frac{1}{4} \left|\{\text{proper vertex $4$-colorings of } \Delta_n \} \right|,
\]
where $\Delta_n$ denotes the equilateral triangular grid graph with $n$ edges on each side.
\end{conj}

Our first main theorem resolves and significantly refines and extends \Cref{conj:3-color-conjecture}.

\begin{theorem}
\label{thm:intro-rank-3-bijections}
Let $n \geq 1$ and fix $\lambda^{(1)},\lambda^{(2)},\lambda^{(3)} \in [3]^n$. Then the following sets of objects are in bijection:
\begin{itemize}
    \item[(1)] Interlacing triangular arrays $\T_{3,n}(\bs{\lambda})$ of rank $3$ with top row $\bs{\lambda}$;
    \item[(2)] $1/2/3$-puzzles $\mc{P}_n(\bs{\lambda})$ with boundary conditions $\bs{\lambda}$;
    \item[(3)] $0/1/10$-puzzles $\tilde{\mc{P}}_n(\bs{\xi})$ with boundary conditions $\bs{\xi}=\partn(\bs{\lambda})$;
    \item[(4)] Proper vertex $4$-colorings $\mc{C}_n(\bs{\kappa})$ of $\Delta_n$ with boundary colors $\bs{\kappa}=\colr({\bs{\lambda}})$.
\end{itemize}
\end{theorem}

The functions $\partn, \colr,$ and $\toprow$ are conversions between the different kinds of indexing data; their definitions can be found in \Cref{sec:triangular-coloring,sec:puzzle-conversion}. To avoid confusion, throughout the paper we draw interlacing triangular arrays in green ($\textcolor{green!80}{\blacksquare}$), $1/2/3$-puzzles in blue ($\textcolor{blue!80}{\blacksquare}$), $0/1/10$-puzzles in cyan ($\textcolor{cyan!80}{\blacksquare}$), and vertex-colored graphs in red ($\textcolor{red!80}{\blacksquare}$).

The colorings $\bigsqcup_{\bs{\kappa}=\colr(\bs{\lambda})} \mc{C}_n(\bs{\kappa})$ from \Cref{thm:intro-rank-3-bijections}(4) are exactly those proper vertex $4$-colorings of $\Delta_n$ in which the base vertex is colored with the first color. The number of these is one fourth the total number of proper vertex $4$-colorings, so \Cref{thm:intro-rank-3-bijections} implies \Cref{conj:3-color-conjecture}. 

The particular bijections underlying \Cref{thm:intro-rank-3-bijections} allow us to give geometric interpretations for certain sets of interlacing triangular arrays (see \Cref{sec:intro-geometry}).

Aggarwal, Borodin, and Wheeler also conjectured a connection between interlacing triangular arrays of rank $4$ and graph colorings. Let $\boxtimes_n$ be the graph obtained from the $n \times n$ square grid graph $\square_n$ by adding the two diagonal edges of each face of $\square_n$.

\begin{conj}[Conj.~A.5 of Aggarwal--Borodin--Wheeler \cite{Aggarwal-Borodin-Wheeler}]
\label{conj:4-color-conjecture-false}
For $n \geq 1$, we have
\begin{equation}
\label{eq:rank-4-coloring-conjecture-false}
    |\T_{4,n}| = \frac{1}{5} \left|\{\text{proper vertex $5$-colorings of } \boxtimes_n \} \right|.
\end{equation}
\end{conj}

In \Cref{thm:intro-rank-4-bijections} we give a bijection between interlacing triangular arrays of rank $4$ and certain edge labelings of the square grid graph $\square_n$, which are analogous to the $1/2/3$-puzzles of \Cref{thm:intro-rank-3-bijections}(2).

\begin{theorem}
\label{thm:intro-rank-4-bijections}
Let $n \geq 1$ and fix $\lambda^{(1)},\lambda^{(2)},\lambda^{(3)},\lambda^{(4)} \in [4]^n$. Then the following sets of objects are in bijection:
\begin{itemize}
    \item[(1)] Interlacing triangular arrays $\T_{4,n}(\bs{\lambda})$ of rank $4$ with top row $\bs{\lambda}$;
    \item[(2)] Edge labelings $\D_n(\bs{\lambda})$ of the square grid graph $\square_n$ having boundary conditions $\bs{\lambda}$ and satisfying the conditions of \Cref{sec:squares}.
\end{itemize}
\end{theorem}

Unlike in the rank-$3$ case, there is no straightforward way to biject the objects in \Cref{thm:intro-rank-4-bijections} with proper \textit{vertex} colorings. In particular, by enumerating $\D_4$ we show that \Cref{conj:4-color-conjecture-false} is false\footnote{Leonid Petrov has independently observed the failure of \Cref{conj:4-color-conjecture-false} (personal communication).}, as
\[
|\T_{4,4}| = 191232 \neq 187008 = \frac{1}{5} \left|\{\text{proper vertex $5$-colorings of } \boxtimes_4 \} \right|.
\]

\subsection{Geometric interpretations of interlacing triangular arrays}
\label{sec:intro-geometry}

The $0/1/10$-puzzles appearing in \Cref{thm:intro-rank-3-bijections}(3) are known to have various geometric interpretations when certain puzzle pieces are forbidden and when the boundary conditions are taken appropriately. The number of such puzzles with boundary conditions $\bs{\xi}=(\xi^{(1)},\xi^{(2)},\xi^{(3)})$ computes the coefficient of the basis element indexed by $\xi^{(3)}$ in the product of basis elements indexed by $\xi^{(1)}$ and $\xi^{(2)}$ in cohomology $H^*(\gr(d,n))$ and $K$-theory $K(\gr(d,n))$ of Grassmannians, in the (appropriately localized) cohomology $H^{*\text{loc}}_{\C^{\times}}(T^*\gr(d,n))$ of their cotangent bundles, and for the multiplication in the cohomology of the 2-step flag variety \cite{Knutson-Tao-Woodward, Knutson-Zinn-Justin-2, Vakil,wheeler-zinn-justin}. 

We use the specific bijections underlying \Cref{thm:intro-rank-3-bijections} to show that interlacing triangular arrays with forbidden patterns and specified top row likewise compute these coefficients. One advantage of interlacing triangular arrays is that they allow for an interpretation of coefficients in the expansion of an $(m-1)$-fold product, without the need to iteratively apply a rule for products of two elements.


For $\xi$ a $0,1$-string with content $0^d1^{n-d}$, let $G_{\xi}$ denote the class of the structure sheaf of the Schubert variety $X_{\xi} \subset \gr(d,n)$ inside $K(\gr(d,n))$. These classes can be represented by the \textit{(Grassmannian) Grothendieck polynomials}. The $\{G_{\xi}\}$ form a basis for $K(\gr(d,n))$. In \Cref{thm:k-theory} we show that the structure constants for multiplication in the basis $\{G_{\xi}\}$ are equal (up to signs) to the number of certain interlacing triangular arrays. Even the positivity of these structure constants (up to predictable signs, based on degree) is not obvious, and is due originally to Buch \cite{Buch}.

Given a tuple $\bs{\xi}=(\xi^{(1)},\ldots,\xi^{(m)})$ of $0,1$-strings, define $|\bs{\xi}|=\sum_i |\xi^{(i)}|$, where $|\xi^{(i)}|$ is the number of inversions of $\xi^{(i)}$. For $\xi$ of content $0^d1^{n-d}$, we denote by $\xi^{\perp}$ the reversed string.

\begin{theorem}
\label{thm:k-theory}
Let $\xi^{(1)},\ldots,\xi^{(m)}$ have content $0^d1^{n-d}$. Let coefficients $g_{\bs{\xi}}=g_{\xi^{(1)},\ldots,\xi^{(m)}}$ be determined by 
\[
\prod_{i=1}^{m-1} G_{\xi^{(i)}} = \sum_{\xi^{(m)}} g_{\bs{\xi}} G_{(\xi^{(m)})^{\perp}}.
\]
Then $(-1)^{d(n-d)-|\bs{\xi}|}g_{\bs{\xi}}$ is the number of interlacing triangular arrays $T$ from $\T_{m,n}(\toprow(\bs{\xi}))$ such that, for $i=2,\ldots,m-1$, $T^{(i)}$ avoids
\begin{equation}
\label{eq:k-theory-thm-forbidden}
\raisebox{-0.5\height}{
\begin{tikzpicture}
    \draw[green!80,thin] (0,0)--(-0.5,.866);
    \node at (-0.5,.866) {$m-i$};
    \node at (0,0) {$m$};
\end{tikzpicture}} \qquad and \qquad
\raisebox{-0.5\height}{
\begin{tikzpicture}
    \draw[green!80,thin] (0,0)--(-0.5,-.866)--(0.5,-.866)--(0,0);
    \node at (-0.75,-.866) {$m-i+1$};
    \node at (0.6,-.866) {$m$};
    \node at (0,0) {$m$};
\end{tikzpicture}.}
\end{equation}
\end{theorem}

Dual to $\{G_{\xi}\}$ is the basis $\{G^{\ast}_{\xi}\}$ of ideal sheaves: functions on Schubert varieties vanishing on smaller Schubert varieties. The $G^{\ast}_{\xi}$ can be represented by \textit{dual Grothendieck polynomials}.

\begin{theorem}
\label{thm:dual-k}
Let $\xi^{(1)},\ldots,\xi^{(m)}$ have content $0^d1^{n-d}$. Let coefficients $g^{\ast}_{\bs{\xi}}=g^{\ast}_{\xi^{(1)},\ldots,\xi^{(m)}}$ be determined by 
\[
\prod_{i=1}^{m-1} G^{\ast}_{\xi^{(i)}} = \sum_{\xi^{(m)}} g^{\ast}_{\bs{\xi}} G^{\ast}_{(\xi^{(m)})^{\perp}}.
\]
Then $(-1)^{d(n-d)-|\bs{\xi}|}g^{\ast}_{\bs{\xi}}$ is the number of interlacing triangular arrays $T$ from $\T_{m,n}(\toprow(\bs{\xi}))$ such that, for $i=2,\ldots,m-1$, $T^{(i)}$ avoids 
\begin{equation}
\label{eq:dual-k-thm-forbidden}
\raisebox{-0.5\height}{
\begin{tikzpicture}
    \draw[green!80,thin] (0,0)--(-0.5,.866);
    \node at (-0.5,.866) {$m$};
    \node at (0,0) {$m-i$};
\end{tikzpicture}} \qquad and \qquad
\raisebox{-0.5\height}{
\begin{tikzpicture}
    \draw[green!80,thin] (0,0)--(-0.5,.866)--(0.5,.866)--(0,0);
    \node at (-0.6,.866) {$m$};
    \node at (0.75,.866) {$m-i+1$};
    \node at (0,0) {$m$};
\end{tikzpicture}.}
\end{equation}
\end{theorem}

Interlacing triangular arrays avoiding \emph{both} the patterns (\ref{eq:k-theory-thm-forbidden}) and (\ref{eq:dual-k-thm-forbidden}) compute structure constants in the ordinary cohomology of the Grassmannian. In fact, in this case we can generalize to products of certain classes in the cohomology of arbitrary partial flag varieties.

For $\d = (0=d_0 \leq d_1 \leq \cdots \leq d_m=n)$, let $\Fl(\bs{d};n)$ denote the partial flag variety of flags of subspaces of $\C^n$ with dimension vector $\d$. Let $S_n^{\d}$ denote the set of permutations whose descents are contained in $\d$. Then $H^{\ast}(\Fl(\d;n))$ has a basis $\{\sigma_{w}\}_{w \in S_n^{\d}}$ consisting of the classes of the Schubert varieties in $\Fl(\d;n)$ (see \Cref{sec:prelim-flags}). In particular, the class $\sigma_{w_0^{\d}}$ of the longest element of $S_n^{\d}$ is the class of a point. For $w^{(1)},\ldots,w^{(m)} \in S_n^{\d}$, let coefficients $c_{\bs{w}}=c_{w^{(1)},\ldots,w^{(m)}}$ be determined by
\[
\prod_{i=1}^{m-1} \sigma_{w^{(i)}} = \sum_{w^{(m)}} c_{\bs{w}} \sigma_{(w^{(m)})^{\vee_{\d}}},
\]
where $w^{\vee_{\d}} \coloneqq w_0 w w_0(\d)$.

\begin{theorem}
\label{thm:2-step}
Let $\d = (0=d_0 \leq d_1 \leq \cdots \leq d_m=n)$. For $i \in [m]$ let $\Sigma_i=\{m-i<m<m-i+1\}$, let $\lambda^{(i)}$ be a string of type $(m-i)^{d_{m-i}}m^{d_{m-i+1}-d_{m-i}}(m-i+1)^{n-d_{m-i+1}}$, and let $w^{(i)}=w(\lambda^{(i)})_{\Sigma_i}\in S_n^{\d}$ be the corresponding permutation (see \Cref{def:Schubert-string-permutation}). Then $c_{\bs{w}}$ is the number of interlacing triangular arrays $T$ from $\T_{m,n}(\bs{\lambda})$ such that, for $i=2,
\ldots,m-1$, $T^{(i)}$ avoids the patterns from (\ref{eq:k-theory-thm-forbidden}) and (\ref{eq:dual-k-thm-forbidden}).
\end{theorem}
The Schubert classes $\sigma_{w^{(i)}}$ appearing in \Cref{thm:2-step} are the pullbacks of Schubert classes under the projection from $\Fl(\d;n)$ to certain $2$-step flag varieties. It is not clear how one might hope to extend \Cref{thm:2-step} to provide rules for products of more general Schubert classes; combinatorial rules for products of arbitrary Schubert classes are known only for flag varieties with $m \leq 4$ steps (see, e.g. \cite{Buch-Kresch-Purbhoo-Tamvakis,Knutson-Zinn-Justin-1,Knutson-Zinn-Justin-2}).

Finally, let $H^{*\text{loc}}_{\C^{\times}}(T^*\gr(d,n))$ denote the equivariant cohomology of the cotangent bundle of $\gr(d,n)$ with respect to the $\C^{\times}$-action scaling the cotangent spaces, localized as in \cite[\S2.2]{Knutson-Zinn-Justin-2}. For $\xi$ of content $0^d1^{n-d}$, let $S_{\xi} \in H^{*\text{loc}}_{\C^{\times}}(T^*\gr(d,n))$ denote the \emph{Segre--Schwartz--MacPherson (SSM) class} of the corresponding Schubert variety, using the conventions of \cite[\S2.4 \& \S5.2]{Knutson-Zinn-Justin-2} (see also \cite{Feher-Rimanyi, Macpherson}). In our last main theorem, we show that interlacing triangular arrays, with \emph{no} forbidden patterns, compute structure constants for the $\{S_{\xi}\}$.

\begin{theorem}
\label{thm:ssm}
Let $\xi^{(1)},\ldots,\xi^{(m)}$ have content $0^d1^{n-d}$. Let coefficients $s_{\bs{\xi}}=s_{\xi^{(1)},\ldots,\xi^{(m)}}$ be determined by 
\[
\prod_{i=1}^{m-1} S_{\xi^{(i)}} = \sum_{\xi^{(m)}} s_{\bs{\xi}} S_{(\xi^{(m)})^{\perp}}.
\]
Then $(-1)^{d(n-d)-|\bs{\xi}|}s_{\bs{\xi}}$ is the cardinality of $\T_{m,n}(\toprow(\bs{\xi}))$.
\end{theorem}

\subsection{Examples of the geometric interpretations}

\begin{ex}\label{ex:grassmannian-}
Let $m=4$ and $n=4$. Consider the $0,1$-strings $\xi^{(1)}=\xi^{(2)}=\xi^{(3)}=0101$ with length $1$. Correspondingly, $\lambda^{(1)}=\toprow(\xi^{(1)})=3434$, and analogously, $\lambda^{(2)}=2323$ and $\lambda^{(3)}=1212$. They also correspond to the permutation $1324$ and the partition with one box. As an example to \Cref{thm:2-step}, in $H^*(\gr(2,4))$, $\sigma_{1324}^3=2\sigma_{2413}$, whose $0,1$-string is $1010$ and the coefficient $2$ is given by the interlacing triangular arrays in Figure~\ref{fig:Gr24-interlacing-example}.
\begin{figure}[h!]
\centering
\begin{tikzpicture}[scale=0.5]
\def\sep{2};
\draw[green,thin](0.0,0.0)--(3.0,0.0);
\draw[green,thin](3,0)--(1.5,-2.598);
\draw[green,thin](0,0)--(1.5,-2.598);
\draw[green,thin](0.5,-0.866)--(2.5,-0.866);
\draw[green,thin](2,0)--(1.0,-1.7319999999999998);
\draw[green,thin](1,0)--(2.0,-1.7319999999999998);
\draw[green,thin](1.0,-1.732)--(2.0,-1.732);
\draw[green,thin](1,0)--(0.5,-0.8659999999999999);
\draw[green,thin](2,0)--(2.5,-0.8659999999999999);
\draw[green,thin](1.5,-2.598)--(1.5,-2.598);
\draw[green,thin](0,0)--(0.0,0.0);
\draw[green,thin](3,0)--(3.0,0.0);
\draw[green,thin](5.0,0.0)--(8.0,0.0);
\draw[green,thin](8,0)--(6.5,-2.598);
\draw[green,thin](5,0)--(6.5,-2.598);
\draw[green,thin](5.5,-0.866)--(7.5,-0.866);
\draw[green,thin](7,0)--(6.0,-1.7319999999999998);
\draw[green,thin](6,0)--(7.0,-1.7319999999999998);
\draw[green,thin](6.0,-1.732)--(7.0,-1.732);
\draw[green,thin](6,0)--(5.5,-0.8659999999999999);
\draw[green,thin](7,0)--(7.5,-0.8659999999999999);
\draw[green,thin](6.5,-2.598)--(6.5,-2.598);
\draw[green,thin](5,0)--(5.0,0.0);
\draw[green,thin](8,0)--(8.0,0.0);
\draw[green,thin](10.0,0.0)--(13.0,0.0);
\draw[green,thin](13,0)--(11.5,-2.598);
\draw[green,thin](10,0)--(11.5,-2.598);
\draw[green,thin](10.5,-0.866)--(12.5,-0.866);
\draw[green,thin](12,0)--(11.0,-1.7319999999999998);
\draw[green,thin](11,0)--(12.0,-1.7319999999999998);
\draw[green,thin](11.0,-1.732)--(12.0,-1.732);
\draw[green,thin](11,0)--(10.5,-0.8659999999999999);
\draw[green,thin](12,0)--(12.5,-0.8659999999999999);
\draw[green,thin](11.5,-2.598)--(11.5,-2.598);
\draw[green,thin](10,0)--(10.0,0.0);
\draw[green,thin](13,0)--(13.0,0.0);
\draw[green,thin](15.0,0.0)--(18.0,0.0);
\draw[green,thin](18,0)--(16.5,-2.598);
\draw[green,thin](15,0)--(16.5,-2.598);
\draw[green,thin](15.5,-0.866)--(17.5,-0.866);
\draw[green,thin](17,0)--(16.0,-1.7319999999999998);
\draw[green,thin](16,0)--(17.0,-1.7319999999999998);
\draw[green,thin](16.0,-1.732)--(17.0,-1.732);
\draw[green,thin](16,0)--(15.5,-0.8659999999999999);
\draw[green,thin](17,0)--(17.5,-0.8659999999999999);
\draw[green,thin](16.5,-2.598)--(16.5,-2.598);
\draw[green,thin](15,0)--(15.0,0.0);
\draw[green,thin](18,0)--(18.0,0.0);
\node at (0.0,-0.0) {$3$};
\node at (1.0,-0.0) {$4$};
\node at (2.0,-0.0) {$3$};
\node at (3.0,-0.0) {$4$};
\node at (0.5,-0.866) {$3$};
\node at (1.5,-0.866) {$4$};
\node at (2.5,-0.866) {$3$};
\node at (1.0,-1.732) {$3$};
\node at (2.0,-1.732) {$4$};
\node at (1.5,-2.598) {$3$};
\node at (5.0,-0.0) {$2$};
\node at (6.0,-0.0) {$3$};
\node at (7.0,-0.0) {$2$};
\node at (8.0,-0.0) {$3$};
\node at (5.5,-0.866) {$2$};
\node at (6.5,-0.866) {$3$};
\node at (7.5,-0.866) {$2$};
\node at (6.0,-1.732) {$3$};
\node at (7.0,-1.732) {$2$};
\node at (6.5,-2.598) {$4$};
\node at (10.0,-0.0) {$1$};
\node at (11.0,-0.0) {$2$};
\node at (12.0,-0.0) {$1$};
\node at (13.0,-0.0) {$2$};
\node at (10.5,-0.866) {$1$};
\node at (11.5,-0.866) {$4$};
\node at (12.5,-0.866) {$2$};
\node at (11.0,-1.732) {$1$};
\node at (12.0,-1.732) {$2$};
\node at (11.5,-2.598) {$2$};
\node at (15.0,-0.0) {$4$};
\node at (16.0,-0.0) {$1$};
\node at (17.0,-0.0) {$4$};
\node at (18.0,-0.0) {$1$};
\node at (15.5,-0.866) {$1$};
\node at (16.5,-0.866) {$4$};
\node at (17.5,-0.866) {$1$};
\node at (16.0,-1.732) {$4$};
\node at (17.0,-1.732) {$1$};
\node at (16.5,-2.598) {$1$};
\end{tikzpicture}

\begin{tikzpicture}[scale=0.5]
\def\sep{2};
\draw[green,thin](0.0,0.0)--(3.0,0.0);
\draw[green,thin](3,0)--(1.5,-2.598);
\draw[green,thin](0,0)--(1.5,-2.598);
\draw[green,thin](0.5,-0.866)--(2.5,-0.866);
\draw[green,thin](2,0)--(1.0,-1.7319999999999998);
\draw[green,thin](1,0)--(2.0,-1.7319999999999998);
\draw[green,thin](1.0,-1.732)--(2.0,-1.732);
\draw[green,thin](1,0)--(0.5,-0.8659999999999999);
\draw[green,thin](2,0)--(2.5,-0.8659999999999999);
\draw[green,thin](1.5,-2.598)--(1.5,-2.598);
\draw[green,thin](0,0)--(0.0,0.0);
\draw[green,thin](3,0)--(3.0,0.0);
\draw[green,thin](5.0,0.0)--(8.0,0.0);
\draw[green,thin](8,0)--(6.5,-2.598);
\draw[green,thin](5,0)--(6.5,-2.598);
\draw[green,thin](5.5,-0.866)--(7.5,-0.866);
\draw[green,thin](7,0)--(6.0,-1.7319999999999998);
\draw[green,thin](6,0)--(7.0,-1.7319999999999998);
\draw[green,thin](6.0,-1.732)--(7.0,-1.732);
\draw[green,thin](6,0)--(5.5,-0.8659999999999999);
\draw[green,thin](7,0)--(7.5,-0.8659999999999999);
\draw[green,thin](6.5,-2.598)--(6.5,-2.598);
\draw[green,thin](5,0)--(5.0,0.0);
\draw[green,thin](8,0)--(8.0,0.0);
\draw[green,thin](10.0,0.0)--(13.0,0.0);
\draw[green,thin](13,0)--(11.5,-2.598);
\draw[green,thin](10,0)--(11.5,-2.598);
\draw[green,thin](10.5,-0.866)--(12.5,-0.866);
\draw[green,thin](12,0)--(11.0,-1.7319999999999998);
\draw[green,thin](11,0)--(12.0,-1.7319999999999998);
\draw[green,thin](11.0,-1.732)--(12.0,-1.732);
\draw[green,thin](11,0)--(10.5,-0.8659999999999999);
\draw[green,thin](12,0)--(12.5,-0.8659999999999999);
\draw[green,thin](11.5,-2.598)--(11.5,-2.598);
\draw[green,thin](10,0)--(10.0,0.0);
\draw[green,thin](13,0)--(13.0,0.0);
\draw[green,thin](15.0,0.0)--(18.0,0.0);
\draw[green,thin](18,0)--(16.5,-2.598);
\draw[green,thin](15,0)--(16.5,-2.598);
\draw[green,thin](15.5,-0.866)--(17.5,-0.866);
\draw[green,thin](17,0)--(16.0,-1.7319999999999998);
\draw[green,thin](16,0)--(17.0,-1.7319999999999998);
\draw[green,thin](16.0,-1.732)--(17.0,-1.732);
\draw[green,thin](16,0)--(15.5,-0.8659999999999999);
\draw[green,thin](17,0)--(17.5,-0.8659999999999999);
\draw[green,thin](16.5,-2.598)--(16.5,-2.598);
\draw[green,thin](15,0)--(15.0,0.0);
\draw[green,thin](18,0)--(18.0,0.0);
\node at (0.0,-0.0) {$3$};
\node at (1.0,-0.0) {$4$};
\node at (2.0,-0.0) {$3$};
\node at (3.0,-0.0) {$4$};
\node at (0.5,-0.866) {$3$};
\node at (1.5,-0.866) {$4$};
\node at (2.5,-0.866) {$3$};
\node at (1.0,-1.732) {$3$};
\node at (2.0,-1.732) {$4$};
\node at (1.5,-2.598) {$3$};
\node at (5.0,-0.0) {$2$};
\node at (6.0,-0.0) {$3$};
\node at (7.0,-0.0) {$2$};
\node at (8.0,-0.0) {$3$};
\node at (5.5,-0.866) {$2$};
\node at (6.5,-0.866) {$4$};
\node at (7.5,-0.866) {$3$};
\node at (6.0,-1.732) {$2$};
\node at (7.0,-1.732) {$3$};
\node at (6.5,-2.598) {$2$};
\node at (10.0,-0.0) {$1$};
\node at (11.0,-0.0) {$2$};
\node at (12.0,-0.0) {$1$};
\node at (13.0,-0.0) {$2$};
\node at (10.5,-0.866) {$2$};
\node at (11.5,-0.866) {$1$};
\node at (12.5,-0.866) {$2$};
\node at (11.0,-1.732) {$2$};
\node at (12.0,-1.732) {$1$};
\node at (11.5,-2.598) {$4$};
\node at (15.0,-0.0) {$4$};
\node at (16.0,-0.0) {$1$};
\node at (17.0,-0.0) {$4$};
\node at (18.0,-0.0) {$1$};
\node at (15.5,-0.866) {$1$};
\node at (16.5,-0.866) {$4$};
\node at (17.5,-0.866) {$1$};
\node at (16.0,-1.732) {$4$};
\node at (17.0,-1.732) {$1$};
\node at (16.5,-2.598) {$1$};
\end{tikzpicture}
\caption{Interlacing triangular arrays with top row $\lambda^{(1)}=3434,\lambda^{(2)}=2323,\lambda^{(3)}=1212,\lambda^{(4)}=4141$.}
\label{fig:Gr24-interlacing-example}
\end{figure}

There are more interlacing triangular arrays whose top row starts with $\lambda^{(1)}=3434,\lambda^{(2)}=2323,\lambda^{(3)}=1212$. Besides the ones in Figure~\ref{fig:Gr24-interlacing-example}, all the others have top row $\lambda^{(4)}=4411$ shown in Figure~\ref{fig:Gr24-interlacing-example-more}. The first five of them contain patterns from (\ref{eq:k-theory-thm-forbidden}) and the last one contains patterns from (\ref{eq:dual-k-thm-forbidden}), highlighted in the figure. 
\begin{figure}[h!]
\centering
\begin{tikzpicture}[scale=0.5]
\def\sep{2};
\draw[green,thin](0.0,0.0)--(3.0,0.0);
\draw[green,thin](3,0)--(1.5,-2.598);
\draw[green,thin](0,0)--(1.5,-2.598);
\draw[green,thin](0.5,-0.866)--(2.5,-0.866);
\draw[green,thin](2,0)--(1.0,-1.7319999999999998);
\draw[green,thin](1,0)--(2.0,-1.7319999999999998);
\draw[green,thin](1.0,-1.732)--(2.0,-1.732);
\draw[green,thin](1,0)--(0.5,-0.8659999999999999);
\draw[green,thin](2,0)--(2.5,-0.8659999999999999);
\draw[green,thin](1.5,-2.598)--(1.5,-2.598);
\draw[green,thin](0,0)--(0.0,0.0);
\draw[green,thin](3,0)--(3.0,0.0);
\draw[green,thin](5.0,0.0)--(8.0,0.0);
\draw[green,thin](8,0)--(6.5,-2.598);
\draw[green,thin](5,0)--(6.5,-2.598);
\draw[green,thin](5.5,-0.866)--(7.5,-0.866);
\draw[green,thin](7,0)--(6.0,-1.7319999999999998);
\draw[green,thin](6,0)--(7.0,-1.7319999999999998);
\draw[green,thin](6.0,-1.732)--(7.0,-1.732);
\draw[green,thin](6,0)--(5.5,-0.8659999999999999);
\draw[green,thin](7,0)--(7.5,-0.8659999999999999);
\draw[green,thin](6.5,-2.598)--(6.5,-2.598);
\draw[green,thin](5,0)--(5.0,0.0);
\draw[green,thin](8,0)--(8.0,0.0);
\draw[green,thin](10.0,0.0)--(13.0,0.0);
\draw[green,thin](13,0)--(11.5,-2.598);
\draw[green,thin](10,0)--(11.5,-2.598);
\draw[green,thin](10.5,-0.866)--(12.5,-0.866);
\draw[green,thin](12,0)--(11.0,-1.7319999999999998);
\draw[green,thin](11,0)--(12.0,-1.7319999999999998);
\draw[green,thin](11.0,-1.732)--(12.0,-1.732);
\draw[green,thin](11,0)--(10.5,-0.8659999999999999);
\draw[green,thin](12,0)--(12.5,-0.8659999999999999);
\draw[green,thin](11.5,-2.598)--(11.5,-2.598);
\draw[green,thin](10,0)--(10.0,0.0);
\draw[green,thin](13,0)--(13.0,0.0);
\draw[green,thin](15.0,0.0)--(18.0,0.0);
\draw[green,thin](18,0)--(16.5,-2.598);
\draw[green,thin](15,0)--(16.5,-2.598);
\draw[green,thin](15.5,-0.866)--(17.5,-0.866);
\draw[green,thin](17,0)--(16.0,-1.7319999999999998);
\draw[green,thin](16,0)--(17.0,-1.7319999999999998);
\draw[green,thin](16.0,-1.732)--(17.0,-1.732);
\draw[green,thin](16,0)--(15.5,-0.8659999999999999);
\draw[green,thin](17,0)--(17.5,-0.8659999999999999);
\draw[green,thin](16.5,-2.598)--(16.5,-2.598);
\draw[green,thin](15,0)--(15.0,0.0);
\draw[green,thin](18,0)--(18.0,0.0);
\node at (0.0,-0.0) {$3$};
\node at (1.0,-0.0) {$4$};
\node at (2.0,-0.0) {$3$};
\node at (3.0,-0.0) {$4$};
\node at (0.5,-0.866) {$3$};
\node at (1.5,-0.866) {$4$};
\node at (2.5,-0.866) {$3$};
\node at (1.0,-1.732) {$3$};
\node at (2.0,-1.732) {$4$};
\node at (1.5,-2.598) {$3$};
\node at (5.0,-0.0) {$2$};
\node at (6.0,-0.0) {$3$};
\node at (7.0,-0.0) {$2$};
\node at (8.0,-0.0) {$3$};
\node at (5.5,-0.866) {$2$};
\node at (6.5,-0.866) {$3$};
\node at (7.5,-0.866) {$2$};
\node at (6.0,-1.732) {$3$};
\node at (7.0,-1.732) {$2$};
\node at (6.5,-2.598) {$4$};
\draw (10.0,-0.0) circle [radius=0.4];
\draw (10.5,-0.866) circle [radius=0.4];
\draw[ultra thick,green](10.0,-0.0)--(10.5,-0.866);
\node at (10.0,-0.0) {$1$};
\node at (11.0,-0.0) {$2$};
\node at (12.0,-0.0) {$1$};
\node at (13.0,-0.0) {$2$};
\node at (10.5,-0.866) {$4$};
\node at (11.5,-0.866) {$1$};
\node at (12.5,-0.866) {$2$};
\node at (11.0,-1.732) {$4$};
\node at (12.0,-1.732) {$2$};
\node at (11.5,-2.598) {$2$};
\node at (15.0,-0.0) {$4$};
\node at (16.0,-0.0) {$4$};
\node at (17.0,-0.0) {$1$};
\node at (18.0,-0.0) {$1$};
\node at (15.5,-0.866) {$4$};
\node at (16.5,-0.866) {$1$};
\node at (17.5,-0.866) {$1$};
\node at (16.0,-1.732) {$1$};
\node at (17.0,-1.732) {$1$};
\node at (16.5,-2.598) {$1$};
\end{tikzpicture}
\begin{tikzpicture}[scale=0.5]
\def\sep{2};
\draw[green,thin](0.0,0.0)--(3.0,0.0);
\draw[green,thin](3,0)--(1.5,-2.598);
\draw[green,thin](0,0)--(1.5,-2.598);
\draw[green,thin](0.5,-0.866)--(2.5,-0.866);
\draw[green,thin](2,0)--(1.0,-1.7319999999999998);
\draw[green,thin](1,0)--(2.0,-1.7319999999999998);
\draw[green,thin](1.0,-1.732)--(2.0,-1.732);
\draw[green,thin](1,0)--(0.5,-0.8659999999999999);
\draw[green,thin](2,0)--(2.5,-0.8659999999999999);
\draw[green,thin](1.5,-2.598)--(1.5,-2.598);
\draw[green,thin](0,0)--(0.0,0.0);
\draw[green,thin](3,0)--(3.0,0.0);
\draw[green,thin](5.0,0.0)--(8.0,0.0);
\draw[green,thin](8,0)--(6.5,-2.598);
\draw[green,thin](5,0)--(6.5,-2.598);
\draw[green,thin](5.5,-0.866)--(7.5,-0.866);
\draw[green,thin](7,0)--(6.0,-1.7319999999999998);
\draw[green,thin](6,0)--(7.0,-1.7319999999999998);
\draw[green,thin](6.0,-1.732)--(7.0,-1.732);
\draw[green,thin](6,0)--(5.5,-0.8659999999999999);
\draw[green,thin](7,0)--(7.5,-0.8659999999999999);
\draw[green,thin](6.5,-2.598)--(6.5,-2.598);
\draw[green,thin](5,0)--(5.0,0.0);
\draw[green,thin](8,0)--(8.0,0.0);
\draw[green,thin](10.0,0.0)--(13.0,0.0);
\draw[green,thin](13,0)--(11.5,-2.598);
\draw[green,thin](10,0)--(11.5,-2.598);
\draw[green,thin](10.5,-0.866)--(12.5,-0.866);
\draw[green,thin](12,0)--(11.0,-1.7319999999999998);
\draw[green,thin](11,0)--(12.0,-1.7319999999999998);
\draw[green,thin](11.0,-1.732)--(12.0,-1.732);
\draw[green,thin](11,0)--(10.5,-0.8659999999999999);
\draw[green,thin](12,0)--(12.5,-0.8659999999999999);
\draw[green,thin](11.5,-2.598)--(11.5,-2.598);
\draw[green,thin](10,0)--(10.0,0.0);
\draw[green,thin](13,0)--(13.0,0.0);
\draw[green,thin](15.0,0.0)--(18.0,0.0);
\draw[green,thin](18,0)--(16.5,-2.598);
\draw[green,thin](15,0)--(16.5,-2.598);
\draw[green,thin](15.5,-0.866)--(17.5,-0.866);
\draw[green,thin](17,0)--(16.0,-1.7319999999999998);
\draw[green,thin](16,0)--(17.0,-1.7319999999999998);
\draw[green,thin](16.0,-1.732)--(17.0,-1.732);
\draw[green,thin](16,0)--(15.5,-0.8659999999999999);
\draw[green,thin](17,0)--(17.5,-0.8659999999999999);
\draw[green,thin](16.5,-2.598)--(16.5,-2.598);
\draw[green,thin](15,0)--(15.0,0.0);
\draw[green,thin](18,0)--(18.0,0.0);
\node at (0.0,-0.0) {$3$};
\node at (1.0,-0.0) {$4$};
\node at (2.0,-0.0) {$3$};
\node at (3.0,-0.0) {$4$};
\node at (0.5,-0.866) {$3$};
\node at (1.5,-0.866) {$4$};
\node at (2.5,-0.866) {$3$};
\node at (1.0,-1.732) {$3$};
\node at (2.0,-1.732) {$4$};
\node at (1.5,-2.598) {$3$};
\draw[ultra thick,green] (7.5,-0.866)--(7.0,-0.0);
\draw (7.5,-0.866) circle [radius=0.4];
\draw (7.0,-0.0) circle [radius=0.4];
\node at (5.0,-0.0) {$2$};
\node at (6.0,-0.0) {$3$};
\node at (7.0,-0.0) {$2$};
\node at (8.0,-0.0) {$3$};
\node at (5.5,-0.866) {$2$};
\node at (6.5,-0.866) {$3$};
\node at (7.5,-0.866) {$4$};
\node at (6.0,-1.732) {$3$};
\node at (7.0,-1.732) {$2$};
\node at (6.5,-2.598) {$4$};
\node at (10.0,-0.0) {$1$};
\node at (11.0,-0.0) {$2$};
\node at (12.0,-0.0) {$1$};
\node at (13.0,-0.0) {$2$};
\node at (10.5,-0.866) {$2$};
\node at (11.5,-0.866) {$1$};
\node at (12.5,-0.866) {$2$};
\node at (11.0,-1.732) {$4$};
\node at (12.0,-1.732) {$2$};
\node at (11.5,-2.598) {$2$};
\node at (15.0,-0.0) {$4$};
\node at (16.0,-0.0) {$4$};
\node at (17.0,-0.0) {$1$};
\node at (18.0,-0.0) {$1$};
\node at (15.5,-0.866) {$4$};
\node at (16.5,-0.866) {$1$};
\node at (17.5,-0.866) {$1$};
\node at (16.0,-1.732) {$1$};
\node at (17.0,-1.732) {$1$};
\node at (16.5,-2.598) {$1$};
\end{tikzpicture}
\begin{tikzpicture}[scale=0.5]
\def\sep{2};
\draw[green,thin](0.0,0.0)--(3.0,0.0);
\draw[green,thin](3,0)--(1.5,-2.598);
\draw[green,thin](0,0)--(1.5,-2.598);
\draw[green,thin](0.5,-0.866)--(2.5,-0.866);
\draw[green,thin](2,0)--(1.0,-1.7319999999999998);
\draw[green,thin](1,0)--(2.0,-1.7319999999999998);
\draw[green,thin](1.0,-1.732)--(2.0,-1.732);
\draw[green,thin](1,0)--(0.5,-0.8659999999999999);
\draw[green,thin](2,0)--(2.5,-0.8659999999999999);
\draw[green,thin](1.5,-2.598)--(1.5,-2.598);
\draw[green,thin](0,0)--(0.0,0.0);
\draw[green,thin](3,0)--(3.0,0.0);
\draw[green,thin](5.0,0.0)--(8.0,0.0);
\draw[green,thin](8,0)--(6.5,-2.598);
\draw[green,thin](5,0)--(6.5,-2.598);
\draw[green,thin](5.5,-0.866)--(7.5,-0.866);
\draw[green,thin](7,0)--(6.0,-1.7319999999999998);
\draw[green,thin](6,0)--(7.0,-1.7319999999999998);
\draw[green,thin](6.0,-1.732)--(7.0,-1.732);
\draw[green,thin](6,0)--(5.5,-0.8659999999999999);
\draw[green,thin](7,0)--(7.5,-0.8659999999999999);
\draw[green,thin](6.5,-2.598)--(6.5,-2.598);
\draw[green,thin](5,0)--(5.0,0.0);
\draw[green,thin](8,0)--(8.0,0.0);
\draw[green,thin](10.0,0.0)--(13.0,0.0);
\draw[green,thin](13,0)--(11.5,-2.598);
\draw[green,thin](10,0)--(11.5,-2.598);
\draw[green,thin](10.5,-0.866)--(12.5,-0.866);
\draw[green,thin](12,0)--(11.0,-1.7319999999999998);
\draw[green,thin](11,0)--(12.0,-1.7319999999999998);
\draw[green,thin](11.0,-1.732)--(12.0,-1.732);
\draw[green,thin](11,0)--(10.5,-0.8659999999999999);
\draw[green,thin](12,0)--(12.5,-0.8659999999999999);
\draw[green,thin](11.5,-2.598)--(11.5,-2.598);
\draw[green,thin](10,0)--(10.0,0.0);
\draw[green,thin](13,0)--(13.0,0.0);
\draw[green,thin](15.0,0.0)--(18.0,0.0);
\draw[green,thin](18,0)--(16.5,-2.598);
\draw[green,thin](15,0)--(16.5,-2.598);
\draw[green,thin](15.5,-0.866)--(17.5,-0.866);
\draw[green,thin](17,0)--(16.0,-1.7319999999999998);
\draw[green,thin](16,0)--(17.0,-1.7319999999999998);
\draw[green,thin](16.0,-1.732)--(17.0,-1.732);
\draw[green,thin](16,0)--(15.5,-0.8659999999999999);
\draw[green,thin](17,0)--(17.5,-0.8659999999999999);
\draw[green,thin](16.5,-2.598)--(16.5,-2.598);
\draw[green,thin](15,0)--(15.0,0.0);
\draw[green,thin](18,0)--(18.0,0.0);
\node at (0.0,-0.0) {$3$};
\node at (1.0,-0.0) {$4$};
\node at (2.0,-0.0) {$3$};
\node at (3.0,-0.0) {$4$};
\node at (0.5,-0.866) {$3$};
\node at (1.5,-0.866) {$4$};
\node at (2.5,-0.866) {$3$};
\node at (1.0,-1.732) {$3$};
\node at (2.0,-1.732) {$4$};
\node at (1.5,-2.598) {$3$};
\node at (5.0,-0.0) {$2$};
\node at (6.0,-0.0) {$3$};
\node at (7.0,-0.0) {$2$};
\node at (8.0,-0.0) {$3$};
\node at (5.5,-0.866) {$2$};
\node at (6.5,-0.866) {$4$};
\node at (7.5,-0.866) {$3$};
\node at (6.0,-1.732) {$2$};
\node at (7.0,-1.732) {$3$};
\node at (6.5,-2.598) {$2$};
\draw[ultra thick,green] (11.5,-0.866)--(12.0,-1.732);
\draw (11.5,-0.866) circle [radius=0.4];
\draw (12.0,-1.732) circle [radius=0.4];
\node at (10.0,-0.0) {$1$};
\node at (11.0,-0.0) {$2$};
\node at (12.0,-0.0) {$1$};
\node at (13.0,-0.0) {$2$};
\node at (10.5,-0.866) {$2$};
\node at (11.5,-0.866) {$1$};
\node at (12.5,-0.866) {$2$};
\node at (11.0,-1.732) {$2$};
\node at (12.0,-1.732) {$4$};
\node at (11.5,-2.598) {$4$};
\node at (15.0,-0.0) {$4$};
\node at (16.0,-0.0) {$4$};
\node at (17.0,-0.0) {$1$};
\node at (18.0,-0.0) {$1$};
\node at (15.5,-0.866) {$4$};
\node at (16.5,-0.866) {$1$};
\node at (17.5,-0.866) {$1$};
\node at (16.0,-1.732) {$1$};
\node at (17.0,-1.732) {$1$};
\node at (16.5,-2.598) {$1$};
\end{tikzpicture}
\begin{tikzpicture}[scale=0.5]
\def\sep{2};
\draw[green,thin](0.0,0.0)--(3.0,0.0);
\draw[green,thin](3,0)--(1.5,-2.598);
\draw[green,thin](0,0)--(1.5,-2.598);
\draw[green,thin](0.5,-0.866)--(2.5,-0.866);
\draw[green,thin](2,0)--(1.0,-1.7319999999999998);
\draw[green,thin](1,0)--(2.0,-1.7319999999999998);
\draw[green,thin](1.0,-1.732)--(2.0,-1.732);
\draw[green,thin](1,0)--(0.5,-0.8659999999999999);
\draw[green,thin](2,0)--(2.5,-0.8659999999999999);
\draw[green,thin](1.5,-2.598)--(1.5,-2.598);
\draw[green,thin](0,0)--(0.0,0.0);
\draw[green,thin](3,0)--(3.0,0.0);
\draw[green,thin](5.0,0.0)--(8.0,0.0);
\draw[green,thin](8,0)--(6.5,-2.598);
\draw[green,thin](5,0)--(6.5,-2.598);
\draw[green,thin](5.5,-0.866)--(7.5,-0.866);
\draw[green,thin](7,0)--(6.0,-1.7319999999999998);
\draw[green,thin](6,0)--(7.0,-1.7319999999999998);
\draw[green,thin](6.0,-1.732)--(7.0,-1.732);
\draw[green,thin](6,0)--(5.5,-0.8659999999999999);
\draw[green,thin](7,0)--(7.5,-0.8659999999999999);
\draw[green,thin](6.5,-2.598)--(6.5,-2.598);
\draw[green,thin](5,0)--(5.0,0.0);
\draw[green,thin](8,0)--(8.0,0.0);
\draw[green,thin](10.0,0.0)--(13.0,0.0);
\draw[green,thin](13,0)--(11.5,-2.598);
\draw[green,thin](10,0)--(11.5,-2.598);
\draw[green,thin](10.5,-0.866)--(12.5,-0.866);
\draw[green,thin](12,0)--(11.0,-1.7319999999999998);
\draw[green,thin](11,0)--(12.0,-1.7319999999999998);
\draw[green,thin](11.0,-1.732)--(12.0,-1.732);
\draw[green,thin](11,0)--(10.5,-0.8659999999999999);
\draw[green,thin](12,0)--(12.5,-0.8659999999999999);
\draw[green,thin](11.5,-2.598)--(11.5,-2.598);
\draw[green,thin](10,0)--(10.0,0.0);
\draw[green,thin](13,0)--(13.0,0.0);
\draw[green,thin](15.0,0.0)--(18.0,0.0);
\draw[green,thin](18,0)--(16.5,-2.598);
\draw[green,thin](15,0)--(16.5,-2.598);
\draw[green,thin](15.5,-0.866)--(17.5,-0.866);
\draw[green,thin](17,0)--(16.0,-1.7319999999999998);
\draw[green,thin](16,0)--(17.0,-1.7319999999999998);
\draw[green,thin](16.0,-1.732)--(17.0,-1.732);
\draw[green,thin](16,0)--(15.5,-0.8659999999999999);
\draw[green,thin](17,0)--(17.5,-0.8659999999999999);
\draw[green,thin](16.5,-2.598)--(16.5,-2.598);
\draw[green,thin](15,0)--(15.0,0.0);
\draw[green,thin](18,0)--(18.0,0.0);
\node at (0.0,-0.0) {$3$};
\node at (1.0,-0.0) {$4$};
\node at (2.0,-0.0) {$3$};
\node at (3.0,-0.0) {$4$};
\node at (0.5,-0.866) {$3$};
\node at (1.5,-0.866) {$4$};
\node at (2.5,-0.866) {$3$};
\node at (1.0,-1.732) {$3$};
\node at (2.0,-1.732) {$4$};
\node at (1.5,-2.598) {$3$};
\draw[ultra thick,green] (5.0,-0.0)--(5.5,-0.866);
\draw (5.0,-0.0) circle [radius=0.4];
\draw (5.5,-0.866) circle [radius=0.4];
\node at (5.0,-0.0) {$2$};
\node at (6.0,-0.0) {$3$};
\node at (7.0,-0.0) {$2$};
\node at (8.0,-0.0) {$3$};
\node at (5.5,-0.866) {$4$};
\node at (6.5,-0.866) {$2$};
\node at (7.5,-0.866) {$3$};
\node at (6.0,-1.732) {$3$};
\node at (7.0,-1.732) {$2$};
\node at (6.5,-2.598) {$4$};
\node at (10.0,-0.0) {$1$};
\node at (11.0,-0.0) {$2$};
\node at (12.0,-0.0) {$1$};
\node at (13.0,-0.0) {$2$};
\node at (10.5,-0.866) {$2$};
\node at (11.5,-0.866) {$1$};
\node at (12.5,-0.866) {$2$};
\node at (11.0,-1.732) {$4$};
\node at (12.0,-1.732) {$2$};
\node at (11.5,-2.598) {$2$};
\node at (15.0,-0.0) {$4$};
\node at (16.0,-0.0) {$4$};
\node at (17.0,-0.0) {$1$};
\node at (18.0,-0.0) {$1$};
\node at (15.5,-0.866) {$4$};
\node at (16.5,-0.866) {$1$};
\node at (17.5,-0.866) {$1$};
\node at (16.0,-1.732) {$1$};
\node at (17.0,-1.732) {$1$};
\node at (16.5,-2.598) {$1$};
\end{tikzpicture}
\begin{tikzpicture}[scale=0.5]
\def\sep{2};
\draw[green,thin](0.0,0.0)--(3.0,0.0);
\draw[green,thin](3,0)--(1.5,-2.598);
\draw[green,thin](0,0)--(1.5,-2.598);
\draw[green,thin](0.5,-0.866)--(2.5,-0.866);
\draw[green,thin](2,0)--(1.0,-1.7319999999999998);
\draw[green,thin](1,0)--(2.0,-1.7319999999999998);
\draw[green,thin](1.0,-1.732)--(2.0,-1.732);
\draw[green,thin](1,0)--(0.5,-0.8659999999999999);
\draw[green,thin](2,0)--(2.5,-0.8659999999999999);
\draw[green,thin](1.5,-2.598)--(1.5,-2.598);
\draw[green,thin](0,0)--(0.0,0.0);
\draw[green,thin](3,0)--(3.0,0.0);
\draw[green,thin](5.0,0.0)--(8.0,0.0);
\draw[green,thin](8,0)--(6.5,-2.598);
\draw[green,thin](5,0)--(6.5,-2.598);
\draw[green,thin](5.5,-0.866)--(7.5,-0.866);
\draw[green,thin](7,0)--(6.0,-1.7319999999999998);
\draw[green,thin](6,0)--(7.0,-1.7319999999999998);
\draw[green,thin](6.0,-1.732)--(7.0,-1.732);
\draw[green,thin](6,0)--(5.5,-0.8659999999999999);
\draw[green,thin](7,0)--(7.5,-0.8659999999999999);
\draw[green,thin](6.5,-2.598)--(6.5,-2.598);
\draw[green,thin](5,0)--(5.0,0.0);
\draw[green,thin](8,0)--(8.0,0.0);
\draw[green,thin](10.0,0.0)--(13.0,0.0);
\draw[green,thin](13,0)--(11.5,-2.598);
\draw[green,thin](10,0)--(11.5,-2.598);
\draw[green,thin](10.5,-0.866)--(12.5,-0.866);
\draw[green,thin](12,0)--(11.0,-1.7319999999999998);
\draw[green,thin](11,0)--(12.0,-1.7319999999999998);
\draw[green,thin](11.0,-1.732)--(12.0,-1.732);
\draw[green,thin](11,0)--(10.5,-0.8659999999999999);
\draw[green,thin](12,0)--(12.5,-0.8659999999999999);
\draw[green,thin](11.5,-2.598)--(11.5,-2.598);
\draw[green,thin](10,0)--(10.0,0.0);
\draw[green,thin](13,0)--(13.0,0.0);
\draw[green,thin](15.0,0.0)--(18.0,0.0);
\draw[green,thin](18,0)--(16.5,-2.598);
\draw[green,thin](15,0)--(16.5,-2.598);
\draw[green,thin](15.5,-0.866)--(17.5,-0.866);
\draw[green,thin](17,0)--(16.0,-1.7319999999999998);
\draw[green,thin](16,0)--(17.0,-1.7319999999999998);
\draw[green,thin](16.0,-1.732)--(17.0,-1.732);
\draw[green,thin](16,0)--(15.5,-0.8659999999999999);
\draw[green,thin](17,0)--(17.5,-0.8659999999999999);
\draw[green,thin](16.5,-2.598)--(16.5,-2.598);
\draw[green,thin](15,0)--(15.0,0.0);
\draw[green,thin](18,0)--(18.0,0.0);
\node at (0.0,-0.0) {$3$};
\node at (1.0,-0.0) {$4$};
\node at (2.0,-0.0) {$3$};
\node at (3.0,-0.0) {$4$};
\node at (0.5,-0.866) {$3$};
\node at (1.5,-0.866) {$4$};
\node at (2.5,-0.866) {$3$};
\node at (1.0,-1.732) {$3$};
\node at (2.0,-1.732) {$4$};
\node at (1.5,-2.598) {$3$};
\draw[ultra thick,green] (6.0,-1.732)--(6.5,-2.598);
\draw (6.0,-1.732) circle [radius=0.4];
\draw (6.5,-2.598) circle [radius=0.4];
\node at (5.0,-0.0) {$2$};
\node at (6.0,-0.0) {$3$};
\node at (7.0,-0.0) {$2$};
\node at (8.0,-0.0) {$3$};
\node at (5.5,-0.866) {$2$};
\node at (6.5,-0.866) {$4$};
\node at (7.5,-0.866) {$3$};
\node at (6.0,-1.732) {$2$};
\node at (7.0,-1.732) {$3$};
\node at (6.5,-2.598) {$4$};
\node at (10.0,-0.0) {$1$};
\node at (11.0,-0.0) {$2$};
\node at (12.0,-0.0) {$1$};
\node at (13.0,-0.0) {$2$};
\node at (10.5,-0.866) {$2$};
\node at (11.5,-0.866) {$1$};
\node at (12.5,-0.866) {$2$};
\node at (11.0,-1.732) {$4$};
\node at (12.0,-1.732) {$2$};
\node at (11.5,-2.598) {$2$};
\node at (15.0,-0.0) {$4$};
\node at (16.0,-0.0) {$4$};
\node at (17.0,-0.0) {$1$};
\node at (18.0,-0.0) {$1$};
\node at (15.5,-0.866) {$4$};
\node at (16.5,-0.866) {$1$};
\node at (17.5,-0.866) {$1$};
\node at (16.0,-1.732) {$1$};
\node at (17.0,-1.732) {$1$};
\node at (16.5,-2.598) {$1$};
\end{tikzpicture}
\begin{tikzpicture}[scale=0.5]
\def\sep{2};
\draw[green,thin](0.0,0.0)--(3.0,0.0);
\draw[green,thin](3,0)--(1.5,-2.598);
\draw[green,thin](0,0)--(1.5,-2.598);
\draw[green,thin](0.5,-0.866)--(2.5,-0.866);
\draw[green,thin](2,0)--(1.0,-1.7319999999999998);
\draw[green,thin](1,0)--(2.0,-1.7319999999999998);
\draw[green,thin](1.0,-1.732)--(2.0,-1.732);
\draw[green,thin](1,0)--(0.5,-0.8659999999999999);
\draw[green,thin](2,0)--(2.5,-0.8659999999999999);
\draw[green,thin](1.5,-2.598)--(1.5,-2.598);
\draw[green,thin](0,0)--(0.0,0.0);
\draw[green,thin](3,0)--(3.0,0.0);
\draw[green,thin](5.0,0.0)--(8.0,0.0);
\draw[green,thin](8,0)--(6.5,-2.598);
\draw[green,thin](5,0)--(6.5,-2.598);
\draw[green,thin](5.5,-0.866)--(7.5,-0.866);
\draw[green,thin](7,0)--(6.0,-1.7319999999999998);
\draw[green,thin](6,0)--(7.0,-1.7319999999999998);
\draw[green,thin](6.0,-1.732)--(7.0,-1.732);
\draw[green,thin](6,0)--(5.5,-0.8659999999999999);
\draw[green,thin](7,0)--(7.5,-0.8659999999999999);
\draw[green,thin](6.5,-2.598)--(6.5,-2.598);
\draw[green,thin](5,0)--(5.0,0.0);
\draw[green,thin](8,0)--(8.0,0.0);
\draw[green,thin](10.0,0.0)--(13.0,0.0);
\draw[green,thin](13,0)--(11.5,-2.598);
\draw[green,thin](10,0)--(11.5,-2.598);
\draw[green,thin](10.5,-0.866)--(12.5,-0.866);
\draw[green,thin](12,0)--(11.0,-1.7319999999999998);
\draw[green,thin](11,0)--(12.0,-1.7319999999999998);
\draw[green,thin](11.0,-1.732)--(12.0,-1.732);
\draw[green,thin](11,0)--(10.5,-0.8659999999999999);
\draw[green,thin](12,0)--(12.5,-0.8659999999999999);
\draw[green,thin](11.5,-2.598)--(11.5,-2.598);
\draw[green,thin](10,0)--(10.0,0.0);
\draw[green,thin](13,0)--(13.0,0.0);
\draw[green,thin](15.0,0.0)--(18.0,0.0);
\draw[green,thin](18,0)--(16.5,-2.598);
\draw[green,thin](15,0)--(16.5,-2.598);
\draw[green,thin](15.5,-0.866)--(17.5,-0.866);
\draw[green,thin](17,0)--(16.0,-1.7319999999999998);
\draw[green,thin](16,0)--(17.0,-1.7319999999999998);
\draw[green,thin](16.0,-1.732)--(17.0,-1.732);
\draw[green,thin](16,0)--(15.5,-0.8659999999999999);
\draw[green,thin](17,0)--(17.5,-0.8659999999999999);
\draw[green,thin](16.5,-2.598)--(16.5,-2.598);
\draw[green,thin](15,0)--(15.0,0.0);
\draw[green,thin](18,0)--(18.0,0.0);
\node at (0.0,-0.0) {$3$};
\node at (1.0,-0.0) {$4$};
\node at (2.0,-0.0) {$3$};
\node at (3.0,-0.0) {$4$};
\node at (0.5,-0.866) {$3$};
\node at (1.5,-0.866) {$4$};
\node at (2.5,-0.866) {$3$};
\node at (1.0,-1.732) {$3$};
\node at (2.0,-1.732) {$4$};
\node at (1.5,-2.598) {$3$};
\draw[ultra thick,green] (6.5,-0.866)--(7.0,-1.732);
\draw (6.5,-0.866) circle [radius=0.4];
\draw (7.0,-1.732) circle [radius=0.4];
\node at (5.0,-0.0) {$2$};
\node at (6.0,-0.0) {$3$};
\node at (7.0,-0.0) {$2$};
\node at (8.0,-0.0) {$3$};
\node at (5.5,-0.866) {$2$};
\node at (6.5,-0.866) {$4$};
\node at (7.5,-0.866) {$3$};
\node at (6.0,-1.732) {$3$};
\node at (7.0,-1.732) {$2$};
\node at (6.5,-2.598) {$4$};
\node at (10.0,-0.0) {$1$};
\node at (11.0,-0.0) {$2$};
\node at (12.0,-0.0) {$1$};
\node at (13.0,-0.0) {$2$};
\node at (10.5,-0.866) {$2$};
\node at (11.5,-0.866) {$1$};
\node at (12.5,-0.866) {$2$};
\node at (11.0,-1.732) {$4$};
\node at (12.0,-1.732) {$2$};
\node at (11.5,-2.598) {$2$};
\node at (15.0,-0.0) {$4$};
\node at (16.0,-0.0) {$4$};
\node at (17.0,-0.0) {$1$};
\node at (18.0,-0.0) {$1$};
\node at (15.5,-0.866) {$4$};
\node at (16.5,-0.866) {$1$};
\node at (17.5,-0.866) {$1$};
\node at (16.0,-1.732) {$1$};
\node at (17.0,-1.732) {$1$};
\node at (16.5,-2.598) {$1$};
\end{tikzpicture}
\caption{Interlacing triangular arrays with top row $\lambda^{(1)}=3434,\lambda^{(2)}=2323,\lambda^{(3)}=1212,\lambda^{(4)}=4411$.}
\label{fig:Gr24-interlacing-example-more}
\end{figure}

Now, \Cref{thm:k-theory}, \Cref{thm:dual-k} and \Cref{thm:ssm} imply that
\[G_{0101}^3=2G_{1010}-G_{1100},\quad (G_{0101}^*)^3=2G_{1010}^*-5G_{1100}^*,\quad S_{0101}^3=2S_{1010}-6S_{1100}.\]
\end{ex}

\subsection{Outline}

In \Cref{sec:prelim} we give background and definitions for interlacing triangular arrays and for partial flag varieties. In \Cref{sec:puzzle}, we prove \Cref{thm:psi-is-bijection}, establishing bijections
\[
\T_{3,n}\overset{\mathscr{T}'}{\underset{\mathscr{T}}\rightleftarrows} \mc{P}_n
\]
between rank-$3$ interlacing triangular arrays and $1/2/3$-puzzles. In \Cref{sec:puzzle-to-colorings} we in turn prove \Cref{thm:counting-m3}, giving bijections
\[
\mc{P}_n \overset{\mathscr{C}}{\underset{\mathscr{P}}\rightleftarrows} \mc{C}_n
\]
between $1/2/3$-puzzles and proper vertex colorings of $\Delta_n$. In \Cref{thm:counting-m4} we also give bijections
\[
\T_{4,n} \overset{\mathscr{D}}{\underset{\mathscr{D}'}\rightleftarrows} \mc{D}_n
\]
between rank-$4$ interlacing triangular arrays and certain edge labelings of $\square_n$. 

In \Cref{sec:geometry}, we show how to convert between $1/2/3$-puzzles and $0/1/10$-puzzles. Tracing this correspondence through to interlacing triangular arrays using the bijection $\mathscr{T}$, we show that forbidding certain of the $0/1/10$-puzzle pieces corresponds to forbidding certain patterns in the arrays. In \Cref{sec:splitting}, we prove the key \Cref{lem:splitting-interlacing} which allows us to split arrays into pairs of arrays of lower rank. Finally, these results are applied in \Cref{sec:geometry} to prove \Cref{thm:k-theory,thm:dual-k,thm:ssm,thm:2-step}.

An extended abstract of part of this work will appear in the proceedings of FPSAC 2025 \cite{fpsac-version}.
\section{Preliminaries}
\label{sec:prelim}
\subsection{Interlacing triangular arrays} 
\label{sec:prelim-interlacing}

We now define interlacing triangular arrays, the main objects of study.

\begin{defin}[Aggarwal--Borodin--Wheeler \cite{Aggarwal-Borodin-Wheeler}]\label{def:interlacing-triangle}
An \emph{interlacing triangular array} $T$ of rank $m$ and height $n$ is a collection $\{T^{(i)}_{j,k} \mid 1 \leq i \leq m, 1 \leq j \leq k \leq n \}$ of positive integers from $[1,m]$, subject to the following conditions:
\begin{itemize}
    \item[(a)] For each $k=1,\ldots,n$ we have an equality of multisets:
    \[\{T^{(i)}_{j,k} \mid 1 \leq i \leq m, 1 \leq j \leq k\} = \{1^k\} \cup \cdots \cup \{m^k\}.\]
    \item[(b)] Let the horizontal coordinate of $T^{(i)}_{j,k}$ be $h(i,j,k) \coloneqq in + j - (n+k)/2$. If $T^{(i)}_{j,k}=T^{(i')}_{j',k}=a$ for some $i,j,i',j',k$ with $h(i,j,k)<h(i',j',k)$, then there must exist $i'',j''$ with $T^{(i'')}_{j'',k-1}=a$ and $h(i,j,k)<h(i'',j'',k-1)<h(i',j',k)$. This entry $T^{(i'')}_{j'',k-1}$ is said to \emph{interlace} with $T^{(i)}_{j,k}$ and $T^{(i')}_{j',k}$.
\end{itemize}

For each $k=1,\ldots,n$ we can view $T^{(\bullet)}_{\bullet,k} \coloneqq \{T^{(i)}_{j,k} \mid 1 \leq i \leq m, 1 \leq j \leq k\}$ as the rows of an array of $m$ triangles, from bottom to top. We denote by $\T_{m,n}$ the set of interlacing triangular arrays of rank $m$ and height $n$ and by $\T_{m,n}(\lambda^{(1)},\ldots,\lambda^{(m)})$ the subset whose top row (that is, the row $k=n$) consists of $\lambda^{(1)},\ldots,\lambda^{(m)}$; here $\lambda^{(i)}\in[m]^n$ for $i\in[m]$. For $T\in\T_{m,n}$, we use $T^{(i)}$ to denote the $i$-th triangle from left to right, and use $T^{(i)}_{\bullet,k}$ to denote its $k$-th row. 
\end{defin}

\begin{ex}
We examine the interlacing triangular array of rank $m=3$ and height $n=4$ from \Cref{fig:interlacing-example}, reproduced below. 

\begin{center}
\begin{tikzpicture}[scale=1]
\def\sep{1.50000000000000};
\draw[green!80,thin](0.000000000000000,0.000000000000000)--(3.00000000000000,0.000000000000000);
\draw[green!80,thin](3.00000000000000,0)--(1.50000000000000,-2.59800000000000);
\draw[green!80,thin](0.000000000000000,0)--(1.50000000000000,-2.59800000000000);
\draw[green!80,thin](0.500000000000000,-0.866000000000000)--(2.50000000000000,-0.866000000000000);
\draw[green!80,thin](2.00000000000000,0)--(1.00000000000000,-1.73200000000000);
\draw[green!80,thin](1.00000000000000,0)--(2.00000000000000,-1.73200000000000);
\draw[green!80,thin](1.00000000000000,-1.73200000000000)--(2.00000000000000,-1.73200000000000);
\draw[green!80,thin](1.00000000000000,0)--(0.500000000000000,-0.866000000000000);
\draw[green!80,thin](2.00000000000000,0)--(2.50000000000000,-0.866000000000000);
\draw[green!80,thin](1.50000000000000,-2.59800000000000)--(1.50000000000000,-2.59800000000000);
\draw[green!80,thin](0.000000000000000,0)--(0.000000000000000,0.000000000000000);
\draw[green!80,thin](3.00000000000000,0)--(3.00000000000000,0.000000000000000);
\draw[green!80,thin](4.50000000000000,0.000000000000000)--(7.50000000000000,0.000000000000000);
\draw[green!80,thin](7.50000000000000,0)--(6.00000000000000,-2.59800000000000);
\draw[green!80,thin](4.50000000000000,0)--(6.00000000000000,-2.59800000000000);
\draw[green!80,thin](5.00000000000000,-0.866000000000000)--(7.00000000000000,-0.866000000000000);
\draw[green!80,thin](6.50000000000000,0)--(5.50000000000000,-1.73200000000000);
\draw[green!80,thin](5.50000000000000,0)--(6.50000000000000,-1.73200000000000);
\draw[green!80,thin](5.50000000000000,-1.73200000000000)--(6.50000000000000,-1.73200000000000);
\draw[green!80,thin](5.50000000000000,0)--(5.00000000000000,-0.866000000000000);
\draw[green!80,thin](6.50000000000000,0)--(7.00000000000000,-0.866000000000000);
\draw[green!80,thin](6.00000000000000,-2.59800000000000)--(6.00000000000000,-2.59800000000000);
\draw[green!80,thin](4.50000000000000,0)--(4.50000000000000,0.000000000000000);
\draw[green!80,thin](7.50000000000000,0)--(7.50000000000000,0.000000000000000);
\draw[green!80,thin](9.00000000000000,0.000000000000000)--(12.0000000000000,0.000000000000000);
\draw[green!80,thin](12.0000000000000,0)--(10.5000000000000,-2.59800000000000);
\draw[green!80,thin](9.00000000000000,0)--(10.5000000000000,-2.59800000000000);
\draw[green!80,thin](9.50000000000000,-0.866000000000000)--(11.5000000000000,-0.866000000000000);
\draw[green!80,thin](11.0000000000000,0)--(10.0000000000000,-1.73200000000000);
\draw[green!80,thin](10.0000000000000,0)--(11.0000000000000,-1.73200000000000);
\draw[green!80,thin](10.0000000000000,-1.73200000000000)--(11.0000000000000,-1.73200000000000);
\draw[green!80,thin](10.0000000000000,0)--(9.50000000000000,-0.866000000000000);
\draw[green!80,thin](11.0000000000000,0)--(11.5000000000000,-0.866000000000000);
\draw[green!80,thin](10.5000000000000,-2.59800000000000)--(10.5000000000000,-2.59800000000000);
\draw[green!80,thin](9.00000000000000,0)--(9.00000000000000,0.000000000000000);
\draw[green!80,thin](12.0000000000000,0)--(12.0000000000000,0.000000000000000);
\node at (0.000000000000000,-0.000000000000000) {$1$};
\node at (1.00000000000000,-0.000000000000000) {$2$};
\draw (1,0) circle [radius=0.3];
\node at (2.00000000000000,-0.000000000000000) {$1$};
\node at (3.00000000000000,-0.000000000000000) {$3$};
\node at (0.500000000000000,-0.866000000000000) {$1$};
\node at (1.50000000000000,-0.866000000000000) {$2$};
\draw (1.5,-.866) circle [radius=0.3];
\node at (2.50000000000000,-0.866000000000000) {$1$};
\node at (1.00000000000000,-1.73200000000000) {$1$};
\node at (2.00000000000000,-1.73200000000000) {$2$};
\node at (1.50000000000000,-2.59800000000000) {$1$};
\node at (4.50000000000000,-0.000000000000000) {$1$};
\node at (5.50000000000000,-0.000000000000000) {$3$};
\node at (6.50000000000000,-0.000000000000000) {$3$};
\node at (7.50000000000000,-0.000000000000000) {$2$};
\draw (7.5,0) circle [radius=0.3];
\node at (5.00000000000000,-0.866000000000000) {$3$};
\node at (6.00000000000000,-0.866000000000000) {$3$};
\node at (7.00000000000000,-0.866000000000000) {$1$};
\node at (5.50000000000000,-1.73200000000000) {$3$};
\node at (6.50000000000000,-1.73200000000000) {$1$};
\node at (6.00000000000000,-2.59800000000000) {$3$};
\node at (9.00000000000000,-0.000000000000000) {$1$};
\node at (10.0000000000000,-0.000000000000000) {$2$};
\draw (10,0) circle [radius=0.3];
\node at (11.0000000000000,-0.000000000000000) {$3$};
\node at (12.0000000000000,-0.000000000000000) {$2$};
\draw (12,0) circle [radius=0.3];
\node at (9.50000000000000,-0.866000000000000) {$2$};
\draw (9.5,-.866) circle [radius=0.3];
\node at (10.5000000000000,-0.866000000000000) {$3$};
\node at (11.5000000000000,-0.866000000000000) {$2$};
\draw (11.5,-.866) circle [radius=0.3];
\node at (10.0000000000000,-1.73200000000000) {$3$};
\node at (11.0000000000000,-1.73200000000000) {$2$};
\node at (10.5000000000000,-2.59800000000000) {$2$};
\end{tikzpicture}
\end{center}

The three triangles in the array have top rows $T^{(1)}_{\bullet,4}=(1,2,1,3)$, $T^{(2)}_{\bullet,4}=(1,3,3,2)$, and $T^{(3)}_{\bullet,4}=(1,2,3,2)$. \Cref{def:interlacing-triangle}(a) requires that each integer $1,2,$ and $3$ appears exactly $4$ times total among these. Similarly, each of $1,2,$ and $3$ appears exactly $k$ times in each of the lower rows $k=1,2,$ and $3$.

The entries equal to $2$ in the top two rows are circled. Those in the top row have horizontal coordinates $2,8,10,$ and $12$, from left to right. Those in the second highest row have horizontal coordinates $\frac{5}{2},\frac{19}{2},$ and $\frac{23}{2}$. The interlacing condition of \Cref{def:interlacing-triangle}(b) asserts that these entries alternate between rows as the horizontal coordinate increases. The same condition holds for other entry values in other pairs of consecutive rows.
\end{ex}

\begin{lemma}\label{lem:interlacing-side-triangles}
Let $T$ be an interlacing triangular array of rank $m$ and height $n$. Then for a fixed $p\in[n]$, $T^{(1)}_{p,k}$ is constant for $p\leq k\leq n$ and $T^{(m)}_{k-p+1,k}$ is constant for $p\leq k\leq n$.
\end{lemma}
\begin{proof}
Let $a=T^{(1)}_{1,n-1}$. By the interlacing condition on $a$ for row $n-1$ and row $n$, $a$ must appear to the left of $h(1,1,n-1)$, meaning that $T^{(1)}_{1,n}=a$. With the same argument continuing to the right, we can show that $T^{(1)}_{j,n-1}=T^{(1)}_{j,n}$ for all $1\leq j\leq n-1$. Restricting to the bottom $p$ rows for $p=1,\ldots,n$, with the same argument, we obtain that $T^{(1)}_{p,k}$ is constant for $p\leq k\leq n$. The statement for $T^{(m)}$ is the same, arguing from right to left.
\end{proof}

\subsection{The partial flag variety}
\label{sec:prelim-flags}
A \emph{partial flag} $F_{\bullet}$ of dimension $\d=(0=d_0 \leq d_1\leq \cdots \leq d_{m}=n)$ is a chain of linear subspaces $0=F_0 \subset F_1\subset \cdots\subset F_{m-1}\subset F_m=\C^n$ such that $\dim F_i=d_i$ for $i=1,\ldots,n$. The collection of such flags form the \emph{partial flag variety} $\Fl(\d;n)$, which admits a \emph{Bruhat decomposition} into \emph{open Schubert cells} $\bigsqcup_{w\in S_n^{\d}}\Omega_w$, with the index set \[S_n^{\d}=\{w\in S_n\:|\:\mathrm{Des}(w)\subset\{d_1,\ldots,d_{m-1}\}\}.\] The closure of each open Schubert cell is the \emph{Schubert variety} $X_w=\overline{\Omega_w}$. Write $\sigma_w:=[X_w]\in H^{2\ell(w)}(\Fl(\d;n),\Z)$ for the \emph{Schubert class} of $w$, the Poincar\'e dual to the fundamental class of the Schubert variety $X_w$; the Schubert classes $\sigma_w$ for $w \in S^{\d}_n$ form a basis of $H^\ast(\Fl(\d;n),\Z)$. Let $w_0^{\d}$ be the unique element in $S_n^{\d}$ of maximum length, whose Schubert class $\sigma_{w_0^{\d}}$ is the class of a point. Note that \[\ell(w_0^{\d})={n\choose 2}-\sum_{i=1}^m{d_i-d_{i-1}\choose 2}.\]

Let $S_{\d}$ be the \emph{parabolic subgroup} of $S_n$ generated by $\{s_j\:|\: j\notin\d\}$ and let $w_0(\d)$ be the longest element of $S_{\d}$. Thus, $S_n^{\d}$ is the set of minimum-length coset representatives for $S_n/S_{\d}$. We have $w_0^{\d}w_0(\d)=w_0$, where $w_0=n\ n{-}1\cdots 1$ is the longest permutation in $S_n$. 

\begin{remark}
Readers may be more familiar with the the convention that $\bs{d}$ is \emph{strictly} increasing. In our setting, if $d_i=d_{i+1}$ for some $i$, then we necessarily have $F_i=F_{i+1}$ for $F_{\bullet}\in \Fl(\d;n)$ and thus the geometry remains the same. Our convention that $\d$ is only weakly increasing will be convenient in the sequel. 
\end{remark}

For a sequence of permutations $w^{(1)},\ldots,w^{(k)}\in S_n^{\d}$, the \emph{Schubert structure constants} $c^w_{w^{(1)},\ldots,w^{(k)}}$ are defined by
\[
\sigma_{w^{(1)}}\cdots\sigma_{w^{(k)}}=\sum_{w\in S_n}c^w_{w^{(1)},\ldots,w^{(k)}}\sigma_w.
\]
The constant $c^w_{w^{(1)},\ldots,w^{(k)}}$ is zero unless $\ell(w)=\ell(w^{(1)})+\cdots+\ell(w^{(k)})$ and $w\in S_n^{\d}$. For a permutation $w\in S_n^{\d}$, its \emph{dual} is $w^{\vee_{\d}}:=w_0ww_0(\d)\in S_n^{\d}$. By the duality theorem in $H^*(\Fl(\d;n))$, for $\ell(u)+\ell(v)=\ell(w_0^{\d})$ we have $\sigma_u\sigma_v=\delta_{u, v^{\vee_{\d}}}\sigma_{w_0^{\d}}$, where $\delta$ is the Kronecker delta function. 

\begin{defin}\label{def:Schubert-string}
Let $\Sigma$ be a finite alphabet with a total order $\{q_1<\cdots <q_m\}$. An element $\lambda\in\Sigma^{n}$ is called a \emph{Schubert string} of size $n$. We say that $\lambda$ has \emph{type} $q^{\alpha}=q_1^{\alpha_1}\cdots q_m^{\alpha_k}$ if $\lambda$ contains $\alpha_i$ copies of $q_i$, for $i=1,\ldots,m$. 
\end{defin}
In this paper, we will deal alternately with Schubert strings and with their corresponding permutations that lie in certain parabolic quotients of $S_n$.
\begin{defin}\label{def:Schubert-string-permutation}
For a Schubert string $\lambda$ of size $n$ and type $q^{\alpha}$, we associate a permutation $w(\lambda)_{\Sigma}=w\in S_n$ such that $w(d_{i-1}+1)<\cdots<w(d_{i})$ are the positions of $q_i$'s in $\lambda$, where $\alpha=\alpha(\bs{d})$ is defined by $\alpha_1+\cdots+\alpha_{i}=d_i$ for $i=1,\ldots,m$. This map is a bijection from to set of Schubert strings of type $q^{\alpha}$ to $S_n^{\d}$. For $w\in S_n^{\d}$, write $\lambda(w)_{\Sigma}$ for the corresponding Schubert string. When the alphabet $\Sigma$ and its total order are understood, the subscripts might be omitted.
\end{defin}
In the case where the alphabet $\Sigma=\{a,b\}$ has only two letters, two permutations $u,v\in S_n^{\{k\}}$ are dual to each other if $\lambda(u)_{a<b}$ can be obtained from $\lambda(v)_{b<a}$ by first reversing the order and then swapping the letters $a$ and $b$.

\section{From interlacing triangular arrays to puzzles}
\label{sec:puzzle}

The goal of this section is to establish a bijection $\mathscr{T}$ between interlacing triangular arrays of rank $3$ and certain edge labelings of $\Delta_n$. We call these labelings \textit{$1/2/3$-puzzles} since they in turn are in bijection (see \Cref{sec:geometry}) with the $0/1/10$-puzzles of Knutson--Tao \cite{Knutson-Tao-equivariant} as generalized by Knutson--Zinn-Justin \cite[\S 4]{Knutson-Zinn-Justin-2}.

\subsection{$1/2/3$-Puzzles} We now define $1/2/3$-puzzles.
\begin{defin}
\label{def:123-puzzles}
We denote by $\Delta_n$ the \emph{triangular grid graph} with side length $n$; see \Cref{fig:123-puzzle-example}. We view $\Delta_n$ as embedded in the plane as pictured, allowing us to distinguish between the $\Delta$-oriented and $\nabla$-oriented faces. We take the lower left corner as a distinguished base point, and view $\Delta_{n-1}$ as a subgraph of $\Delta_{n}$, sharing the base point.  

A \emph{$1/2/3$-puzzle} is a labeling of the edges of $\Delta_n$ with labels $1$, $2$, and $3$ so that each face has distinct edge labels. We write $\mc{P}_n$ for the set of these puzzles.

The \emph{boundary conditions} $\boldsymbol{\mu}=(\mu^{(1)},\mu^{(2)}, \mu^{(3)})$ of a puzzle $P \in \mc{P}_n$ are the labelings of the three sides of the triangle $\Delta_n$, read clockwise starting from the base point. We write $\mc{P}_n(\boldsymbol{\mu})$ for the set of puzzles from $\mc{P}_k$ with boundary conditions $\boldsymbol{\mu}$.
\end{defin}

\begin{figure}
    \begin{tikzpicture}[scale=0.95]
\draw[thin,blue!80](0,0)--(4,0)--(2.00000000000000,3.46400000000000)--(0,0);
\draw[thin,blue!80](0.500000000000000,0.866000000000000)--(3.50000000000000,0.866000000000000);
\draw[thin,blue!80](1,0)--(2.50000000000000,2.59800000000000);
\draw[thin,blue!80](0.500000000000000,0.866000000000000)--(1,0);
\draw[thin,blue!80](1.00000000000000,1.73200000000000)--(3.00000000000000,1.73200000000000);
\draw[thin,blue!80](2,0)--(3.00000000000000,1.73200000000000);
\draw[thin,blue!80](1.00000000000000,1.73200000000000)--(2,0);
\draw[thin,blue!80](1.50000000000000,2.59800000000000)--(2.50000000000000,2.59800000000000);
\draw[thin,blue!80](3,0)--(3.50000000000000,0.866000000000000);
\draw[thin,blue!80](1.50000000000000,2.59800000000000)--(3,0);

\node at (0.250000000000000,0.433000000000000) {$1$};
\node at (2.25000000000000,3.03100000000000) {$1$};
\node at (0.500000000000000,0.000000000000000) {$2$};
\node at (0.750000000000000,1.29900000000000) {$2$};
\node at (2.75000000000000,2.16500000000000) {$3$};
\node at (1.50000000000000,0.000000000000000) {$3$};
\node at (1.25000000000000,2.16500000000000) {$1$};
\node at (3.25000000000000,1.29900000000000) {$3$};
\node at (2.50000000000000,0.000000000000000) {$2$};
\node at (1.75000000000000,3.03100000000000) {$3$};
\node at (3.75000000000000,0.433000000000000) {$2$};
\node at (3.50000000000000,0.000000000000000) {$1$};
\node at (1.25000000000000,0.433000000000000) {$2$};
\node at (1.75000000000000,2.16500000000000) {$3$};
\node at (1.00000000000000,0.866000000000000) {$1$};
\node at (1.75000000000000,1.29900000000000) {$1$};
\node at (2.25000000000000,1.29900000000000) {$3$};
\node at (2.00000000000000,0.866000000000000) {$2$};
\node at (2.25000000000000,2.16500000000000) {$1$};
\node at (2.75000000000000,0.433000000000000) {$1$};
\node at (3.00000000000000,0.866000000000000) {$2$};
\node at (2.25000000000000,0.433000000000000) {$3$};
\node at (1.25000000000000,1.29900000000000) {$3$};
\node at (1.50000000000000,1.73200000000000) {$2$};
\node at (2.75000000000000,1.29900000000000) {$1$};
\node at (1.75000000000000,0.433000000000000) {$1$};
\node at (2.50000000000000,1.73200000000000) {$2$};
\node at (3.25000000000000,0.433000000000000) {$3$};
\node at (0.750000000000000,0.433000000000000) {$3$};
\node at (2.00000000000000,2.59800000000000) {$2$};
\end{tikzpicture}
\caption{A $1/2/3$-puzzle $P$ with boundary conditions $(1213, 1332, 1232)$.}
\label{fig:123-puzzle-example}
\end{figure}

\subsection{The bijection $\mathscr{T}$} We can now construct the bijection $\mathscr{T}$.
\begin{defin}
\label{def:puzzle-to-array-map}
Given a $1/2/3$-puzzle $P \in \mc{P}_n$, we produce a collection 
\[
\mathscr{T}(P) = \{T^{(i)}_{j,k} \mid 1 \leq i \leq 3, 1 \leq j \leq k \leq n \}
\]
of integers from $\{1,2,3\}$ as follows. For each $k=1,\ldots,n$, consider the copy of $\Delta_k$ inside $\Delta_n$ justified into the lower left corner of $\Delta_n$. The $k$-th row
\[
T^{(1)}_{1,k}, \ldots, T^{(1)}_{k,k}, T^{(2)}_{1,k}, \ldots, T^{(2)}_{k,k}, T^{(3)}_{1,k}, \ldots, T^{(3)}_{k,k},
\]
of $\mathscr{T}(P)$ is obtained by reading the labels of $P$ clockwise around $\Delta_k$, starting from the lower left vertex.
\end{defin}

\begin{ex}
The $1/2/3$-puzzle $P$ from \Cref{fig:123-puzzle-example} is sent by $\mathscr{T}$ to the array $T$ from \Cref{fig:interlacing-example}.
\end{ex}

\begin{prop}
For any $P \in \mc{P}_n$, the array $\mathscr{T}(P)$ is an interlacing triangular array of rank $3$ and height $n$. 
\end{prop}
\begin{proof}
Let $T=\mathscr{T}(P)$ and let $a \in \{1,2,3\}$. Each face of $\Delta_n$ has one edge labeled $a$ by $P$, so we may compute the multiplicity of $a$ as an edge label on the boundary of $\Delta_k$ as the number of $\Delta$-oriented faces of $\Delta_k$ minus the number of $\nabla$-oriented faces, since this causes the contribution of all internal edges to cancel. This number is always $k$, so each $a \in \{1,2,3\}$ occurs $k$ times in the $k$-th row of $T$. Thus $T$ satisfies \Cref{def:interlacing-triangle}(a).

We now show that $T$ satisfies \Cref{def:interlacing-triangle}(b), the interlacing condition. Suppose first that there is some failure of interlacing within $T^{(2)}$. Since $\mathscr{T}$ is equivariant with respect to permutations of the edge labels and array entries, we may assume without loss of generality that there are two consecutive $1$'s in the $k$-th row of $T^{(2)}$ with no interlacing $1$ in the $(k-1)$-st row. We consider the corresponding $P$-labeled subgraph of $\Delta_k$, shown below.
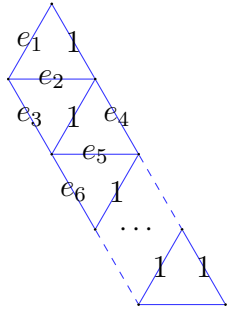
\begin{center}
\begin{tikzpicture}[scale=1, rotate=-60]
    \node[circle, fill=black, scale=0.1] (c0) at ({-2*sqrt(3)/3},2) {};
    \node[circle, fill=black, scale=0.1] (c1) at (0,2) {};
    \node[circle, fill=black, scale=0.1] (c2) at ({2*sqrt(3)/3},2) {};
    \node[circle, fill=black, scale=0.1] (c3) at ({4*sqrt(3)/3},2) {};
    \node[circle, fill=black, scale=0.1] (d0) at ({-3*sqrt(3)/3},3) {};
    \node[circle, fill=black, scale=0.1] (d1) at ({-1*sqrt(3)/3},3) {};
    \node[circle, fill=black, scale=0.1] (d2) at ({1*sqrt(3)/3},3) {};
    \node[circle, fill=black, scale=0.1] (d3) at ({3*sqrt(3)/3},3) {};
    \node (dots) at ({2.5*sqrt(3)/3},2.5) {$\cdots$};
    \node[circle, fill=black, scale=0.1] (d4) at ({5*sqrt(3)/3},3) {};
    \draw[blue!80,thin] (c0) -- (c1) node [midway] {\large\textcolor{black}{$e_3$}};
    \draw[blue!80,thin] (c1) -- (c2) node [midway] {\large\textcolor{black}{$e_6$}};
    \draw[blue!80,thin] (c0) -- (d0) node [midway] {\large\textcolor{black}{$e_1$}};
    \draw[blue!80,thin] (c0) -- (d1) node [midway] {\large\textcolor{black}{$e_2$}};
    \draw[blue!80,thin] (c1) -- (d1) node [midway] {\large\textcolor{black}{$1$}};
    \draw[blue!80,thin] (c1) -- (d2) node [midway] {\large\textcolor{black}{$e_5$}};
    \draw[blue!80,thin] (c2) -- (d2) node [midway] {\large\textcolor{black}{$1$}};
    \draw[blue!80,thin] (d0) -- (d1) node [midway] {\large\textcolor{black}{$1$}};
    \draw[blue!80,thin] (d1) -- (d2) node [midway] {\large\textcolor{black}{$e_4$}};
    \draw[dashed,blue!80] (d2) -- (d3) node [midway] {};
    \draw[dashed,blue!80] (c2) -- (c3) node [midway] {};
    \draw[blue!80,thin] (c3) -- (d4) node [midway] {};
    \draw[blue!80,thin] (c3) -- (d3) node [midway] {\large\textcolor{black}{$1$}};
    \draw[blue!80,thin] (d3) -- (d4) node [midway] {\large\textcolor{black}{$1$}};
\end{tikzpicture}
\end{center}
Since $P \in \mc{P}_n$, we must have $\{e_1,e_2\}=\{2,3\}$. The two indicated $1$'s along $\Delta_k \setminus \Delta_{k-1}$ are assumed consecutive, so $e_4 \neq 1$. By assumption $e_3,e_6 \neq 1$, so the remaining edge of the triangle containing $e_2,e_3$ must be labeled $1$. Thus $e_5 \neq 1$. This all forces each successive southwest-northeast diagonal edge to be labeled $1$, eventually producing an invalidly labeled triangle, a contradiction.

The only other possible failure of interlacing, up to symmetries, is if the last entry in row $k$ of $T^{(1)}$ is a $1$ with the next $1$ in $T^{(2)}$, and with no interlacing $1$ in the $(k-1)$-st row. An almost identical analysis again shows this is impossible, as diagrammed below. 

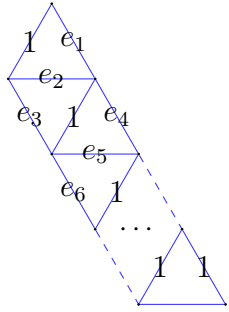
\begin{center}
\begin{tikzpicture}[scale=1,rotate=-60]
    \node[circle, fill=black, scale=0.1] (c0) at ({-2*sqrt(3)/3},2) {};
    \node[circle, fill=black, scale=0.1] (c1) at (0,2) {};
    \node[circle, fill=black, scale=0.1] (c2) at ({2*sqrt(3)/3},2) {};
    \node[circle, fill=black, scale=0.1] (c3) at ({4*sqrt(3)/3},2) {};
    \node[circle, fill=black, scale=0.1] (d0) at ({-3*sqrt(3)/3},3) {};
    \node[circle, fill=black, scale=0.1] (d1) at ({-1*sqrt(3)/3},3) {};
    \node[circle, fill=black, scale=0.1] (d2) at ({1*sqrt(3)/3},3) {};
    \node[circle, fill=black, scale=0.1] (d3) at ({3*sqrt(3)/3},3) {};
    \node (dots) at ({2.5*sqrt(3)/3},2.5) {$\cdots$};
    \node[circle, fill=black, scale=0.1] (d4) at ({5*sqrt(3)/3},3) {};
    \draw[blue!80,thin] (c0) -- (c1) node [midway] {\large\textcolor{black}{$e_3$}};
    \draw[blue!80,thin] (c1) -- (c2) node [midway] {\large\textcolor{black}{$e_6$}};
    \draw[blue!80,thin] (c0) -- (d0) node [midway] {\large\textcolor{black}{$1$}};
    \draw[blue!80,thin] (c0) -- (d1) node [midway] {\large\textcolor{black}{$e_2$}};
    \draw[blue!80,thin] (c1) -- (d1) node [midway] {\large\textcolor{black}{$1$}};
    \draw[blue!80,thin] (c1) -- (d2) node [midway] {\large\textcolor{black}{$e_5$}};
    \draw[blue!80,thin] (c2) -- (d2) node [midway] {\large\textcolor{black}{$1$}};
    \draw[blue!80,thin] (d0) -- (d1) node [midway] {\large\textcolor{black}{$e_1$}};
    \draw[blue!80,thin] (d1) -- (d2) node [midway] {\large\textcolor{black}{$e_4$}};
    \draw[dashed,blue!80] (d2) -- (d3) node [midway] {};
    \draw[dashed,blue!80] (c2) -- (c3) node [midway] {};
    \draw[blue!80,thin] (c3) -- (d4) node [midway] {};
    \draw[blue!80,thin] (c3) -- (d3) node [midway] {\large\textcolor{black}{$1$}};
    \draw[blue!80,thin] (d3) -- (d4) node [midway] {\large\textcolor{black}{$1$}};
\end{tikzpicture}
\end{center}
\end{proof}

\begin{theorem}
\label{thm:psi-is-bijection}
The map $\mathscr{T}$ is a bijection $\mc{P}_n \to \mc{T}_{3,n}$ restricting, for each boundary condition $\boldsymbol{\lambda}$, to a bijection $\mc{P}_n(\boldsymbol{\lambda} )\to \mc{T}_{3,n}(\boldsymbol{\lambda})$.
\end{theorem}

Before proving \Cref{thm:psi-is-bijection}, we prove a lemma.

\begin{defin}
\label{def:interlacing-one-row}
Tuples $\boldsymbol{a}=(a_0,a_1,\ldots,a_{k+1})\in\{1,2,3\}^{k+2}$ and $\boldsymbol{b}=(b_1,\ldots,b_{k-1})\in\{1,2,3\}^{k-1}$ are \emph{interlacing} if for each $i\in\{1,2,3\}$, the indices that $a_{j_0}=\cdots=a_{j_m}=i$ and $b_{k_1}=\cdots=b_{k_m}=i$ satisfy $j_0\leq k_1<j_1\leq\cdots\leq k_m<j_{m}$.
\end{defin}

\begin{lemma}
\label{lem:one-row}
Let $\boldsymbol{a}=(a_0,\ldots,a_{k+1})\in\{1,2,3\}^{k+2}$ and $\boldsymbol{b}=(b_1,\ldots,b_{k-1})\in\{1,2,3\}^{k-1}$ be interlacing. Then there exists a unique proper edge labeling of the one-row triangular array whose boundary is labeled by $\boldsymbol{a}$ and $\boldsymbol{b}$ as shown below. 
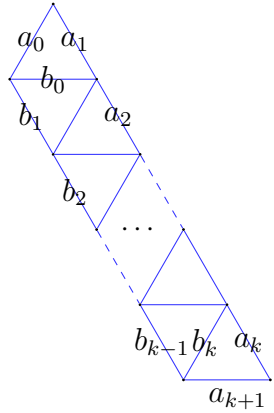
\begin{center}
\begin{tikzpicture}[scale=1,rotate=-60]
    \node[circle, fill=black, scale=0.1] (c0) at ({-2*sqrt(3)/3},2) {};
    \node[circle, fill=black, scale=0.1] (c1) at (0,2) {};
    \node[circle, fill=black, scale=0.1] (c2) at ({2*sqrt(3)/3},2) {};
    \node[circle, fill=black, scale=0.1] (c3) at ({4*sqrt(3)/3},2) {};
    \node[circle, fill=black, scale=0.1] (c4) at ({6*sqrt(3)/3},2) {};
    \node[circle, fill=black, scale=0.1] (d0) at ({-3*sqrt(3)/3},3) {};
    \node[circle, fill=black, scale=0.1] (d1) at ({-1*sqrt(3)/3},3) {};
    \node[circle, fill=black, scale=0.1] (d2) at ({1*sqrt(3)/3},3) {};
    \node[circle, fill=black, scale=0.1] (d3) at ({3*sqrt(3)/3},3) {};
    \node (dots) at ({2.5*sqrt(3)/3},2.5) {$\cdots$};
    \node[circle, fill=black, scale=0.1] (d4) at ({5*sqrt(3)/3},3) {};
    \node[circle, fill=black, scale=0.1] (d5) at ({7*sqrt(3)/3},3) {};
    \draw[blue!80,thin] (c0) -- (c1) node [midway] {\large\textcolor{black}{$b_1$}};
    \draw[blue!80,thin] (c1) -- (c2) node [midway] {\large\textcolor{black}{$b_2$}};
    \draw[blue!80,thin] (c0) -- (d0) node [midway] {\large\textcolor{black}{$a_0$}};
    \draw[blue!80,thin] (c0) -- (d1) node [midway] {\large\textcolor{black}{$b_{0}$}};
    \draw[blue!80,thin] (c1) -- (d1) node [midway] {};
    \draw[blue!80,thin] (c1) -- (d2) node [midway] {};
    \draw[blue!80,thin] (c2) -- (d2) node [midway] {};
    \draw[blue!80,thin] (d0) -- (d1) node [midway] {\large\textcolor{black}{$a_1$}};
    \draw[blue!80,thin] (d1) -- (d2) node [midway] {\large\textcolor{black}{$a_2$}};
    \draw[dashed,blue!80] (d2) -- (d3) node [midway] {};
    \draw[dashed,blue!80] (c2) -- (c3) node [midway] {};
    \draw[blue!80,thin] (c3) -- (d4) node [midway] {};
    \draw[blue!80,thin] (c3) -- (d3) node [midway] {};
    \draw[blue!80,thin] (d3) -- (d4) node [midway] {};
    \draw[blue!80,thin] (c4) -- (d5) node at ({7*sqrt(3)/3},2.5) {\large\textcolor{black}{$a_{k{+}1}$}};
    \draw[blue!80,thin] (d4) -- (d5) node [midway] {\large\textcolor{black}{$a_k$}};
    \draw[blue!80,thin] (c3) -- (c4) node [midway] {\large\textcolor{black}{$b_{k{-}1}$}};
    \draw[blue!80,thin] (c4) -- (d4) node [midway] {\large\textcolor{black}{$b_{k}$}};
\end{tikzpicture}
\end{center}
\end{lemma}
\begin{proof}
It is clear that there is at most one such labeling, since the labels of the internal edges are forced one-by-one, moving left to right. We use induction on $k$ to prove that this process in fact always produces a valid edge labeling. 

When $N=1$, the interlacing condition of $\boldsymbol{a}=(a_0,a_1,a_2)$ and $\boldsymbol{b}=()$ says that $\{a_0,a_1,a_2\}=\{1,2,3\}$ and thus we obtain a proper labeling of a single triangular face. 

Now assume $k\geq2$. The interlacing condition implies $a_0\neq a_1$ and $a_{k}\neq a_{k+1}$. Let $b_0$ be the unique label distinct from $a_0$ and $a_1$, and $b_{k}$ be the unique label distinct from $a_k$ and $a_{k+1}$. We now show that $\boldsymbol{b'}=(b_0,b_1,\ldots,b_k)\in\{1,2,3\}^{k+1}$ and $\boldsymbol{a'}=(a_2,\ldots,a_{k-1})\in\{1,2,3\}^{k-2}$ are interlacing. 
Take $i\in\{1,2,3\}$. To see that $i$ appears in $\boldsymbol{b'}$ earlier than in $\boldsymbol{a'}$, we have two cases: $b_0=i$ and $b_0\neq i$, where the first case is evident. When $b_0\neq i$, we have $i\in\{a_0,a_1\}$, so the interlacing condition on $\boldsymbol{a}$ and $\boldsymbol{b}$ implies that $i$ appears in $\boldsymbol{b'}$ first. Similarly, the last occurrence of $i$ in $\boldsymbol{b'}$ happens after the last occurrence of $i$ in $\boldsymbol{a'}$. Moreover, the interlacing condition on $\boldsymbol{a}$ and $\boldsymbol{b}$ states that between two consecutive occurrences of $i$ in $\boldsymbol{a'}$, there is an occurrence of $i$ in $\boldsymbol{b'}$, meaning that, so far, the number of $i$'s in $\boldsymbol{b'}$ is at least $1$ more than the number of $i$'s in $\boldsymbol{a'}$. Together with the fact that $|\boldsymbol{b'}|-|\boldsymbol{a'}|=3$, we conclude that $\boldsymbol{b'}$ and $\boldsymbol{a'}$ are interlacing.

By induction, the middle edges can be labeled properly, extending to a proper labeling of the whole graph.
\end{proof}

\begin{proof}[Proof of \Cref{thm:psi-is-bijection}]
We define a map $\mathscr{T}':\mc{T}_{3,n} \to \mc{P}_n$ inductively; see \Cref{ex:T-prime-example}. Let $T\in\mc{T}_{3,n}$. When $n=1$, $T$ consists of three distinct numbers $T^{(1)}_{1,1},T^{(2)}_{1,1},T^{(3)}_{1,1}$ forming a permutation of $\{1,2,3\}$. We may properly label the edges of $\Delta_1$ by these numbers, in clockwise order starting from the lower left vertex. When $n\geq2$, assume that we have defined a proper edge labeling of $\Delta_{n-1} \subset \Delta_n$. Let
\[
\bs{a}=\left(T^{(1)}_{n,n},T^{(2)}_{1,n},T^{(2)}_{2,n},\ldots,T^{(2)}_{n,n},T^{(3)}_{1,n} \right)\in\{1,2,3\}^{n+2}
\]
consist of the entry in the top right corner of $T^{(1)}$, the top row of $T^{(2)}$, and the top left entry of $T^{(3)}$, and let $\bs{b}=(T^{(2)}_{1,n-1},\ldots,T^{(2)}_{n-1,n-1})\in\{1,2,3\}^{n-1}$ consist of the $(n-1)$-st row of $T^{(2)}$. By definition of $\mc{T}_{3,n}$, $\bs{a}$ and $\bs{b}$ are interlacing in the sense of \Cref{def:interlacing-one-row}. By \Cref{lem:one-row}, the edges of $\Delta_n \setminus \Delta_{n-1}$ can be properly labeled in a unique way, subject to the boundary conditions imposed by $\bs{a}$ and $\bs{b}$. We define $\mathscr{T}'(T)$ to be this unique labeling. By construction, $\mathscr{T}'$ and $\mathscr{T}$ are inverse to each other, so $\mathscr{T}$ is a bijection. Finally, $\mathscr{T}$ preserves the boundary condition $\bs{\lambda}$ by construction. 
\end{proof}

\begin{ex}
\label{ex:T-prime-example}
Let $T \in \mc{T}_{3,4}$ be the interlacing triangular array from \Cref{fig:interlacing-example}. We show how to compute the map $\mathscr{T}'$ from the proof of \Cref{thm:psi-is-bijection} applied to the top row of $T$. We have $\bs{a}\coloneqq\left(T^{(1)}_{4,4},T^{(2)}_{1,4},T^{(2)}_{2,4},T^{(2)}_{3,4},T^{(2)}_{4,4},T^{(3)}_{1,4} \right)=(3,1,3,3,2,1)$ and $\bs{b}=(3,3,1)$. This yields the boundary conditions of $\Delta_4 \setminus \Delta_3$ shown below in bold; the unbolded numbers are the unique proper labeling of the interior edges.
\begin{center}
\begin{tikzpicture}[scale=1,rotate=-60]
    \node[circle, fill=black, scale=0.1] (c0) at ({-2*sqrt(3)/3},2) {};
    \node[circle, fill=black, scale=0.1] (c1) at (0,2) {};
    \node[circle, fill=black, scale=0.1] (c2) at ({2*sqrt(3)/3},2) {};
    \node[circle, fill=black, scale=0.1] (c3) at ({4*sqrt(3)/3},2) {};
    \node[circle, fill=black, scale=0.1] (d0) at ({-3*sqrt(3)/3},3) {};
    \node[circle, fill=black, scale=0.1] (d1) at ({-1*sqrt(3)/3},3) {};
    \node[circle, fill=black, scale=0.1] (d2) at ({1*sqrt(3)/3},3) {};
    \node[circle, fill=black, scale=0.1] (d3) at ({3*sqrt(3)/3},3) {};
    \node[circle, fill=black, scale=0.1] (d4) at ({5*sqrt(3)/3},3) {};
    \draw[blue!80,thin] (c0) -- (c1) node [midway] {\large\textcolor{black}{$\mathbf{3}$}};
    \draw[blue!80,thin] (c1) -- (c2) node [midway] {\large\textcolor{black}{$\mathbf{3}$}};
    \draw[blue!80,thin] (c0) -- (d0) node [midway] {\large\textcolor{black}{$\mathbf{3}$}};
    \draw[blue!80,thin] (c0) -- (d1) node [midway] {\large\textcolor{black}{$2$}};
    \draw[blue!80,thin] (c1) -- (d1) node [midway] {\large\textcolor{black}{$1$}};
    \draw[blue!80,thin] (c1) -- (d2) node [midway] {\large\textcolor{black}{$2$}};
    \draw[blue!80,thin] (c2) -- (d2) node [midway] {\large\textcolor{black}{$1$}};
    \draw[blue!80,thin] (d0) -- (d1) node [midway] {\large\textcolor{black}{$\mathbf{1}$}};
    \draw[blue!80,thin] (d1) -- (d2) node [midway] {\large\textcolor{black}{$\mathbf{3}$}};
    \draw[blue!80,thin] (d2) -- (d3) node [midway] {\large\textcolor{black}{$\mathbf{3}$}};
    \draw[blue!80,thin] (c2) -- (c3) node [midway] {\large\textcolor{black}{$\mathbf{1}$}};
    \draw[blue!80,thin] (c2) -- (d3) node [midway] {\large\textcolor{black}{$2$}};
    \draw[blue!80,thin] (c3) -- (d4) node [midway] {\large\textcolor{black}{$\mathbf{1}$}};
    \draw[blue!80,thin] (c3) -- (d3) node [midway] {\large\textcolor{black}{$3$}};
    \draw[blue!80,thin] (d3) -- (d4) node [midway] {\large\textcolor{black}{$\mathbf{2}$}};
\end{tikzpicture}
\end{center}
This agrees with the rightmost northwest-southeast layer of triangles in the $1/2/3$-puzzle $P=\mathscr{T}'(T)$ from \Cref{fig:123-puzzle-example}.
\end{ex}

\section{From puzzles to graph colorings}
\label{sec:puzzle-to-colorings}
In this section, we relate interlacing triangular arrays and $123$-puzzles to graph colorings. We prove Conjecture~A.3 of \cite{Aggarwal-Borodin-Wheeler} and disprove Conjecture~A.5.  

\subsection{$\T_{3,n}$ and the triangular grid}
\label{sec:triangular-coloring}
Recall from \Cref{sec:puzzle} that $\Delta_n$ denotes the triangular grid graph with side length $n$, a graph on ${n+2\choose 2}$ vertices. Let $\mc{C}_n$ denote the set of proper colorings of the vertices of $\Delta_n$ by $\{0,1,2,3\}$ such that the base point is colored $0$. We write $\mc{C}_n(\bs{\kappa})$ for the set of colorings from $\mc{C}_n$ with fixed coloring $\bs{\kappa}$ of the boundary vertices. We have proven in \Cref{thm:psi-is-bijection} that the map $\mathscr{T}$ is a bijection $\mathscr{T}:\mc{P}_n \to \mc{T}_{3,n}$. We now define a map $\mathscr{P}: \mc{C}_n \to \mc{P}_n$, proven in \Cref{thm:counting-m3} to be a bijection.

\begin{defin}
    Given a proper vertex coloring $C\in\mc{C}_n$, define an edge labeling $\mathscr{P}(C)$ as follows. For an edge $e$ of $\Delta_n$ incident to vertices $v$ and $v'$, if $0 \in \{c(v),c(v')\}$, let\footnote{We thank the anonymous referee for providing the following interesting equivalent description of $\mathscr{P}(C)(e)$. Let $i=c(v)+c(v')$, then $\mathscr{P}(C)(e)=\begin{cases}
        i & \text{if $i\leq 3$} \\ 6-i & \text{otherwise}
    \end{cases}$.} $\mathscr{P}(C)(e)$ be the unique element of $\{c(v),c(v')\} \setminus \{0\}$, otherwise, let $\mathscr{P}(C)(e)$ be the unique element of $\{0,1,2,3\} \setminus \{0,c(v),c(v')\}$. See \Cref{fig:phi-example} for an example. By \Cref{thm:counting-m3} below, $\mathscr{P}$ is invertible on $1/2/3$-puzzles and so determines boundary colors $\colr(\bs{\lambda})$ given boundary conditions $\bs{\lambda}$ of a puzzle.
\end{defin}

\begin{theorem}
\label{thm:counting-m3}
The map $\mathscr{P}$ is a bijection $\mc{C}_n\rightarrow \mc{P}_n$ sending $\mc{C}(\colr(\bs{\lambda}))$ to $\mc{P}_n(\bs{\lambda})$.
\end{theorem}
\begin{proof}
We first show that $\mathscr{P}(C)\in\mc{P}_n$. Indeed, if two edges $vv'$ and $vv''$ bound the same face of $\Delta_n$ and are assigned the same label by $\mathscr{P}(C)$, we must have $C(v')=C(v'')$, contradicting the properness of $C$.

We now define a map $\mathscr{C}:\mc{P}_n\rightarrow\mc{C}_n$ as follows. Given $P\in\mc{P}_n$, color the base point of $\Delta_n$ by $0$ and determine the color of any other vertex $v$ by choosing a neighboring vertex $v'$ which has already been colored and letting $\mathscr{C}(P)(v)=P(e)$ if $\mathscr{C}(P)(v')=0$, letting $\mathscr{C}(P)(v)=0$ if $\mathscr{C}(P)(v')=P(e)$, and otherwise letting $\mathscr{C}(P)(v)$ be the unique element of $\{1,2,3\}\setminus\{\mathscr{C}(P)(v'),P(e)\}$. By inspecting all the configurations of a face, one easily checks that the color assigned to a vertex is the same no matter which way around the triangle the process proceeds. Thus $\mathscr{C}$ is well-defined. By construction, we see that $\mathscr{P}$ and $\mathscr{C}$ are mutually inverse bijections.
\end{proof}
\begin{figure}[h!]
\centering
\raisebox{-0.5\height}{
\begin{tikzpicture}[scale=0.95]
\draw[thin,red!80](0,0)--(4,0)--(2.00000000000000,3.46400000000000)--(0,0);
\draw[thin,red!80](0.500000000000000,0.866000000000000)--(3.50000000000000,0.866000000000000);
\draw[thin,red!80](1,0)--(2.50000000000000,2.59800000000000);
\draw[thin,red!80](0.500000000000000,0.866000000000000)--(1,0);
\draw[thin,red!80](1.00000000000000,1.73200000000000)--(3.00000000000000,1.73200000000000);
\draw[thin,red!80](2,0)--(3.00000000000000,1.73200000000000);
\draw[thin,red!80](1.00000000000000,1.73200000000000)--(2,0);
\draw[thin,red!80](1.50000000000000,2.59800000000000)--(2.50000000000000,2.59800000000000);
\draw[thin,red!80](3,0)--(3.50000000000000,0.866000000000000);
\draw[thin,red!80](1.50000000000000,2.59800000000000)--(3,0);
\node at (0.000000000000000,0.000000000000000) {$0$};
\node at (1.00000000000000,0.000000000000000) {$2$};
\node at (2.00000000000000,0.000000000000000) {$1$};
\node at (3.00000000000000,0.000000000000000) {$3$};
\node at (4.00000000000000,0.000000000000000) {$2$};
\node at (0.500000000000000,0.866000000000000) {$1$};
\node at (1.50000000000000,0.866000000000000) {$0$};
\node at (2.50000000000000,0.866000000000000) {$2$};
\node at (3.50000000000000,0.866000000000000) {$0$};
\node at (1.00000000000000,1.73200000000000) {$3$};
\node at (2.00000000000000,1.73200000000000) {$1$};
\node at (3.00000000000000,1.73200000000000) {$3$};
\node at (1.50000000000000,2.59800000000000) {$2$};
\node at (2.50000000000000,2.59800000000000) {$0$};
\node at (2.00000000000000,3.46400000000000) {$1$};
\end{tikzpicture}}
\qquad $\xmapsto{\mathscr{P}}$ \qquad
\raisebox{-0.5\height}{
\begin{tikzpicture}[scale=0.95]
\draw[thin,blue!80](0,0)--(4,0)--(2.00000000000000,3.46400000000000)--(0,0);
\draw[thin,blue!80](0.500000000000000,0.866000000000000)--(3.50000000000000,0.866000000000000);
\draw[thin,blue!80](1,0)--(2.50000000000000,2.59800000000000);
\draw[thin,blue!80](0.500000000000000,0.866000000000000)--(1,0);
\draw[thin,blue!80](1.00000000000000,1.73200000000000)--(3.00000000000000,1.73200000000000);
\draw[thin,blue!80](2,0)--(3.00000000000000,1.73200000000000);
\draw[thin,blue!80](1.00000000000000,1.73200000000000)--(2,0);
\draw[thin,blue!80](1.50000000000000,2.59800000000000)--(2.50000000000000,2.59800000000000);
\draw[thin,blue!80](3,0)--(3.50000000000000,0.866000000000000);
\draw[thin,blue!80](1.50000000000000,2.59800000000000)--(3,0);

\node at (0.250000000000000,0.433000000000000) {$1$};
\node at (2.25000000000000,3.03100000000000) {$1$};
\node at (0.500000000000000,0.000000000000000) {$2$};
\node at (0.750000000000000,1.29900000000000) {$2$};
\node at (2.75000000000000,2.16500000000000) {$3$};
\node at (1.50000000000000,0.000000000000000) {$3$};
\node at (1.25000000000000,2.16500000000000) {$1$};
\node at (3.25000000000000,1.29900000000000) {$3$};
\node at (2.50000000000000,0.000000000000000) {$2$};
\node at (1.75000000000000,3.03100000000000) {$3$};
\node at (3.75000000000000,0.433000000000000) {$2$};
\node at (3.50000000000000,0.000000000000000) {$1$};
\node at (1.25000000000000,0.433000000000000) {$2$};
\node at (1.75000000000000,2.16500000000000) {$3$};
\node at (1.00000000000000,0.866000000000000) {$1$};
\node at (1.75000000000000,1.29900000000000) {$1$};
\node at (2.25000000000000,1.29900000000000) {$3$};
\node at (2.00000000000000,0.866000000000000) {$2$};
\node at (2.25000000000000,2.16500000000000) {$1$};
\node at (2.75000000000000,0.433000000000000) {$1$};
\node at (3.00000000000000,0.866000000000000) {$2$};
\node at (2.25000000000000,0.433000000000000) {$3$};
\node at (1.25000000000000,1.29900000000000) {$3$};
\node at (1.50000000000000,1.73200000000000) {$2$};
\node at (2.75000000000000,1.29900000000000) {$1$};
\node at (1.75000000000000,0.433000000000000) {$1$};
\node at (2.50000000000000,1.73200000000000) {$2$};
\node at (3.25000000000000,0.433000000000000) {$3$};
\node at (0.750000000000000,0.433000000000000) {$3$};
\node at (2.00000000000000,2.59800000000000) {$2$};
\end{tikzpicture}}
\caption{The bijection $\mathscr{P}$ from \Cref{thm:counting-m3}.}
\label{fig:phi-example}
\end{figure}

\subsection{$\T_{4,n}$ and the square grid}
\label{sec:squares}

Let $\square_n$ be the square grid graph with side length $n$, having $(n+1)^2$ vertices and $n^2$ faces. For $k<n$, we view $\square_k$ as a southwest-justified subgraph of $\square_n$, and we fix the southwest corner as the basepoint.

Let $\mc{D}_n$ be the set of edge labelings of $\square_n$ with labels $\{1,2,3,4\}$ that satisfy the following conditions:
\begin{itemize}
\item for each of the $n$ faces on the main southwest-to-northeast diagonal, all four edges are labeled differently;
\item for each of the off-diagonal faces, the four edges are labeled with exactly two distinct labels so that either one label is assigned to the west and south boundaries (with the other label assigned to the north and east boundaries), or one label is assigned to the north and south boundaries (with the other label assigned to the west and east boundaries).
\end{itemize}

We now define a map $\mathscr{D}'$ from $\mathcal{D}_n$ to certain triangular arrays of integers.

\begin{defin}
\label{def:D-to-T4}
Given an edge labeling $D\in\mc{D}_n$, let $\mathscr{D}'(D)$ be the triangular array of integers whose $k$-th row is obtained by reading the $D$-labels around the boundary of $\square_k$ in the clockwise direction, starting from the base point. See \Cref{fig:phi-example-m4} for an example.
\end{defin}

\begin{figure}[h!]
\centering
\raisebox{-.5\height}{
\begin{tikzpicture}[scale=0.8]
\draw[thin,blue!80] (0,0)--(4,0);
\draw[thin,blue!80] (0,1)--(4,1);
\draw[thin,blue!80] (0,2)--(4,2);
\draw[thin,blue!80] (0,3)--(4,3);
\draw[thin,blue!80] (0,4)--(4,4);
\draw[thin,blue!80] (0,0)--(0,4);
\draw[thin,blue!80] (1,0)--(1,4);
\draw[thin,blue!80] (2,0)--(2,4);
\draw[thin,blue!80] (3,0)--(3,4);
\draw[thin,blue!80] (4,0)--(4,4);
\filldraw[blue!80,opacity=0.2] (0,0)--(0,1)--(1,1)--(1,0);
\filldraw[blue!80,opacity=0.2] (1,1)--(1,2)--(2,2)--(2,1);
\filldraw[blue!80,opacity=0.2] (2,2)--(2,3)--(3,3)--(3,2);
\filldraw[blue!80,opacity=0.2] (3,3)--(3,4)--(4,4)--(4,3);
\draw[thin,blue!80](0.500000000000000,1)--(0.500000000000000,2);
\draw[thin,blue!80](0,1.50000000000000)--(1,1.50000000000000);
\draw[thin,blue!80](0.500000000000000,2)--(0.500000000000000,3);
\draw[thin,blue!80](0,2.50000000000000)--(1,2.50000000000000);
\draw[thin,blue!80](0.500000000000000,3.20000000000000)arc(0:90:0.300000000000000);
\draw[thin,blue!80](0.500000000000000,3.80000000000000)arc(180:270:0.300000000000000);
\draw[thin,blue!80](0.500000000000000,3)--(0.500000000000000,3.20000000000000);
\draw[thin,blue!80](0.200000000000000,3.50000000000000)--(0,3.50000000000000);
\draw[thin,blue!80](0.500000000000000,4)--(0.500000000000000,3.80000000000000);
\draw[thin,blue!80](0.800000000000000,3.50000000000000)--(1,3.50000000000000);
\draw[thin,blue!80](1.50000000000000,0)--(1.50000000000000,1);
\draw[thin,blue!80](1,0.500000000000000)--(2,0.500000000000000);
\draw[thin,blue!80](1.50000000000000,2)--(1.50000000000000,3);
\draw[thin,blue!80](1,2.50000000000000)--(2,2.50000000000000);
\draw[thin,blue!80](1.50000000000000,3.20000000000000)arc(0:90:0.300000000000000);
\draw[thin,blue!80](1.50000000000000,3.80000000000000)arc(180:270:0.300000000000000);
\draw[thin,blue!80](1.50000000000000,3)--(1.50000000000000,3.20000000000000);
\draw[thin,blue!80](1.20000000000000,3.50000000000000)--(1,3.50000000000000);
\draw[thin,blue!80](1.50000000000000,4)--(1.50000000000000,3.80000000000000);
\draw[thin,blue!80](1.80000000000000,3.50000000000000)--(2,3.50000000000000);
\draw[thin,blue!80](2.50000000000000,0)--(2.50000000000000,1);
\draw[thin,blue!80](2,0.500000000000000)--(3,0.500000000000000);
\draw[thin,blue!80](2.50000000000000,1)--(2.50000000000000,2);
\draw[thin,blue!80](2,1.50000000000000)--(3,1.50000000000000);
\draw[thin,blue!80](2.50000000000000,3.20000000000000)arc(0:90:0.300000000000000);
\draw[thin,blue!80](2.50000000000000,3.80000000000000)arc(180:270:0.300000000000000);
\draw[thin,blue!80](2.50000000000000,3)--(2.50000000000000,3.20000000000000);
\draw[thin,blue!80](2.20000000000000,3.50000000000000)--(2,3.50000000000000);
\draw[thin,blue!80](2.50000000000000,4)--(2.50000000000000,3.80000000000000);
\draw[thin,blue!80](2.80000000000000,3.50000000000000)--(3,3.50000000000000);
\draw[thin,blue!80](3.50000000000000,0.200000000000000)arc(0:90:0.300000000000000);
\draw[thin,blue!80](3.50000000000000,0.800000000000000)arc(180:270:0.300000000000000);
\draw[thin,blue!80](3.50000000000000,0)--(3.50000000000000,0.200000000000000);
\draw[thin,blue!80](3.20000000000000,0.500000000000000)--(3,0.500000000000000);
\draw[thin,blue!80](3.50000000000000,1)--(3.50000000000000,0.800000000000000);
\draw[thin,blue!80](3.80000000000000,0.500000000000000)--(4,0.500000000000000);
\draw[thin,blue!80](3.50000000000000,1)--(3.50000000000000,2);
\draw[thin,blue!80](3,1.50000000000000)--(4,1.50000000000000);
\draw[thin,blue!80](3.50000000000000,2)--(3.50000000000000,3);
\draw[thin,blue!80](3,2.50000000000000)--(4,2.50000000000000);
\node at (0.5,0) {$1$};
\node at (1.5,0) {$3$};
\node at (2.5,0) {$1$};
\node at (3.5,0) {$4$};
\node at (0.5,1) {$3$};
\node at (1.5,1) {$3$};
\node at (2.5,1) {$1$};
\node at (3.5,1) {$1$};
\node at (0.5,2) {$3$};
\node at (1.5,2) {$2$};
\node at (2.5,2) {$1$};
\node at (3.5,2) {$1$};
\node at (0.5,3) {$3$};
\node at (1.5,3) {$2$};
\node at (2.5,3) {$3$};
\node at (3.5,3) {$1$};
\node at (0.5,4) {$2$};
\node at (1.5,4) {$3$};
\node at (2.5,4) {$4$};
\node at (3.5,4) {$2$};
\node at (0,0.5) {$2$};
\node at (0,1.5) {$1$};
\node at (0,2.5) {$4$};
\node at (0,3.5) {$3$};
\node at (1,0.5) {$4$};
\node at (1,1.5) {$1$};
\node at (1,2.5) {$4$};
\node at (1,3.5) {$2$};
\node at (2,0.5) {$4$};
\node at (2,1.5) {$4$};
\node at (2,2.5) {$4$};
\node at (2,3.5) {$3$};
\node at (3,0.5) {$4$};
\node at (3,1.5) {$4$};
\node at (3,2.5) {$2$};
\node at (3,3.5) {$4$};
\node at (4,0.5) {$1$};
\node at (4,1.5) {$4$};
\node at (4,2.5) {$2$};
\node at (4,3.5) {$3$};
\node at (0,0) {$\bullet$};
\end{tikzpicture}}
\quad $\xmapsto{\mathscr{D}'}$ \quad
\raisebox{-.6\height}{
\begin{tikzpicture}[scale=0.52]
\def\sep{1.0};
\draw[green!80,thin](0.0,0.0)--(3.0,0.0);
\draw[green!80,thin](3.0,0)--(1.5,-2.598);
\draw[green!80,thin](0.0,0)--(1.5,-2.598);
\draw[green!80,thin](0.5,-0.866)--(2.5,-0.866);
\draw[green!80,thin](2.0,0)--(1.0,-1.7319999999999998);
\draw[green!80,thin](1.0,0)--(2.0,-1.7319999999999998);
\draw[green!80,thin](1.0,-1.732)--(2.0,-1.732);
\draw[green!80,thin](1.0,0)--(0.5,-0.8659999999999999);
\draw[green!80,thin](2.0,0)--(2.5,-0.8659999999999999);
\draw[green!80,thin](1.5,-2.598)--(1.5,-2.598);
\draw[green!80,thin](0.0,0)--(0.0,0.0);
\draw[green!80,thin](3.0,0)--(3.0,0.0);
\draw[green!80,thin](4.0,0.0)--(7.0,0.0);
\draw[green!80,thin](7.0,0)--(5.5,-2.598);
\draw[green!80,thin](4.0,0)--(5.5,-2.598);
\draw[green!80,thin](4.5,-0.866)--(6.5,-0.866);
\draw[green!80,thin](6.0,0)--(5.0,-1.7319999999999998);
\draw[green!80,thin](5.0,0)--(6.0,-1.7319999999999998);
\draw[green!80,thin](5.0,-1.732)--(6.0,-1.732);
\draw[green!80,thin](5.0,0)--(4.5,-0.8659999999999999);
\draw[green!80,thin](6.0,0)--(6.5,-0.8659999999999999);
\draw[green!80,thin](5.5,-2.598)--(5.5,-2.598);
\draw[green!80,thin](4.0,0)--(4.0,0.0);
\draw[green!80,thin](7.0,0)--(7.0,0.0);
\draw[green!80,thin](8.0,0.0)--(11.0,0.0);
\draw[green!80,thin](11.0,0)--(9.5,-2.598);
\draw[green!80,thin](8.0,0)--(9.5,-2.598);
\draw[green!80,thin](8.5,-0.866)--(10.5,-0.866);
\draw[green!80,thin](10.0,0)--(9.0,-1.7319999999999998);
\draw[green!80,thin](9.0,0)--(10.0,-1.7319999999999998);
\draw[green!80,thin](9.0,-1.732)--(10.0,-1.732);
\draw[green!80,thin](9.0,0)--(8.5,-0.8659999999999999);
\draw[green!80,thin](10.0,0)--(10.5,-0.8659999999999999);
\draw[green!80,thin](9.5,-2.598)--(9.5,-2.598);
\draw[green!80,thin](8.0,0)--(8.0,0.0);
\draw[green!80,thin](11.0,0)--(11.0,0.0);
\draw[green!80,thin](12.0,0.0)--(15.0,0.0);
\draw[green!80,thin](15.0,0)--(13.5,-2.598);
\draw[green!80,thin](12.0,0)--(13.5,-2.598);
\draw[green!80,thin](12.5,-0.866)--(14.5,-0.866);
\draw[green!80,thin](14.0,0)--(13.0,-1.7319999999999998);
\draw[green!80,thin](13.0,0)--(14.0,-1.7319999999999998);
\draw[green!80,thin](13.0,-1.732)--(14.0,-1.732);
\draw[green!80,thin](13.0,0)--(12.5,-0.8659999999999999);
\draw[green!80,thin](14.0,0)--(14.5,-0.8659999999999999);
\draw[green!80,thin](13.5,-2.598)--(13.5,-2.598);
\draw[green!80,thin](12.0,0)--(12.0,0.0);
\draw[green!80,thin](15.0,0)--(15.0,0.0);
\node at (0.0,-0.0) {$2$};
\node at (1.0,-0.0) {$1$};
\node at (2.0,-0.0) {$4$};
\node at (3.0,-0.0) {$3$};
\node at (0.5,-0.866) {$2$};
\node at (1.5,-0.866) {$1$};
\node at (2.5,-0.866) {$4$};
\node at (1.0,-1.732) {$2$};
\node at (2.0,-1.732) {$1$};
\node at (1.5,-2.598) {$2$};
\node at (4.0,-0.0) {$2$};
\node at (5.0,-0.0) {$3$};
\node at (6.0,-0.0) {$4$};
\node at (7.0,-0.0) {$2$};
\node at (4.5,-0.866) {$3$};
\node at (5.5,-0.866) {$2$};
\node at (6.5,-0.866) {$3$};
\node at (5.0,-1.732) {$3$};
\node at (6.0,-1.732) {$2$};
\node at (5.5,-2.598) {$3$};
\node at (8.0,-0.0) {$3$};
\node at (9.0,-0.0) {$2$};
\node at (10.0,-0.0) {$4$};
\node at (11.0,-0.0) {$1$};
\node at (8.5,-0.866) {$2$};
\node at (9.5,-0.866) {$4$};
\node at (10.5,-0.866) {$4$};
\node at (9.0,-1.732) {$4$};
\node at (10.0,-1.732) {$4$};
\node at (9.5,-2.598) {$4$};
\node at (12.0,-0.0) {$4$};
\node at (13.0,-0.0) {$1$};
\node at (14.0,-0.0) {$3$};
\node at (15.0,-0.0) {$1$};
\node at (12.5,-0.866) {$1$};
\node at (13.5,-0.866) {$3$};
\node at (14.5,-0.866) {$1$};
\node at (13.0,-1.732) {$3$};
\node at (14.0,-1.732) {$1$};
\node at (13.5,-2.598) {$1$};
\node at (10,-4.2) {};
\end{tikzpicture}}
\caption{The bijection $\mathscr{D}'$ from \Cref{thm:counting-m4}.}
\label{fig:phi-example-m4}
\end{figure}

\begin{theorem}\label{thm:counting-m4}
The map $\mathscr{D}'$ is a bijection $\mc{D}_n\rightarrow\T_{4,n}$.
\end{theorem}
\begin{proof}
Given $D\in\mc{D}_n$, we first show that $\mathscr{D}'(D)$ is an interlacing triangular array. We augment $D$ with a \emph{strand diagram} (these look similar to the \emph{pipe dreams} of \cite{bergeron-billey,knutson-miller} that might be familiar to some readers). For each face off the main diagonal, we connect the edges with the same labels, forming strands connecting the boundary edges of $\square_n$ to the edges on the main diagonal (see \Cref{fig:phi-example-m4}). By the definition of $\mc{D}_n$, strands with the same label do not intersect. For each $k=1, \ldots, n$, each of the $k$ faces of $\square_k$ on the main diagonal has one edge labeled $j$, for each $j\in[4]$. Thus the boundary of $\square_k$ has a total of $k$ edges labeled by $j$ and so the $k$-th row of $\mathscr{D}'(D)$ contains exactly $k$ copies of $j$. 

We now check the interlacing condition. Consider two $j$'s on the boundary of $\square_n$, and keep track of their corresponding non-intersecting strands (the strand may have length $0$ if one of these $j$'s bounds a face on the main diagonal). We call these two strands $j_1$ and $j_2$ in clockwise order from the base point. If both strands lie on the northwest side of the main diagonal, then strand $j_1$ must pass through the boundary of $\square_{n-1}$. This intersection point corresponds to a $j$ in row $n-1$ of $\mathscr{D}'(D)$ interlacing the two $j$'s in row $n$ that we started with. The case where both strands lie on the southeast side of the main diagonal is similar. Therefore we can now assume that $j_1$ lies on the northwest side and $j_2$ on the southeast side of the main diagonal. If either $j_1$ or $j_2$ intersects the boundary of $\square_{n-1}$, we find an interlacing $j$ by locating the intersection, as above. If $j_1$ does not intersect the boundary of $\square_{n-1}$, then it must terminate at the face $F$ in the northeastern corner of $\square_n$. The same is true for $j_2$. But $F$ lies on the main diagonal and thus cannot have two edges with the same label, a contradiction. We conclude that $\mathscr{D}'(D)$ is interlacing. 

Next, we construct a map $\mathscr{D}:\T_{4,n}\rightarrow\mc{D}_n$ inverse to $\mathscr{D}'$. Let $T\in\T_{4,n}$ and define $D=\mathscr{D}(T)$ by first labeling the boundary of $\square_k$ clockwise according to the $k$-th row of $T$, for $k=1,\ldots,n$. We now assign labels to other edges of $\square_n$; it suffices to consider the \emph{hook} $H_n:=\square_n\setminus\square_{n-1}$, shown in \Cref{fig:phi-hook}. It contains a horizontal leg on top and a vertical leg on the right whose intersection is the face $F$ in the northeast corner of $\square_n$. The hook $H_n$ has an outer boundary given by the boundary of $\square_n$, and an inner boundary given by the boundary of $\square_{n-1}$. There are $2n+2$ edges on the outer boundary and $2n-2$ edges on the inner boundary. 
\begin{figure}[h!]
\centering
\begin{tikzpicture}[scale=0.8]
\draw[thin,blue!80](0,3)--(0,4)--(4,4)--(4,0)--(3,0)--(3,3)--(0,3);
\draw[thin,blue!80](1,3)--(1,4);
\draw[thin,blue!80](2,3)--(2,4);
\draw[thin,blue!80](4,3)--(3,3)--(3,4);
\draw[thin,blue!80](3,1)--(4,1);
\draw[thin,blue!80](3,2)--(4,2);
\filldraw[blue!80,opacity=0.2] (3,3)--(3,4)--(4,4)--(4,3);
\filldraw[blue!80,opacity=0.4] (0.4,3.4)--(0.4,3.6)--(2.4,3.6)--(2.4,3.7)--(2.7,3.5)--(2.4,3.3)--(2.4,3.4)--(0.4,3.4);
\filldraw[blue!80,opacity=0.4] (3.4,0.4)--(3.6,0.4)--(3.6,2.4)--(3.7,2.4)--(3.5,2.7)--(3.3,2.4)--(3.4,2.4)--(3.4,0.4);

\node at (0,3.5) {$3$};
\node at (0.5,4) {$2$};
\node at (1.5,4) {$3$};
\node at (2.5,4) {$4$};
\node at (3.5,4) {$2$};
\node at (0.5,3) {$3$};
\node at (1.5,3) {$2$};
\node at (2.5,3) {$3$};
\node at (4,3.5) {$3$};
\node at (4,2.5) {$2$};
\node at (4,1.5) {$4$};
\node at (4,0.5) {$1$};
\node at (3.5,0) {$4$};
\node at (3,2.5) {$2$};
\node at (3,1.5) {$4$};
\node at (3,0.5) {$4$};
\end{tikzpicture}
\caption{A labeling of $H_n$ in the construction of $\mathscr{D}$ in \Cref{thm:counting-m4}.}
\label{fig:phi-hook}
\end{figure}
We label the interior edges of the horizontal leg from left to right. 

Suppose that we have moved to a face $F'$ whose western, northern, and southern boundaries are already labeled $a,b,d$ respectively. We need to specify the label $c$ of its eastern boundary. It is clear that there is at most one choice for $c$ compatible with the conditions for $\mc{D}_n$. We will show that there is in fact a compatible choice for $c$. Connect the strands as in \Cref{fig:phi-example-m4} in the squares of $H_n$ to the left of $F'$. If $a=b=j$, there is a strand labeled $j$ starting from the northern boundary of $F'$, continuing to the western boundary of $F'$, and terminating at an edge $e$ on the outer boundary of $H_n$. By the conditions of $\mc{D}_n$ for squares to the left of $F'$, we see that there are no other $j$'s in row $n-1$ of $T$ between the entries that correspond to $e$ and to $b$ on row $n$ of $T$, contradicting the interlacing condition of $T$. Therefore $a\neq b$. Now, if $d$ is also different from both $a$ and $b$, then in the interlacing triangular array $T$, strictly to the left of the entry in row $n-1$ that corresponds to $d$, the number of appearances of this label in row $n-1$ equals the number of appearances in row $n$, because the strands in $H_n$ so far provide a bijection. However this is already a contradiction of the interlacing condition, because this label would occur more times in row $n-1$ than in row $n$, weakly to the left of the corresponding entry for the south boundary of $F'$. Thus $a\neq b$ and $d\in\{a,b\}$, so $c$ is uniquely determined.

The vertical leg can similarly be labeled from bottom to top. The strands provide an injection from the multi-set of labels on the inner boundary of $H_n$ to those of the outer boundary. The labels of the four remaining edges on the outer boundary also biject onto the labels of $F$. Since there is precisely one more $j$ on row $n$ of $T$ than on row $n-1$, we conclude that the four sides of $F$ have distinct labels. Hence $\mathscr{D}$ is well defined and by construction is inverse to $\mathscr{D}'$.
\end{proof}

\begin{remark}
By enumerating the labelings $\mc{D}_n$, we have computed\footnote{Code is available at \url{http://bicmr.pku.edu.cn/~gaoyibo/assets/interlacing.sage}.} for $n=0,1,2,3$ that $|\T_{4,n}|=1, \: 24, \: 1344,\: 191232$. This last value disagrees with the quantity $\frac{1}{5} |\{\text{proper vertex $5$-colorings of } \boxtimes_n\}|$ from \cite[Conjecture~A.5]{Aggarwal-Borodin-Wheeler} which for $n=0,1,2,3$ is equal\footnote{See the OEIS entry A068294.} to $1, \: 24, \: 1344, \: 187008$, disproving the conjecture of Aggarwal--Borodin--Wheeler.

We do, however, make a new conjecture for $|\T_{4,n}|$ that has been checked up to $n=7$. This replaces vertex colorings of $\boxtimes_n$ with edge labelings of $\square_n$ and is a direct rank-$4$ extension of the equinumerosity of $\mc{P}_n$ and $\T_{3,n}$.
\end{remark}

\begin{conj}\label{conj:counting-m4-proper-labeling}
$|\T_{4,n}|$ equals the number of edge labelings of $\square_n$ with four labels such that the four sides of each face have distinct labels. 
\end{conj}

\section{Schubert structure constants}
\label{sec:geometry}
In this section we derive the geometric results stated in \Cref{sec:intro-geometry}.

\subsection{Puzzle conversion}
\label{sec:puzzle-conversion}
A \emph{$0/1/10$-puzzle} is a labeling of the edges of $\Delta_n$ with labels $0$, $1$, and $10$ so that each $\Delta$-oriented face is labeled in one of the following ways

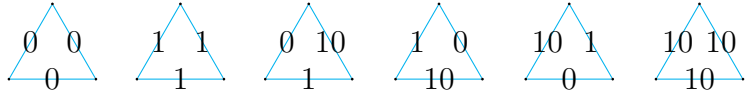
\begin{center}
\begin{tikzpicture}[scale=1, rotate=-60]
    \node[circle, fill=black, scale=0.1] (c0) at ({-2*sqrt(3)/3},2) {};
    \node[circle, fill=black, scale=0.1] (d0) at ({-3*sqrt(3)/3},3) {};
    \node[circle, fill=black, scale=0.1] (d1) at ({-1*sqrt(3)/3},3) {};
    \draw[cyan!90,thin] (c0) -- (d0) node [midway] {\large\textcolor{black}{$0$}};
    \draw[cyan!90,thin] (c0) -- (d1) node [midway] {\large\textcolor{black}{$0$}};
    \draw[cyan!90,thin] (d0) -- (d1) node [midway] {\large\textcolor{black}{$0$}};
\end{tikzpicture} \quad
\begin{tikzpicture}[scale=1, rotate=-60]
    \node[circle, fill=black, scale=0.1] (c0) at ({-2*sqrt(3)/3},2) {};
    \node[circle, fill=black, scale=0.1] (d0) at ({-3*sqrt(3)/3},3) {};
    \node[circle, fill=black, scale=0.1] (d1) at ({-1*sqrt(3)/3},3) {};
    \draw[cyan!90,thin] (c0) -- (d0) node [midway] {\large\textcolor{black}{$1$}};
    \draw[cyan!90,thin] (c0) -- (d1) node [midway] {\large\textcolor{black}{$1$}};
    \draw[cyan!90,thin] (d0) -- (d1) node [midway] {\large\textcolor{black}{$1$}};
\end{tikzpicture} \quad
\begin{tikzpicture}[scale=1, rotate=-60]
    \node[circle, fill=black, scale=0.1] (c0) at ({-2*sqrt(3)/3},2) {};
    \node[circle, fill=black, scale=0.1] (d0) at ({-3*sqrt(3)/3},3) {};
    \node[circle, fill=black, scale=0.1] (d1) at ({-1*sqrt(3)/3},3) {};
    \draw[cyan!90,thin] (c0) -- (d0) node [midway] {\large\textcolor{black}{$0$}};
    \draw[cyan!90,thin] (c0) -- (d1) node [midway] {\large\textcolor{black}{$1$}};
    \draw[cyan!90,thin] (d0) -- (d1) node [midway] {\large\textcolor{black}{$10$}};
\end{tikzpicture} \quad
\begin{tikzpicture}[scale=1, rotate=-60]
    \node[circle, fill=black, scale=0.1] (c0) at ({-2*sqrt(3)/3},2) {};
    \node[circle, fill=black, scale=0.1] (d0) at ({-3*sqrt(3)/3},3) {};
    \node[circle, fill=black, scale=0.1] (d1) at ({-1*sqrt(3)/3},3) {};
    \draw[cyan!90,thin] (c0) -- (d0) node [midway] {\large\textcolor{black}{$1$}};
    \draw[cyan!90,thin] (c0) -- (d1) node [midway] {\large\textcolor{black}{$10$}};
    \draw[cyan!90,thin] (d0) -- (d1) node [midway] {\large\textcolor{black}{$0$}};
\end{tikzpicture} \quad
\begin{tikzpicture}[scale=1, rotate=-60]
    \node[circle, fill=black, scale=0.1] (c0) at ({-2*sqrt(3)/3},2) {};
    \node[circle, fill=black, scale=0.1] (d0) at ({-3*sqrt(3)/3},3) {};
    \node[circle, fill=black, scale=0.1] (d1) at ({-1*sqrt(3)/3},3) {};
    \draw[cyan!90,thin] (c0) -- (d0) node [midway] {\large\textcolor{black}{$10$}};
    \draw[cyan!90,thin] (c0) -- (d1) node [midway] {\large\textcolor{black}{$0$}};
    \draw[cyan!90,thin] (d0) -- (d1) node [midway] {\large\textcolor{black}{$1$}};
\end{tikzpicture} \quad
\begin{tikzpicture}[scale=1, rotate=-60]
    \node[circle, fill=black, scale=0.1] (c0) at ({-2*sqrt(3)/3},2) {};
    \node[circle, fill=black, scale=0.1] (d0) at ({-3*sqrt(3)/3},3) {};
    \node[circle, fill=black, scale=0.1] (d1) at ({-1*sqrt(3)/3},3) {};
    \draw[cyan!90,thin] (c0) -- (d0) node [midway] {\large\textcolor{black}{$10$}};
    \draw[cyan!90,thin] (c0) -- (d1) node [midway] {\large\textcolor{black}{$10$}};
    \draw[cyan!90,thin] (d0) -- (d1) node [midway] {\large\textcolor{black}{$10$}};
\end{tikzpicture},
\end{center}
and so that each $\nabla$-oriented face is labeled by a $180^{\circ}$ rotation of one of these\footnote{See \cite[\S5]{Knutson-Zinn-Justin-3} for the relationship between these puzzles and others which have appeared in the literature.}. The boundary conditions of such a puzzle are the $0,1$-strings $\xi^{(1)},\xi^{(2)},\xi^{(3)}$ obtained by reading the labels of the boundary edges of the three sides of $\Delta_n$ clockwise from the basepoint. We write $\tilde{\mc{P}}_n(\bs{\xi})$ for the set of $0/1/10$-puzzles on $\Delta_n$ with boundary conditions $\bs{\xi}$ and $\tilde{\mc{P}}_n$ for the set of all $0/1/10$-puzzles on $\Delta_n$. 

We are grateful to Allen Knutson for sharing with us the following correspondence between $1/2/3$-puzzles and $0/1/10$-puzzles. Given a boundary condition $\bs{\lambda}=(\lambda^{(1)},\lambda^{(2)},\lambda^{(3)})$, let $\partn(\bs{\lambda})$ be the be the $0,1$-strings $(\xi^{(1)},\xi^{(2)},\xi^{(3)})$ obtained by applying the transformation below:

    
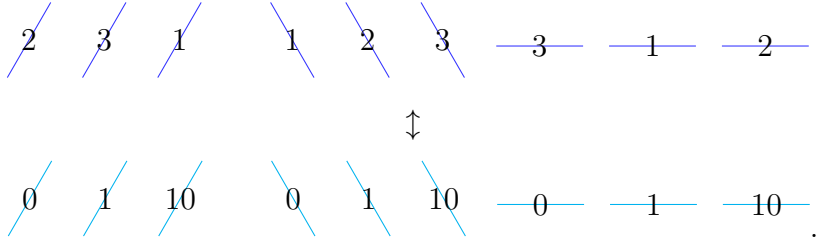
\begin{center}
    \begin{tikzpicture}[scale=1]
    \draw[blue!80,thin] (0,0) -- ({sqrt(3)/3},1) node [midway] {\large\textcolor{black}{$2$}};
    \draw[blue!80,thin] (1,0) -- ({1+sqrt(3)/3},1) node [midway] {\large\textcolor{black}{$3$}};
    \draw[blue!80,thin] (2,0) -- ({2+sqrt(3)/3},1) node [midway] {\large\textcolor{black}{$1$}};
    \draw[blue!80,thin] (6.5,0.42) -- ({6.5+2*sqrt(3)/3},.42) node [midway] {\large\textcolor{black}{$3$}};
    \draw[blue!80,thin] (8,.42) -- ({8+2*sqrt(3)/3},.42) node [midway] {\large\textcolor{black}{$1$}};
    \draw[blue!80,thin] (9.5,.42) -- ({9.5+2*sqrt(3)/3},.42) node [midway] {\large\textcolor{black}{$2$}};
    \draw[blue!80,thin] (3.5,1) -- ({3.5+sqrt(3)/3},0) node [midway] {\large\textcolor{black}{$1$}};
    \draw[blue!80,thin] (4.5,1) -- ({4.5+sqrt(3)/3},0) node [midway] {\large\textcolor{black}{$2$}};
    \draw[blue!80,thin] (5.5,1) -- ({5.5+sqrt(3)/3},0) node [midway] {\large\textcolor{black}{$3$}};
\end{tikzpicture} \vspace{0.1in}
\\ $\updownarrow$ \\ \vspace{0.1in}
\begin{tikzpicture}[scale=1]
    \draw[cyan!90,thin] (0,0) -- ({sqrt(3)/3},1) node [midway] {\large\textcolor{black}{$0$}};
    \draw[cyan!90,thin] (1,0) -- ({1+sqrt(3)/3},1) node [midway] {\large\textcolor{black}{$1$}};
    \draw[cyan!90,thin] (2,0) -- ({2+sqrt(3)/3},1) node [midway] {\large\textcolor{black}{$10$}};
    \draw[cyan!90,thin] (6.5,0.42) -- ({6.5+2*sqrt(3)/3},.42) node [midway] {\large\textcolor{black}{$0$}};
    \draw[cyan!90,thin] (8,.42) -- ({8+2*sqrt(3)/3},.42) node [midway] {\large\textcolor{black}{$1$}};
    \draw[cyan!90,thin] (9.5,.42) -- ({9.5+2*sqrt(3)/3},.42) node [midway] {\large\textcolor{black}{$10$}};
    \draw[cyan!90,thin] (3.5,1) -- ({3.5+sqrt(3)/3},0) node [midway] {\large\textcolor{black}{$0$}};
    \draw[cyan!90,thin] (4.5,1) -- ({4.5+sqrt(3)/3},0) node [midway] {\large\textcolor{black}{$1$}};
    \draw[cyan!90,thin] (5.5,1) -- ({5.5+sqrt(3)/3},0) node [midway] {\large\textcolor{black}{$10$}};
\end{tikzpicture}.
\end{center}
We write $\toprow$ for the inverse to $\partn$. We have following proposition.

\begin{prop}
\label{prop:puzzle-cryptomorphism}
For any boundary condition $\bs{\lambda}$, the transformation of edge labels shown above determines a bijection $\mc{P}_n(\bs{\lambda}) \to \tilde{\mc{P}}_n(\partn(\bs{\lambda}))$.
\end{prop}
\begin{proof}
    A simple check shows that the pieces for $1/2/3$-puzzles are sent to the pieces for $0/1/10$-puzzles.
\end{proof}

We write $\tilde{P}$ for the $0/1/10$-puzzle corresponding to a $1/2/3$-puzzle $P$.

\subsection{The splitting lemma} 
\label{sec:splitting}
In this section we prove an important lemma which will allow us to reduce to the case $m=3$.

\begin{defin}
\label{def:dual-array}
Let $\mu \in \{a,b\}^n$ be a string on two letters. The \emph{dual string} $\mu^\dagger$ is obtained by reversing $\mu$ and swapping $a$'s with $b$'s. If $T^{(i)}$ is a triangle such that $T^{(i)}_{j,k} \in \{a,b\}$ for all $1\leq j \leq k \leq n$, the \emph{dual triangle} $(T^{(i)})^{\dagger}$ is obtained by dualizing each row of $T^{(i)}$.
\end{defin}

\begin{defin}
\label{def:split-map}
Fix $m\geq3$ and $\d=(0=d_0 \leq d_1\leq d_2\leq\cdots\leq d_{m-1}\leq d_m=n)$ and for $i=1,\ldots,m$ fix Schubert strings $\lambda^{(i)}$ of type
\[
(m-i)^{d_{m-i}} \cdot m^{d_{m-i+1}-d_{m-i}} \cdot (m-i+1)^{n-d_{m-i+1}}.
\]
We define a map $\splt$ sending each $T \in \T_{m,n}(\lambda^{(1)},\ldots,\lambda^{(m)})$ to a pair $\splt(T)=(R,S)$ of arrays, as follows. For each $k=1,\ldots,n$, let $a_k$ be the unique element of the difference of multisets
\[
\left\{(m-2)^k,(m-1)^k,m^k\right\} \setminus \left(T^{(1)}_{\bullet,k} \cup T^{(2)}_{\bullet,k} \cup R^{(3)}_{\bullet,k-1}\right),
\]
where $R^{(3)}_{\bullet,k}$ is defined inductively by $R^{(3)}_{j,k}=R^{(3)}_{j-1,k-1}$ for $j=2,\ldots,k$ and $R^{(3)}_{1,k}=a_k$, with the base case $R^{(3)}_{\bullet,0}=\emptyset$. Then define $R$ and $S$ by
\begin{align*}
    R &= (T^{(1)},T^{(2)},R^{(3)}), \\
    S &= ((R^{(3)})^{\dagger}, T^{(3)}, \ldots, T^{(m)}).
\end{align*}
\end{defin}

\begin{ex}
\label{ex:split}
We give an example of the map $\splt$. Let $m=4$, $n=5$ and $\alpha(\bs{d})=(1,2,0,2)$. Take $\lambda^{(1)}=33443$, $\lambda^{(2)}=22323$, $\lambda^{(3)}=12424$, and $\lambda^{(4)}=11411$. Let $T\in\mc{T}_{4,5}(\lambda^{(1)},\ldots,\lambda^{(4)})$ be as shown below (top). Then $\splt(T)=(R,S) \in \T_{3,5} \times \T_{3,5}$ is shown on bottom. Note that $T=(R^{(1)},R^{(2)},S^{(2)},S^{(3)})$ and that $(R^{(3)})^\dagger = S^{(1)}$.
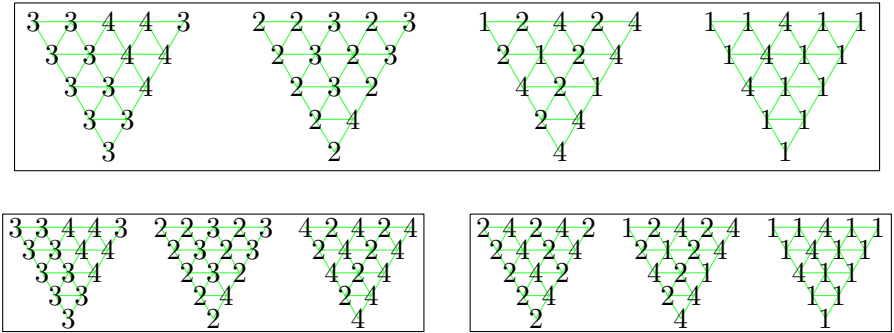
\begin{center}
\begin{tikzpicture}[scale=0.500000000000000]
\def\sep{2};
\draw[green!80,thin](0.000000000000000,0.000000000000000)--(4.00000000000000,0.000000000000000);
\draw[green!80,thin](4,0)--(2.00000000000000,-3.46400000000000);
\draw[green!80,thin](0,0)--(2.00000000000000,-3.46400000000000);
\draw[green!80,thin](0.500000000000000,-0.866000000000000)--(3.50000000000000,-0.866000000000000);
\draw[green!80,thin](3,0)--(1.50000000000000,-2.59800000000000);
\draw[green!80,thin](1,0)--(2.50000000000000,-2.59800000000000);
\draw[green!80,thin](1.00000000000000,-1.73200000000000)--(3.00000000000000,-1.73200000000000);
\draw[green!80,thin](2,0)--(1.00000000000000,-1.73200000000000);
\draw[green!80,thin](2,0)--(3.00000000000000,-1.73200000000000);
\draw[green!80,thin](1.50000000000000,-2.59800000000000)--(2.50000000000000,-2.59800000000000);
\draw[green!80,thin](1,0)--(0.500000000000000,-0.866000000000000);
\draw[green!80,thin](3,0)--(3.50000000000000,-0.866000000000000);
\draw[green!80,thin](2.00000000000000,-3.46400000000000)--(2.00000000000000,-3.46400000000000);
\draw[green!80,thin](0,0)--(0.000000000000000,0.000000000000000);
\draw[green!80,thin](4,0)--(4.00000000000000,0.000000000000000);
\draw[green!80,thin](6.00000000000000,0.000000000000000)--(10.0000000000000,0.000000000000000);
\draw[green!80,thin](10,0)--(8.00000000000000,-3.46400000000000);
\draw[green!80,thin](6,0)--(8.00000000000000,-3.46400000000000);
\draw[green!80,thin](6.50000000000000,-0.866000000000000)--(9.50000000000000,-0.866000000000000);
\draw[green!80,thin](9,0)--(7.50000000000000,-2.59800000000000);
\draw[green!80,thin](7,0)--(8.50000000000000,-2.59800000000000);
\draw[green!80,thin](7.00000000000000,-1.73200000000000)--(9.00000000000000,-1.73200000000000);
\draw[green!80,thin](8,0)--(7.00000000000000,-1.73200000000000);
\draw[green!80,thin](8,0)--(9.00000000000000,-1.73200000000000);
\draw[green!80,thin](7.50000000000000,-2.59800000000000)--(8.50000000000000,-2.59800000000000);
\draw[green!80,thin](7,0)--(6.50000000000000,-0.866000000000000);
\draw[green!80,thin](9,0)--(9.50000000000000,-0.866000000000000);
\draw[green!80,thin](8.00000000000000,-3.46400000000000)--(8.00000000000000,-3.46400000000000);
\draw[green!80,thin](6,0)--(6.00000000000000,0.000000000000000);
\draw[green!80,thin](10,0)--(10.0000000000000,0.000000000000000);
\draw[green!80,thin](12.0000000000000,0.000000000000000)--(16.0000000000000,0.000000000000000);
\draw[green!80,thin](16,0)--(14.0000000000000,-3.46400000000000);
\draw[green!80,thin](12,0)--(14.0000000000000,-3.46400000000000);
\draw[green!80,thin](12.5000000000000,-0.866000000000000)--(15.5000000000000,-0.866000000000000);
\draw[green!80,thin](15,0)--(13.5000000000000,-2.59800000000000);
\draw[green!80,thin](13,0)--(14.5000000000000,-2.59800000000000);
\draw[green!80,thin](13.0000000000000,-1.73200000000000)--(15.0000000000000,-1.73200000000000);
\draw[green!80,thin](14,0)--(13.0000000000000,-1.73200000000000);
\draw[green!80,thin](14,0)--(15.0000000000000,-1.73200000000000);
\draw[green!80,thin](13.5000000000000,-2.59800000000000)--(14.5000000000000,-2.59800000000000);
\draw[green!80,thin](13,0)--(12.5000000000000,-0.866000000000000);
\draw[green!80,thin](15,0)--(15.5000000000000,-0.866000000000000);
\draw[green!80,thin](14.0000000000000,-3.46400000000000)--(14.0000000000000,-3.46400000000000);
\draw[green!80,thin](12,0)--(12.0000000000000,0.000000000000000);
\draw[green!80,thin](16,0)--(16.0000000000000,0.000000000000000);
\draw[green!80,thin](18.0000000000000,0.000000000000000)--(22.0000000000000,0.000000000000000);
\draw[green!80,thin](22,0)--(20.0000000000000,-3.46400000000000);
\draw[green!80,thin](18,0)--(20.0000000000000,-3.46400000000000);
\draw[green!80,thin](18.5000000000000,-0.866000000000000)--(21.5000000000000,-0.866000000000000);
\draw[green!80,thin](21,0)--(19.5000000000000,-2.59800000000000);
\draw[green!80,thin](19,0)--(20.5000000000000,-2.59800000000000);
\draw[green!80,thin](19.0000000000000,-1.73200000000000)--(21.0000000000000,-1.73200000000000);
\draw[green!80,thin](20,0)--(19.0000000000000,-1.73200000000000);
\draw[green!80,thin](20,0)--(21.0000000000000,-1.73200000000000);
\draw[green!80,thin](19.5000000000000,-2.59800000000000)--(20.5000000000000,-2.59800000000000);
\draw[green!80,thin](19,0)--(18.5000000000000,-0.866000000000000);
\draw[green!80,thin](21,0)--(21.5000000000000,-0.866000000000000);
\draw[green!80,thin](20.0000000000000,-3.46400000000000)--(20.0000000000000,-3.46400000000000);
\draw[green!80,thin](18,0)--(18.0000000000000,0.000000000000000);
\draw[green!80,thin](22,0)--(22.0000000000000,0.000000000000000);
\node at (0.000000000000000,-0.000000000000000) {$3$};
\node at (1.00000000000000,-0.000000000000000) {$3$};
\node at (2.00000000000000,-0.000000000000000) {$4$};
\node at (3.00000000000000,-0.000000000000000) {$4$};
\node at (4.00000000000000,-0.000000000000000) {$3$};
\node at (0.500000000000000,-0.866000000000000) {$3$};
\node at (1.50000000000000,-0.866000000000000) {$3$};
\node at (2.50000000000000,-0.866000000000000) {$4$};
\node at (3.50000000000000,-0.866000000000000) {$4$};
\node at (1.00000000000000,-1.73200000000000) {$3$};
\node at (2.00000000000000,-1.73200000000000) {$3$};
\node at (3.00000000000000,-1.73200000000000) {$4$};
\node at (1.50000000000000,-2.59800000000000) {$3$};
\node at (2.50000000000000,-2.59800000000000) {$3$};
\node at (2.00000000000000,-3.46400000000000) {$3$};
\node at (6.00000000000000,-0.000000000000000) {$2$};
\node at (7.00000000000000,-0.000000000000000) {$2$};
\node at (8.00000000000000,-0.000000000000000) {$3$};
\node at (9.00000000000000,-0.000000000000000) {$2$};
\node at (10.0000000000000,-0.000000000000000) {$3$};
\node at (6.50000000000000,-0.866000000000000) {$2$};
\node at (7.50000000000000,-0.866000000000000) {$3$};
\node at (8.50000000000000,-0.866000000000000) {$2$};
\node at (9.50000000000000,-0.866000000000000) {$3$};
\node at (7.00000000000000,-1.73200000000000) {$2$};
\node at (8.00000000000000,-1.73200000000000) {$3$};
\node at (9.00000000000000,-1.73200000000000) {$2$};
\node at (7.50000000000000,-2.59800000000000) {$2$};
\node at (8.50000000000000,-2.59800000000000) {$4$};
\node at (8.00000000000000,-3.46400000000000) {$2$};
\node at (12.0000000000000,-0.000000000000000) {$1$};
\node at (13.0000000000000,-0.000000000000000) {$2$};
\node at (14.0000000000000,-0.000000000000000) {$4$};
\node at (15.0000000000000,-0.000000000000000) {$2$};
\node at (16.0000000000000,-0.000000000000000) {$4$};
\node at (12.5000000000000,-0.866000000000000) {$2$};
\node at (13.5000000000000,-0.866000000000000) {$1$};
\node at (14.5000000000000,-0.866000000000000) {$2$};
\node at (15.5000000000000,-0.866000000000000) {$4$};
\node at (13.0000000000000,-1.73200000000000) {$4$};
\node at (14.0000000000000,-1.73200000000000) {$2$};
\node at (15.0000000000000,-1.73200000000000) {$1$};
\node at (13.5000000000000,-2.59800000000000) {$2$};
\node at (14.5000000000000,-2.59800000000000) {$4$};
\node at (14.0000000000000,-3.46400000000000) {$4$};
\node at (18.0000000000000,-0.000000000000000) {$1$};
\node at (19.0000000000000,-0.000000000000000) {$1$};
\node at (20.0000000000000,-0.000000000000000) {$4$};
\node at (21.0000000000000,-0.000000000000000) {$1$};
\node at (22.0000000000000,-0.000000000000000) {$1$};
\node at (18.5000000000000,-0.866000000000000) {$1$};
\node at (19.5000000000000,-0.866000000000000) {$4$};
\node at (20.5000000000000,-0.866000000000000) {$1$};
\node at (21.5000000000000,-0.866000000000000) {$1$};
\node at (19.0000000000000,-1.73200000000000) {$4$};
\node at (20.0000000000000,-1.73200000000000) {$1$};
\node at (21.0000000000000,-1.73200000000000) {$1$};
\node at (19.5000000000000,-2.59800000000000) {$1$};
\node at (20.5000000000000,-2.59800000000000) {$1$};
\node at (20.0000000000000,-3.46400000000000) {$1$};
\draw[black,thin](-0.5,0.5)--(22.5,0.5)--(22.5,-3.964)--(-0.5,-3.964)--(-0.5,0.5);
\end{tikzpicture}

\

\begin{tikzpicture}[scale=0.350000000000000]
\def\sep{1.50000000000000};
\draw[green!80,thin](0.000000000000000,0.000000000000000)--(4.00000000000000,0.000000000000000);
\draw[green!80,thin](4.00000000000000,0)--(2.00000000000000,-3.46400000000000);
\draw[green!80,thin](0.000000000000000,0)--(2.00000000000000,-3.46400000000000);
\draw[green!80,thin](0.500000000000000,-0.866000000000000)--(3.50000000000000,-0.866000000000000);
\draw[green!80,thin](3.00000000000000,0)--(1.50000000000000,-2.59800000000000);
\draw[green!80,thin](1.00000000000000,0)--(2.50000000000000,-2.59800000000000);
\draw[green!80,thin](1.00000000000000,-1.73200000000000)--(3.00000000000000,-1.73200000000000);
\draw[green!80,thin](2.00000000000000,0)--(1.00000000000000,-1.73200000000000);
\draw[green!80,thin](2.00000000000000,0)--(3.00000000000000,-1.73200000000000);
\draw[green!80,thin](1.50000000000000,-2.59800000000000)--(2.50000000000000,-2.59800000000000);
\draw[green!80,thin](1.00000000000000,0)--(0.500000000000000,-0.866000000000000);
\draw[green!80,thin](3.00000000000000,0)--(3.50000000000000,-0.866000000000000);
\draw[green!80,thin](2.00000000000000,-3.46400000000000)--(2.00000000000000,-3.46400000000000);
\draw[green!80,thin](0.000000000000000,0)--(0.000000000000000,0.000000000000000);
\draw[green!80,thin](4.00000000000000,0)--(4.00000000000000,0.000000000000000);
\draw[green!80,thin](5.50000000000000,0.000000000000000)--(9.50000000000000,0.000000000000000);
\draw[green!80,thin](9.50000000000000,0)--(7.50000000000000,-3.46400000000000);
\draw[green!80,thin](5.50000000000000,0)--(7.50000000000000,-3.46400000000000);
\draw[green!80,thin](6.00000000000000,-0.866000000000000)--(9.00000000000000,-0.866000000000000);
\draw[green!80,thin](8.50000000000000,0)--(7.00000000000000,-2.59800000000000);
\draw[green!80,thin](6.50000000000000,0)--(8.00000000000000,-2.59800000000000);
\draw[green!80,thin](6.50000000000000,-1.73200000000000)--(8.50000000000000,-1.73200000000000);
\draw[green!80,thin](7.50000000000000,0)--(6.50000000000000,-1.73200000000000);
\draw[green!80,thin](7.50000000000000,0)--(8.50000000000000,-1.73200000000000);
\draw[green!80,thin](7.00000000000000,-2.59800000000000)--(8.00000000000000,-2.59800000000000);
\draw[green!80,thin](6.50000000000000,0)--(6.00000000000000,-0.866000000000000);
\draw[green!80,thin](8.50000000000000,0)--(9.00000000000000,-0.866000000000000);
\draw[green!80,thin](7.50000000000000,-3.46400000000000)--(7.50000000000000,-3.46400000000000);
\draw[green!80,thin](5.50000000000000,0)--(5.50000000000000,0.000000000000000);
\draw[green!80,thin](9.50000000000000,0)--(9.50000000000000,0.000000000000000);
\draw[green!80,thin](11.0000000000000,0.000000000000000)--(15.0000000000000,0.000000000000000);
\draw[green!80,thin](15.0000000000000,0)--(13.0000000000000,-3.46400000000000);
\draw[green!80,thin](11.0000000000000,0)--(13.0000000000000,-3.46400000000000);
\draw[green!80,thin](11.5000000000000,-0.866000000000000)--(14.5000000000000,-0.866000000000000);
\draw[green!80,thin](14.0000000000000,0)--(12.5000000000000,-2.59800000000000);
\draw[green!80,thin](12.0000000000000,0)--(13.5000000000000,-2.59800000000000);
\draw[green!80,thin](12.0000000000000,-1.73200000000000)--(14.0000000000000,-1.73200000000000);
\draw[green!80,thin](13.0000000000000,0)--(12.0000000000000,-1.73200000000000);
\draw[green!80,thin](13.0000000000000,0)--(14.0000000000000,-1.73200000000000);
\draw[green!80,thin](12.5000000000000,-2.59800000000000)--(13.5000000000000,-2.59800000000000);
\draw[green!80,thin](12.0000000000000,0)--(11.5000000000000,-0.866000000000000);
\draw[green!80,thin](14.0000000000000,0)--(14.5000000000000,-0.866000000000000);
\draw[green!80,thin](13.0000000000000,-3.46400000000000)--(13.0000000000000,-3.46400000000000);
\draw[green!80,thin](11.0000000000000,0)--(11.0000000000000,0.000000000000000);
\draw[green!80,thin](15.0000000000000,0)--(15.0000000000000,0.000000000000000);
\node at (0.000000000000000,-0.000000000000000) {$3$};
\node at (1.00000000000000,-0.000000000000000) {$3$};
\node at (2.00000000000000,-0.000000000000000) {$4$};
\node at (3.00000000000000,-0.000000000000000) {$4$};
\node at (4.00000000000000,-0.000000000000000) {$3$};
\node at (0.500000000000000,-0.866000000000000) {$3$};
\node at (1.50000000000000,-0.866000000000000) {$3$};
\node at (2.50000000000000,-0.866000000000000) {$4$};
\node at (3.50000000000000,-0.866000000000000) {$4$};
\node at (1.00000000000000,-1.73200000000000) {$3$};
\node at (2.00000000000000,-1.73200000000000) {$3$};
\node at (3.00000000000000,-1.73200000000000) {$4$};
\node at (1.50000000000000,-2.59800000000000) {$3$};
\node at (2.50000000000000,-2.59800000000000) {$3$};
\node at (2.00000000000000,-3.46400000000000) {$3$};
\node at (5.50000000000000,-0.000000000000000) {$2$};
\node at (6.50000000000000,-0.000000000000000) {$2$};
\node at (7.50000000000000,-0.000000000000000) {$3$};
\node at (8.50000000000000,-0.000000000000000) {$2$};
\node at (9.50000000000000,-0.000000000000000) {$3$};
\node at (6.00000000000000,-0.866000000000000) {$2$};
\node at (7.00000000000000,-0.866000000000000) {$3$};
\node at (8.00000000000000,-0.866000000000000) {$2$};
\node at (9.00000000000000,-0.866000000000000) {$3$};
\node at (6.50000000000000,-1.73200000000000) {$2$};
\node at (7.50000000000000,-1.73200000000000) {$3$};
\node at (8.50000000000000,-1.73200000000000) {$2$};
\node at (7.00000000000000,-2.59800000000000) {$2$};
\node at (8.00000000000000,-2.59800000000000) {$4$};
\node at (7.50000000000000,-3.46400000000000) {$2$};
\node at (11.0000000000000,-0.000000000000000) {$4$};
\node at (12.0000000000000,-0.000000000000000) {$2$};
\node at (13.0000000000000,-0.000000000000000) {$4$};
\node at (14.0000000000000,-0.000000000000000) {$2$};
\node at (15.0000000000000,-0.000000000000000) {$4$};
\node at (11.5000000000000,-0.866000000000000) {$2$};
\node at (12.5000000000000,-0.866000000000000) {$4$};
\node at (13.5000000000000,-0.866000000000000) {$2$};
\node at (14.5000000000000,-0.866000000000000) {$4$};
\node at (12.0000000000000,-1.73200000000000) {$4$};
\node at (13.0000000000000,-1.73200000000000) {$2$};
\node at (14.0000000000000,-1.73200000000000) {$4$};
\node at (12.5000000000000,-2.59800000000000) {$2$};
\node at (13.5000000000000,-2.59800000000000) {$4$};
\node at (13.0000000000000,-3.46400000000000) {$4$};
\draw[black,thin] (-0.5,0.5)--(15.5,0.5)--(15.5,-3.964)--(-0.5,-3.964)--(-0.5,0.5);
\end{tikzpicture}
\quad
\begin{tikzpicture}[scale=0.350000000000000]
\def\sep{1.50000000000000};
\draw[green!80,thin](0.000000000000000,0.000000000000000)--(4.00000000000000,0.000000000000000);
\draw[green!80,thin](4.00000000000000,0)--(2.00000000000000,-3.46400000000000);
\draw[green!80,thin](0.000000000000000,0)--(2.00000000000000,-3.46400000000000);
\draw[green!80,thin](0.500000000000000,-0.866000000000000)--(3.50000000000000,-0.866000000000000);
\draw[green!80,thin](3.00000000000000,0)--(1.50000000000000,-2.59800000000000);
\draw[green!80,thin](1.00000000000000,0)--(2.50000000000000,-2.59800000000000);
\draw[green!80,thin](1.00000000000000,-1.73200000000000)--(3.00000000000000,-1.73200000000000);
\draw[green!80,thin](2.00000000000000,0)--(1.00000000000000,-1.73200000000000);
\draw[green!80,thin](2.00000000000000,0)--(3.00000000000000,-1.73200000000000);
\draw[green!80,thin](1.50000000000000,-2.59800000000000)--(2.50000000000000,-2.59800000000000);
\draw[green!80,thin](1.00000000000000,0)--(0.500000000000000,-0.866000000000000);
\draw[green!80,thin](3.00000000000000,0)--(3.50000000000000,-0.866000000000000);
\draw[green!80,thin](2.00000000000000,-3.46400000000000)--(2.00000000000000,-3.46400000000000);
\draw[green!80,thin](0.000000000000000,0)--(0.000000000000000,0.000000000000000);
\draw[green!80,thin](4.00000000000000,0)--(4.00000000000000,0.000000000000000);
\draw[green!80,thin](5.50000000000000,0.000000000000000)--(9.50000000000000,0.000000000000000);
\draw[green!80,thin](9.50000000000000,0)--(7.50000000000000,-3.46400000000000);
\draw[green!80,thin](5.50000000000000,0)--(7.50000000000000,-3.46400000000000);
\draw[green!80,thin](6.00000000000000,-0.866000000000000)--(9.00000000000000,-0.866000000000000);
\draw[green!80,thin](8.50000000000000,0)--(7.00000000000000,-2.59800000000000);
\draw[green!80,thin](6.50000000000000,0)--(8.00000000000000,-2.59800000000000);
\draw[green!80,thin](6.50000000000000,-1.73200000000000)--(8.50000000000000,-1.73200000000000);
\draw[green!80,thin](7.50000000000000,0)--(6.50000000000000,-1.73200000000000);
\draw[green!80,thin](7.50000000000000,0)--(8.50000000000000,-1.73200000000000);
\draw[green!80,thin](7.00000000000000,-2.59800000000000)--(8.00000000000000,-2.59800000000000);
\draw[green!80,thin](6.50000000000000,0)--(6.00000000000000,-0.866000000000000);
\draw[green!80,thin](8.50000000000000,0)--(9.00000000000000,-0.866000000000000);
\draw[green!80,thin](7.50000000000000,-3.46400000000000)--(7.50000000000000,-3.46400000000000);
\draw[green!80,thin](5.50000000000000,0)--(5.50000000000000,0.000000000000000);
\draw[green!80,thin](9.50000000000000,0)--(9.50000000000000,0.000000000000000);
\draw[green!80,thin](11.0000000000000,0.000000000000000)--(15.0000000000000,0.000000000000000);
\draw[green!80,thin](15.0000000000000,0)--(13.0000000000000,-3.46400000000000);
\draw[green!80,thin](11.0000000000000,0)--(13.0000000000000,-3.46400000000000);
\draw[green!80,thin](11.5000000000000,-0.866000000000000)--(14.5000000000000,-0.866000000000000);
\draw[green!80,thin](14.0000000000000,0)--(12.5000000000000,-2.59800000000000);
\draw[green!80,thin](12.0000000000000,0)--(13.5000000000000,-2.59800000000000);
\draw[green!80,thin](12.0000000000000,-1.73200000000000)--(14.0000000000000,-1.73200000000000);
\draw[green!80,thin](13.0000000000000,0)--(12.0000000000000,-1.73200000000000);
\draw[green!80,thin](13.0000000000000,0)--(14.0000000000000,-1.73200000000000);
\draw[green!80,thin](12.5000000000000,-2.59800000000000)--(13.5000000000000,-2.59800000000000);
\draw[green!80,thin](12.0000000000000,0)--(11.5000000000000,-0.866000000000000);
\draw[green!80,thin](14.0000000000000,0)--(14.5000000000000,-0.866000000000000);
\draw[green!80,thin](13.0000000000000,-3.46400000000000)--(13.0000000000000,-3.46400000000000);
\draw[green!80,thin](11.0000000000000,0)--(11.0000000000000,0.000000000000000);
\draw[green!80,thin](15.0000000000000,0)--(15.0000000000000,0.000000000000000);
\node at (0.000000000000000,-0.000000000000000) {$2$};
\node at (1.00000000000000,-0.000000000000000) {$4$};
\node at (2.00000000000000,-0.000000000000000) {$2$};
\node at (3.00000000000000,-0.000000000000000) {$4$};
\node at (4.00000000000000,-0.000000000000000) {$2$};
\node at (0.500000000000000,-0.866000000000000) {$2$};
\node at (1.50000000000000,-0.866000000000000) {$4$};
\node at (2.50000000000000,-0.866000000000000) {$2$};
\node at (3.50000000000000,-0.866000000000000) {$4$};
\node at (1.00000000000000,-1.73200000000000) {$2$};
\node at (2.00000000000000,-1.73200000000000) {$4$};
\node at (3.00000000000000,-1.73200000000000) {$2$};
\node at (1.50000000000000,-2.59800000000000) {$2$};
\node at (2.50000000000000,-2.59800000000000) {$4$};
\node at (2.00000000000000,-3.46400000000000) {$2$};
\node at (5.50000000000000,-0.000000000000000) {$1$};
\node at (6.50000000000000,-0.000000000000000) {$2$};
\node at (7.50000000000000,-0.000000000000000) {$4$};
\node at (8.50000000000000,-0.000000000000000) {$2$};
\node at (9.50000000000000,-0.000000000000000) {$4$};
\node at (6.00000000000000,-0.866000000000000) {$2$};
\node at (7.00000000000000,-0.866000000000000) {$1$};
\node at (8.00000000000000,-0.866000000000000) {$2$};
\node at (9.00000000000000,-0.866000000000000) {$4$};
\node at (6.50000000000000,-1.73200000000000) {$4$};
\node at (7.50000000000000,-1.73200000000000) {$2$};
\node at (8.50000000000000,-1.73200000000000) {$1$};
\node at (7.00000000000000,-2.59800000000000) {$2$};
\node at (8.00000000000000,-2.59800000000000) {$4$};
\node at (7.50000000000000,-3.46400000000000) {$4$};
\node at (11.0000000000000,-0.000000000000000) {$1$};
\node at (12.0000000000000,-0.000000000000000) {$1$};
\node at (13.0000000000000,-0.000000000000000) {$4$};
\node at (14.0000000000000,-0.000000000000000) {$1$};
\node at (15.0000000000000,-0.000000000000000) {$1$};
\node at (11.5000000000000,-0.866000000000000) {$1$};
\node at (12.5000000000000,-0.866000000000000) {$4$};
\node at (13.5000000000000,-0.866000000000000) {$1$};
\node at (14.5000000000000,-0.866000000000000) {$1$};
\node at (12.0000000000000,-1.73200000000000) {$4$};
\node at (13.0000000000000,-1.73200000000000) {$1$};
\node at (14.0000000000000,-1.73200000000000) {$1$};
\node at (12.5000000000000,-2.59800000000000) {$1$};
\node at (13.5000000000000,-2.59800000000000) {$1$};
\node at (13.0000000000000,-3.46400000000000) {$1$};
\draw[black,thin] (-0.5,0.5)--(15.5,0.5)--(15.5,-3.964)--(-0.5,-3.964)--(-0.5,0.5);
\end{tikzpicture}
\end{center}
\end{ex}


Given any totally ordered set $\Sigma$ of cardinality $m$, we write $\T_{m,n}^{\Sigma}(\bs{\lambda})$ for the set of interlacing triangular arrays of rank $m$, height $n$, and top row $\bs{\lambda}$, with entries from $\Sigma$ (instead of from $[m]$); clearly the choice of $\Sigma$ does not affect the cardinality of this set. We can now state the splitting lemma.

\begin{lemma}\label{lem:splitting-interlacing}
Fix $m\geq3$ and $\d=(0=d_0 \leq d_1\leq d_2\leq\cdots\leq d_{m-1}\leq d_m=n)$ and for $i=1,\ldots,m$, fix Schubert strings $\lambda^{(i)}$ of type $(m-i)^{d_{m-i}}m^{d_{m-i+1}-d_{m-i}}(m-i+1)^{n-d_{m-i+1}}$. Then $\splt$ is a bijection from $\T_{m,n}(\lambda^{(1)},\ldots,\lambda^{(m)})$ to
\begin{equation}
\label{eq:splitting-union}
\bigsqcup_{\mu}\T_{3,n}^{\{m-2,m-1,m\}}(\lambda^{(1)},\lambda^{(2)},\mu)\times\T_{m-1,n}^{[m]\setminus \{m-1\}}(\mu^{\dagger},\lambda^{(3)},\ldots,\lambda^{(m)}),
\end{equation}
where the union runs over $\mu$ of type $(m-2)^{n-d_{m-2}}m^{d_{m-2}}$.
\end{lemma}

\begin{proof}

Let $T \in \T_{m,n}(\lambda^{(1)},\ldots,\lambda^{(m)})$ and $(R,S)=\splt(T)$. We first argue that $R$ and $S$ are interlacing triangular arrays. 

By our assumption on $\bs{\lambda}$, only the numbers $m,m-1,m-2$ appear in $T^{(1)}$ and $T^{(2)}$. Moreover, $m-1$ appears $k$ times in $T^{(1)}_{\bullet,k} \cup T^{(2)}_{\bullet,k}$ for $k=1,\ldots,n$ since it does not appear in any $T^{(i)}$ with $i>2$. Thus it is easy to see inductively that the number $a_k$ from \Cref{def:split-map} is well-defined, that all entries $R^{(3)}_{j,k}$ lie in $\{m-2,m\}$ and that $T^{(1)}_{\bullet,k} \cup T^{(2)}_{\bullet,k} \cup R^{(3)}_{\bullet,k} = \{(m-2)^k,(m-1)^k,m^k\}$ for $k=1,\ldots,n$. 

It remains to check the interlacing condition for $R$. The $(m-1)$'s interlace, because they all appear in $T^{(1)}$ and $T^{(2)}$ and because $T$ is interlacing by hypothesis. Consider the $m$'s. By the inductive definition of $R^{(3)}$, if $R^{(3)}_{j,k}=m$ for $j\neq 1$, then $R^{(3)}_{j-1,k-1}=m$. Thus between any pair of $m$'s in $R^{(3)}_{\bullet,k}$ there is an interlacing $m$ in $R^{(3)}_{\bullet,k-1}$, and so the $m$'s are interlacing within $R^{(3)}$. Furthermore, $a_k=m$ if and only if the number of $m$'s in $T^{(1)}_{\bullet,k} \cup T^{(2)}_{\bullet,k}$ is equal to the number in $T^{(1)}_{\bullet,k-1} \cup T^{(2)}_{\bullet,k-1}$. Thus the $m$'s alternate between rows $k-1$ and $k$ as we move the horizontal coordinate left to right, and so they interlace. The $(m-2)$'s are likewise interlacing. An analogous argument show that $S$ is an interlacing triangular array.

Now, define a map $\mathsf{merge}$ on
\[
\bigsqcup_{\mu}\T_{3,n}^{\{m-2,m-1,m\}}(\lambda^{(1)},\lambda^{(2)},\mu)\times\T_{m-1,n}^{[m]\setminus \{m-1\}}(\mu^{\dagger},\lambda^{(3)},\ldots,\lambda^{(m)})
\]
sending $(R,S) \mapsto (R^{(1)},R^{(2)},S^{(2)},S^{(3)},\ldots,S^{(m-1)}).$ It is easy to check that $\merge(R,S)$ lies in $\T_{m,n}(\lambda^{(1)},\ldots,\lambda^{(m)})$ and that $\merge$ is the inverse of $\splt$.
\end{proof}

\subsection{Avoiding puzzle pieces} The geometric interpretations of $0/1/10$-puzzles in terms of Schubert classes in cohomology, the $G_{\xi}$ basis in $K$-theory, and the $G^{\ast}_{\xi}$ basis in $K$-theory involve forbidding the $\Delta$- and $\nabla$-oriented $10$-$10$-$10$ pieces, or one or the other of these pieces. We call these two pieces, as well as their $1/2/3$-analogs, the \emph{$K$-pieces}. In this section, we describe the forbidden substructures in interlacing triangular arrays corresponding to forbidden one or both of the $K$-pieces.

\begin{prop}
\label{prop:avoiding-delta-10-10-10}
Suppose that $\lambda^{(1)} \in \{2,3\}^n, \lambda^{(2)} \in \{1,2,3\}^n,$ and $\lambda^{(3)} \in \{1,3\}^n$. Then the followings are equivalent:
\begin{itemize}[leftmargin=0.3in]
    \item[(a)] The $0/1/10$-puzzle $\tilde{P}$ with boundary conditions $\partn(\bs{\lambda})$ avoids the piece 
    \raisebox{-0.5\height}{
    \begin{tikzpicture}[scale=0.8, rotate=-60]
    \node[circle, fill=black, scale=0.1] (c0) at ({-2*sqrt(3)/3},2) {};
    \node[circle, fill=black, scale=0.1] (d0) at ({-3*sqrt(3)/3},3) {};
    \node[circle, fill=black, scale=0.1] (d1) at ({-1*sqrt(3)/3},3) {};
    \draw[cyan!90,thin] (c0) -- (d0) node [midway] {\textcolor{black}{$10$}};
    \draw[cyan!90,thin] (c0) -- (d1) node [midway] {\textcolor{black}{$10$}};
    \draw[cyan!90,thin] (d0) -- (d1) node [midway] {\textcolor{black}{$10$}};
\end{tikzpicture}}.
    \item[(b)] The $1/2/3$-puzzle $P$ with boundary conditions $\bs{\lambda}$ avoids the piece 
    \raisebox{-0.5\height}{
    \begin{tikzpicture}[scale=0.8, rotate=-60]
    \node[circle, fill=black, scale=0.1] (c0) at ({-2*sqrt(3)/3},2) {};
    \node[circle, fill=black, scale=0.1] (d0) at ({-3*sqrt(3)/3},3) {};
    \node[circle, fill=black, scale=0.1] (d1) at ({-1*sqrt(3)/3},3) {};
    \draw[blue!80,thin] (c0) -- (d0) node [midway] {\textcolor{black}{$1$}};
    \draw[blue!80,thin] (c0) -- (d1) node [midway] {\textcolor{black}{$2$}};
    \draw[blue!80,thin] (d0) -- (d1) node [midway] {\textcolor{black}{$3$}};
\end{tikzpicture}}.
    \item[(c)] The array $T=\mathscr{T}(P)$ with top row $\bs{\lambda}$ avoids \raisebox{-0.5\height}{
\begin{tikzpicture}[scale=0.8]
    \draw[green!80,thin] (0,0)--(-0.5,.866);
    \node at (-0.5,.866) {$3$};
    \node at (0,0) {$1$};
\end{tikzpicture}} and 
\raisebox{-0.5\height}{
\begin{tikzpicture}[scale=0.8]
    \draw[green!80,thin] (0,0)--(-0.5,.866)--(0.5,.866)--(0,0);
    \node at (-0.5,.866) {$3$};
    \node at (0.5,.866) {$2$};
    \node at (0,0) {$3$};
\end{tikzpicture}} in $T^{(2)}$.
  \item[(d)] The array $T=\mathscr{T}(P)$ with top row $\bs{\lambda}$ avoids \raisebox{-0.5\height}{
\begin{tikzpicture}[scale=0.8]
    \draw[green!80,thin] (0,0)--(0.5,.866);
    \node at (0.5,.866) {$3$};
    \node at (0,0) {$2$};
\end{tikzpicture}} and 
\raisebox{-0.5\height}{
\begin{tikzpicture}[scale=0.8]
    \draw[green!80,thin] (0,0)--(-0.5,.866)--(0.5,.866)--(0,0);
    \node at (-0.5,.866) {$1$};
    \node at (0.5,.866) {$3$};
    \node at (0,0) {$3$};
\end{tikzpicture}} in $T^{(2)}$.
\end{itemize}
\end{prop}
\begin{proof}
    Since the piece \raisebox{-0.2\height}{
    \begin{tikzpicture}[scale=0.7, rotate=-60]
    \node[circle, fill=black, scale=0.1] (c0) at ({-2*sqrt(3)/3},2) {};
    \node[circle, fill=black, scale=0.1] (d0) at ({-3*sqrt(3)/3},3) {};
    \node[circle, fill=black, scale=0.1] (d1) at ({-1*sqrt(3)/3},3) {};
    \draw[blue!80,thin] (c0) -- (d0) node [midway] {\textcolor{black}{$1$}};
    \draw[blue!80,thin] (c0) -- (d1) node [midway] {\textcolor{black}{$2$}};
    \draw[blue!80,thin] (d0) -- (d1) node [midway] {\textcolor{black}{$3$}};
\end{tikzpicture}} is sent to the piece \raisebox{-0.2\height}{
    \begin{tikzpicture}[scale=0.7, rotate=-60]
    \node[circle, fill=black, scale=0.1] (c0) at ({-2*sqrt(3)/3},2) {};
    \node[circle, fill=black, scale=0.1] (d0) at ({-3*sqrt(3)/3},3) {};
    \node[circle, fill=black, scale=0.1] (d1) at ({-1*sqrt(3)/3},3) {};
    \draw[cyan!90,thin] (c0) -- (d0) node [midway] {\textcolor{black}{$10$}};
    \draw[cyan!90,thin] (c0) -- (d1) node [midway] {\textcolor{black}{$10$}};
    \draw[cyan!90,thin] (d0) -- (d1) node [midway] {\textcolor{black}{$10$}};
\end{tikzpicture}} by the transformation of \Cref{prop:puzzle-cryptomorphism}, the equivalence of (a) and (b) is clear. We now prove the equivalence of (b) and (d).

We use the contrapositive. Suppose that $P$ contains the piece \raisebox{-0.5\height}{
    \begin{tikzpicture}[scale=0.8, rotate=-60]
    \node[circle, fill=black, scale=0.1] (c0) at ({-2*sqrt(3)/3},2) {};
    \node[circle, fill=black, scale=0.1] (d0) at ({-3*sqrt(3)/3},3) {};
    \node[circle, fill=black, scale=0.1] (d1) at ({-1*sqrt(3)/3},3) {};
    \draw[blue!80,thin] (c0) -- (d0) node [midway] {\textcolor{black}{$1$}};
    \draw[blue!80,thin] (c0) -- (d1) node [midway] {\textcolor{black}{$2$}};
    \draw[blue!80,thin] (d0) -- (d1) node [midway] {\textcolor{black}{$3$}};
\end{tikzpicture}} at position~$A$. Consider the southeast-to-northwest slice of $P$ containing the piece and the maximal sequence of consecutive $K$-pieces beginning at $A$ and continuing to the northwest within the slice (see diagrams below).

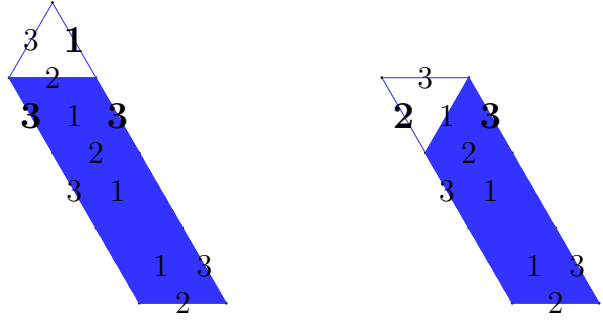
\begin{center}
\begin{tikzpicture}[scale=1,rotate=-60]
    \node[circle, fill=black, scale=0.1] (c0) at ({-2*sqrt(3)/3},2) {};
    \node[circle, fill=black, scale=0.1] (c1) at (0,2) {};
    \node[circle, fill=black, scale=0.1] (c2) at ({2*sqrt(3)/3},2) {};
    \node[circle, fill=black, scale=0.1] (c3) at ({4*sqrt(3)/3},2) {};
    \node[circle, fill=black, scale=0.1] (d0) at ({-3*sqrt(3)/3},3) {};
    \node[circle, fill=black, scale=0.1] (d1) at ({-1*sqrt(3)/3},3) {};
    \node[circle, fill=black, scale=0.1] (d2) at ({1*sqrt(3)/3},3) {};
    \node[circle, fill=black, scale=0.1] (d3) at ({3*sqrt(3)/3},3) {};
    \node (A) at ({.8+2.5*sqrt(3)/3},2.7) {\textcolor{blue}{$A$}};
    \node[circle, fill=black, scale=0.1] (d4) at ({5*sqrt(3)/3},3) {};
    \filldraw[blue!80,opacity=0.2] ({5*sqrt(3)/3},3)--({4*sqrt(3)/3},2)--({-2*sqrt(3)/3},2)--({-1*sqrt(3)/3},3)--({5*sqrt(3)/3},3);
    \draw[blue!80,thin] (c0) -- (c1) node [midway] {\Large\textcolor{black}{$\bs{3}$}};
    \draw[blue!80,thin] (c1) -- (c2) node [midway] {\large\textcolor{black}{$3$}};
    \draw[blue!80,thin] (c0) -- (d0) node [midway] {\large\textcolor{black}{$3$}};
    \draw[blue!80,thin] (c0) -- (d1) node [midway] {\large\textcolor{black}{$2$}};
    \draw[blue!80,thin] (c1) -- (d1) node [midway] {\large\textcolor{black}{$1$}};
    \draw[blue!80,thin] (c1) -- (d2) node [midway] {\large\textcolor{black}{$2$}};
    \draw[blue!80,thin] (c2) -- (d2) node [midway] {\large\textcolor{black}{$1$}};
    \draw[blue!80,thin] (d0) -- (d1) node [midway] {\Large\textcolor{black}{$\bs{1}$}};
    \draw[blue!80,thin] (d1) -- (d2) node [midway] {\Large\textcolor{black}{$\bs{3}$}};
    \draw[dashed,blue!80] (d2) -- (d3) node [midway] {};
    \draw[dashed,blue!80] (c2) -- (c3) node [midway] {};
    \draw[blue!80,thin] (c3) -- (d4) node [midway] {\large\textcolor{black}{$2$}};
    \draw[blue!80,thin] (c3) -- (d3) node [midway] {\large\textcolor{black}{$1$}};
    \draw[blue!80,thin] (d3) -- (d4) node [midway] {\large\textcolor{black}{$3$}};
\end{tikzpicture}\hspace{0.75in}
\begin{tikzpicture}[scale=1,rotate=-60]
    \node[circle, fill=black, scale=0.1] (c0) at ({-2*sqrt(3)/3},2) {};
    \node[circle, fill=black, scale=0.1] (c1) at (0,2) {};
    \node[circle, fill=black, scale=0.1] (c2) at ({2*sqrt(3)/3},2) {};
    \node[circle, fill=black, scale=0.1] (c3) at ({4*sqrt(3)/3},2) {};
    \node[circle, fill=black, scale=0.1] (d1) at ({-1*sqrt(3)/3},3) {};
    \node[circle, fill=black, scale=0.1] (d2) at ({1*sqrt(3)/3},3) {};
    \node[circle, fill=black, scale=0.1] (d3) at ({3*sqrt(3)/3},3) {};
    \node (A) at ({.8+2.5*sqrt(3)/3},2.7) {\textcolor{blue}{$A$}};
    \node[circle, fill=black, scale=0.1] (d4) at ({5*sqrt(3)/3},3) {};
    \filldraw[blue!80,opacity=0.2] ({5*sqrt(3)/3},3)--({4*sqrt(3)/3},2)--(0,2)--({-1*sqrt(3)/3},3)--({5*sqrt(3)/3},3);
    \draw[blue!80,thin] (c0) -- (c1) node [midway] {\Large\textcolor{black}{$\bs{2}$}};
    \draw[blue!80,thin] (c1) -- (c2) node [midway] {\large\textcolor{black}{$3$}};
    \draw[blue!80,thin] (c0) -- (d1) node [midway] {\large\textcolor{black}{$3$}};
    \draw[blue!80,thin] (c1) -- (d1) node [midway] {\large\textcolor{black}{$1$}};
    \draw[blue!80,thin] (c1) -- (d2) node [midway] {\large\textcolor{black}{$2$}};
    \draw[blue!80,thin] (c2) -- (d2) node [midway] {\large\textcolor{black}{$1$}};
    \draw[blue!80,thin] (d1) -- (d2) node [midway] {\Large\textcolor{black}{$\bs{3}$}};
    \draw[dashed,blue!80] (d2) -- (d3) node [midway] {};
    \draw[dashed,blue!80] (c2) -- (c3) node [midway] {};
    \draw[blue!80,thin] (c3) -- (d4) node [midway] {\large\textcolor{black}{$2$}};
    \draw[blue!80,thin] (c3) -- (d3) node [midway] {\large\textcolor{black}{$1$}};
    \draw[blue!80,thin] (d3) -- (d4) node [midway] {\large\textcolor{black}{$3$}};
\end{tikzpicture}
\end{center}
This sequence does not continue to the end of the slice because $\lambda^{(1)}$ does not contain $1$ by hypothesis. If the last $K$-piece in the sequence is $\nabla$-oriented, then the slice must be as above, left. Thus $T^{(2)}$ contains the bolded instance of \raisebox{-0.5\height}{
\begin{tikzpicture}[scale=0.8]
    \draw[green!80,thin] (0,0)--(-0.5,.866)--(0.5,.866)--(0,0);
    \node at (-0.5,.866) {$1$};
    \node at (0.5,.866) {$3$};
    \node at (0,0) {$3$};
\end{tikzpicture}}. If instead the last $K$-piece in the sequence is $\Delta$-oriented, then the slice must be as above, right. In this case $T^{(2)}$ contains the bolded instance of \raisebox{-0.5\height}{
\begin{tikzpicture}[scale=0.8]
    \draw[green!80,thin] (0,0)--(0.5,.866);
    \node at (0.5,.866) {$3$};
    \node at (0,0) {$2$};
\end{tikzpicture}}.

If, conversely, $T^{(2)}$ contains one of the patterns, then $P=\mathscr{T}'(T)$ must contain one of the substructures below:

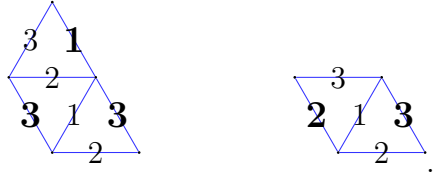
\begin{center}
\begin{tikzpicture}[scale=1,rotate=-60]
    \node[circle, fill=black, scale=0.1] (c0) at ({-2*sqrt(3)/3},2) {};
    \node[circle, fill=black, scale=0.1] (c1) at (0,2) {};
    \node[circle, fill=black, scale=0.1] (d0) at ({-3*sqrt(3)/3},3) {};
    \node[circle, fill=black, scale=0.1] (d1) at ({-1*sqrt(3)/3},3) {};
    \node[circle, fill=black, scale=0.1] (d2) at ({1*sqrt(3)/3},3) {};
    \draw[blue!80,thin] (c0) -- (c1) node [midway] {\Large\textcolor{black}{$\bs{3}$}};
    \draw[blue!80,thin] (c0) -- (d0) node [midway] {\large\textcolor{black}{$3$}};
    \draw[blue!80,thin] (c0) -- (d1) node [midway] {\large\textcolor{black}{$2$}};
    \draw[blue!80,thin] (c1) -- (d1) node [midway] {\large\textcolor{black}{$1$}};
    \draw[blue!80,thin] (c1) -- (d2) node [midway] {\large\textcolor{black}{$2$}};
    \draw[blue!80,thin] (d0) -- (d1) node [midway] {\Large\textcolor{black}{$\bs{1}$}};
    \draw[blue!80,thin] (d1) -- (d2) node [midway] {\Large\textcolor{black}{$\bs{3}$}};
\end{tikzpicture}\hspace{0.75in}
\begin{tikzpicture}[scale=1,rotate=-60]
    \node[circle, fill=black, scale=0.1] (c0) at ({-2*sqrt(3)/3},2) {};
    \node[circle, fill=black, scale=0.1] (c1) at (0,2) {};
    \node[circle, fill=black, scale=0.1] (d1) at ({-1*sqrt(3)/3},3) {};
    \node[circle, fill=black, scale=0.1] (d2) at ({1*sqrt(3)/3},3) {};
    \draw[blue!80,thin] (c0) -- (c1) node [midway] {\Large\textcolor{black}{$\bs{2}$}};
    \draw[blue!80,thin] (c0) -- (d1) node [midway] {\large\textcolor{black}{$3$}};
    \draw[blue!80,thin] (c1) -- (d1) node [midway] {\large\textcolor{black}{$1$}};
    \draw[blue!80,thin] (c1) -- (d2) node [midway] {\large\textcolor{black}{$2$}};
    \draw[blue!80,thin] (d1) -- (d2) node [midway] {\Large\textcolor{black}{$\bs{3}$}};
\end{tikzpicture}.
\end{center}
In either case, $P$ contains the desired piece.

The equivalence of (b) and (c) can be proven similarly, by considering the slice heading southeast (rather than northwest) from $A$.
\end{proof}

The following proposition is closely analogous to \Cref{prop:avoiding-delta-10-10-10} and its proof is omitted.
\begin{prop}
\label{prop:avoiding-nabla-10-10-10}
Suppose that $\lambda^{(1)} \in \{2,3\}^n, \lambda^{(2)} \in \{1,2,3\}^n,$ and $\lambda^{(3)} \in \{1,3\}^n$. Then the following are equivalent:
\begin{itemize}[leftmargin=0.3in]
    \item[(a)] The $0/1/10$-puzzle $\tilde{P}$ with boundary conditions $\partn(\bs{\lambda})$ avoids the piece 
    \raisebox{-0.5\height}{
    \begin{tikzpicture}[scale=0.8, rotate=-240]
    \node[circle, fill=black, scale=0.1] (c0) at ({-2*sqrt(3)/3},2) {};
    \node[circle, fill=black, scale=0.1] (d0) at ({-3*sqrt(3)/3},3) {};
    \node[circle, fill=black, scale=0.1] (d1) at ({-1*sqrt(3)/3},3) {};
    \draw[cyan!90,thin] (c0) -- (d0) node [midway] {\textcolor{black}{$10$}};
    \draw[cyan!90,thin] (c0) -- (d1) node [midway] {\textcolor{black}{$10$}};
    \draw[cyan!90,thin] (d0) -- (d1) node [midway] {\textcolor{black}{$10$}};
\end{tikzpicture}}.
    \item[(b)] The $1/2/3$-puzzle $P$ with boundary conditions $\bs{\lambda}$ avoids the piece 
    \raisebox{-0.5\height}{
    \begin{tikzpicture}[scale=0.8, rotate=-240]
    \node[circle, fill=black, scale=0.1] (c0) at ({-2*sqrt(3)/3},2) {};
    \node[circle, fill=black, scale=0.1] (d0) at ({-3*sqrt(3)/3},3) {};
    \node[circle, fill=black, scale=0.1] (d1) at ({-1*sqrt(3)/3},3) {};
    \draw[blue!80,thin] (c0) -- (d0) node [midway] {\textcolor{black}{$1$}};
    \draw[blue!80,thin] (c0) -- (d1) node [midway] {\textcolor{black}{$2$}};
    \draw[blue!80,thin] (d0) -- (d1) node [midway] {\textcolor{black}{$3$}};
\end{tikzpicture}}.
    \item[(c)] The array $T=\mathscr{T}(P)$ with top row $\bs{\lambda}$ avoids \raisebox{-0.5\height}{
\begin{tikzpicture}[scale=0.8,rotate=180]
    \draw[green!80,thin] (0,0)--(-0.5,.866);
    \node at (-0.5,.866) {$3$};
    \node at (0,0) {$1$};
\end{tikzpicture}} and 
\raisebox{-0.5\height}{
\begin{tikzpicture}[scale=0.8,rotate=180]
    \draw[green!80,thin] (0,0)--(-0.5,.866)--(0.5,.866)--(0,0);
    \node at (-0.5,.866) {$3$};
    \node at (0.5,.866) {$2$};
    \node at (0,0) {$3$};
\end{tikzpicture}} in $T^{(2)}$.
  \item[(d)] The array $T=\mathscr{T}(P)$ with top row $\bs{\lambda}$ avoids \raisebox{-0.5\height}{
\begin{tikzpicture}[scale=0.8,rotate=180]
    \draw[green!80,thin] (0,0)--(0.5,.866);
    \node at (0.5,.866) {$3$};
    \node at (0,0) {$2$};
\end{tikzpicture}} and 
\raisebox{-0.5\height}{
\begin{tikzpicture}[scale=0.8,rotate=180]
    \draw[green!80,thin] (0,0)--(-0.5,.866)--(0.5,.866)--(0,0);
    \node at (-0.5,.866) {$1$};
    \node at (0.5,.866) {$3$};
    \node at (0,0) {$3$};
\end{tikzpicture}} in $T^{(2)}$.
\end{itemize}
\end{prop}

\subsection{Proofs of geometric interpretations} We can now prove \Cref{thm:k-theory,thm:dual-k,thm:ssm,thm:2-step}.

We first verify the easy cases $m=1$ and $2$. The first is immediate:

\begin{prop}
\label{prop:m-equals-1-interlacing}
For all $n \geq 1$ there is a unique interlacing triangular array in $\T_{1,n}$.
\end{prop}
\begin{proof}
    Clearly we must have $T^{(i)}_{j,k}=1$ for all $i,j,k$, and this is indeed an interlacing triangular array.
\end{proof}

\Cref{prop:m-equals-1-interlacing} corresponds to the fact that, for any of the bases appearing in \Cref{thm:k-theory,thm:dual-k,thm:ssm,thm:2-step}, the basis element corresponding to the top row $1^n$ is the multiplicative identity element of the ring in which it resides.

\begin{prop}
\label{prop:m-equals-2-interlacing}
For each $\lambda \in \{1,2\}^n$, there is a unique $T \in \T_{2,n}$ such that $T^{(1)}$ has top row $\lambda$. Furthermore, $T$ satisfies $T^{(2)}=(T^{(1)})^{\dagger}$.
\end{prop}
\begin{proof}
    Suppose the result is true for arrays of height $n-1$. And suppose without loss of generality that $\lambda_1=1$. Then any such $T$ has $T^{(1)}_{1,k}=1$ for all $k=1,\ldots,n$ and therefore has $T^{(2)}_{k,k}=2$ for all $k$, by \Cref{lem:interlacing-side-triangles}. Now observe that the remainder $S=\{T^{(1)}_{j,k} \mid 2 \leq j \leq k\} \cup \{T^{(2)}_{j,k} j \leq k-1 \geq 1\}$ of the array is in fact an interlacing triangular array of rank $2$ and height $n-1$, with $S^{(1)}$ having top row $(\lambda_2, \ldots, \lambda_n)$. The result follows by induction.
\end{proof}

For $\lambda \in \{1,2\}^n$ we have $\partn((\lambda, \lambda^{\dagger}))=(\xi, \xi^{\perp})$ for some $\xi$. \Cref{thm:k-theory,thm:dual-k,thm:ssm,thm:2-step} hold in the case $m=2$ since $\xi=(\xi^{\perp})^{\perp}$ and so $B_{\xi}=1 \cdot B_{(\xi^{\perp})^{\perp}}$ for any of the bases appearing in the theorems, agreeing with \Cref{prop:m-equals-2-interlacing}.

We now turn to the proofs for general $m$.

\begin{proof}[Proof of \Cref{thm:k-theory}]
    Let $\xi^{(1)},\ldots, \xi^{(m)}$ have content $0^d1^{n-d}$, define $\bs{\lambda}=\toprow(\bs{\xi})$, and let $\mc{G}_{m,n}(\bs{\lambda})$ denote the set of interlacing triangular arrays appearing in the theorem statement. If $m=1$ or $2$, then the theorem holds by the discussion above. Suppose that $m=3$. Then by a result of Vakil \cite[Thm.~3.6]{Vakil}, $g_{\bs{\xi}}$ is $(-1)^{d(n-d)-|\bs{\xi}|}$ times the number of $0/1/10$-puzzles $\tilde{P}$ with boundary conditions $\bs{\xi}$ which avoid the $\nabla$-oriented $10$-$10$-$10$ piece. By \Cref{prop:avoiding-nabla-10-10-10}, the bijection $\mathscr{T}$ maps the associated $1/2/3$-puzzles $P$ to interlacing triangular arrays $\mathscr{T}(P)$ with top row $\bs{\lambda}$ avoiding \raisebox{-0.5\height}{
\begin{tikzpicture}[scale=0.6,rotate=180]
    \draw[green!80,thin] (0,0)--(-0.5,.866);
    \node at (-0.5,.866) {$3$};
    \node at (0,0) {$1$};
\end{tikzpicture}} and 
\raisebox{-0.5\height}{
\begin{tikzpicture}[scale=0.6,rotate=180]
    \draw[green!80,thin] (0,0)--(-0.5,.866)--(0.5,.866)--(0,0);
    \node at (-0.5,.866) {$3$};
    \node at (0.5,.866) {$2$};
    \node at (0,0) {$3$};
\end{tikzpicture}} in $T^{(2)}$. The set of these arrays is exactly $\mc{G}_{3,n}(\bs{\lambda})$.

Now suppose that $m \geq 4$. By associativity and the definition of $g_{\bs{\xi}}$ we can write:
\begin{align*}
    \prod_{i=1}^{m-1} G_{\xi^{(i)}} &= (G_{\xi^{(1)}}G_{\xi^{(2)}})\prod_{i=3}^{m-1} G_{\xi^{(i)}} \\
    &=\sum_{\zeta} g_{(\xi^{(1)},\xi^{(2)},\zeta)} G_{\zeta^{\perp}} \prod_{i=3}^{m-1} G_{\xi^{(i)}} \\
    &=\sum_{\zeta, \xi^{(m)}} g_{(\xi^{(1)},\xi^{(2)},\zeta)}g_{(\zeta^{\perp},\xi^{(3)},\ldots,\xi^{(m)})} G_{(\xi^{(m)})^{\perp}}.
\end{align*}
By induction on $m$, both $g_{(\xi^{(1)},\xi^{(2)},\zeta)}$ and $g_{(\zeta^{\perp},\xi^{(3)},\ldots,\xi^{(m)})}$ are the (signed) sizes of the corresponding sets $\mc{G}_{3,n}(\hat{\lambda}^{(1)},\hat{\lambda}^{(2)},\mu)$ and $\mc{G}_{m-1,n}(\mu^{\dagger},\hat{\lambda}^{(3)},\ldots,\hat{\lambda}^{(m)})$ of arrays, where $\mu$ is defined by $\toprow(\xi^{(1)},\xi^{(2)},\zeta)=(\hat{\lambda}^{(1)},\hat{\lambda}^{(2)},\mu)$ and where $\toprow(\zeta^{\perp},\xi^{(3)},\ldots,\xi^{(m)})=(\mu^{\dagger},\hat{\lambda}^{(3)},\ldots,\hat{\lambda}^{(m)})$.

Consider the restriction of the splitting map $\splt$ (see \Cref{lem:splitting-interlacing}) from $\mc{T}_{m,n}(\bs{\lambda})$ to $\mc{G}_{m,n}(\bs{\lambda})$. For $T \in \mc{G}_{m,n}(\bs{\lambda})$ and $(R,S)=\splt(T)$, it is clear by construction that $R$ and $S$ also avoid the patterns from (\ref{eq:k-theory-thm-forbidden}), with indices shifted to match the supports of $R$ and $S$. Likewise, $\merge$ sends pairs of arrays avoiding these patterns to arrays avoiding the patterns. Note that for each $\zeta$, we have
\begin{align*}
    \left(d(n-d)-|(\xi^{(1)},\xi^{(2)},\zeta)|\right) &+ \left(d(n-d)-|(\zeta^{\perp},\xi^{(3)},\ldots,\xi^{(m)})|\right) \\&= 2d(n-d)-(|\zeta|+|\zeta^{\perp}|)-|\bs{\xi}| \\
    &=d(n-d)-|\bs{\xi}|.
\end{align*}
Thus, applying $\merge$ to the support-shifted sets of arrays from the previous paragraph, we conclude that $g_{\bs{\xi}}$ is $(-1)^{d(n-d)-|\bs{\xi}|}$ times $|\mc{G}_{m,n}(\bs{\lambda})|$.
\end{proof}

\begin{proof}[Proof of \Cref{thm:dual-k}]
The proof is very similar to the proof of \Cref{thm:k-theory}. Let $\mc{G}^{\ast}_{m,n}(\bs{\lambda})$ denote the set of interlacing triangular arrays appearing in the statement of \Cref{thm:dual-k}. The cases $m=1$ and $2$ are again covered by \Cref{prop:m-equals-1-interlacing,prop:m-equals-2-interlacing}. For $m=3$, we instead use a result \cite[Thm.~1' \& Rmk.~2]{wheeler-zinn-justin} of Wheeler--Zinn-Justin, which implies that $g^{\ast}_{\bs{\xi}}$ is $(-1)^{d(n-d)-|\bs{\xi}|}$ times the number of $0/1/10$-puzzles $\tilde{P}$ with boundary conditions $\bs{\xi}$ which avoid instead the \emph{$\mathit{\Delta}$-oriented} $10$-$10$-$10$ piece. By \Cref{prop:avoiding-delta-10-10-10}, the bijection $\mathscr{T}$ maps the associated $1/2/3$-puzzles $P$ to interlacing triangular arrays $\mathscr{T}(P)$ with top row $\bs{\lambda}$ avoiding \raisebox{-0.5\height}{
\begin{tikzpicture}[scale=0.6]
    \draw[green!80,thin] (0,0)--(-0.5,.866);
    \node at (-0.5,.866) {$3$};
    \node at (0,0) {$1$};
\end{tikzpicture}} and 
\raisebox{-0.5\height}{
\begin{tikzpicture}[scale=0.6]
    \draw[green!80,thin] (0,0)--(-0.5,.866)--(0.5,.866)--(0,0);
    \node at (-0.5,.866) {$3$};
    \node at (0.5,.866) {$2$};
    \node at (0,0) {$3$};
\end{tikzpicture}} in $T^{(2)}$. The set of these arrays is exactly $\mc{G}^{\ast}_{3,n}(\bs{\lambda})$. We can prove the cases $m \geq 4$ using \Cref{lem:splitting-interlacing} as we did in the proof of \Cref{thm:k-theory}.
\end{proof}

\begin{proof}[Proof of \Cref{thm:2-step}]
We again use the same proof strategy, although some care is now required to properly account for the different partial flag varieties involved. Let $\d, \bs{\lambda},$ and $\bs{w}$ be as in the theorem statement and define $\d'=(d_{m-2} \leq d_{m-1})$ and $\d''=(d_1 \leq \cdots \leq d_{m-2})$. Let $\mc{S}_{m,n}(\bs{\lambda})$ denote the set of interlacing triangular arrays from the theorem.

For $m=3$, the set $\mc{S}_{3,n}(\bs{\lambda})$ is in bijection with $0/1/10$-puzzles avoiding both the $\Delta$- and $\nabla$-oriented $10$-$10$-$10$ pieces by \Cref{prop:avoiding-delta-10-10-10,prop:avoiding-nabla-10-10-10}. According to the theorem statement, $\lambda^{(1)},\lambda^{(2)},\lambda^{(3)}$ have contents $2^{d_2}3^{n-d_2}, 1^{d_1}2^{n-d_2}3^{d_2-d_1},$ and $1^{n-d_1}3^{d_1}$, respectively. Thus, by \Cref{prop:puzzle-cryptomorphism}, the corresponding $0/1/10$-puzzles have left, right, and bottom boundary conditions $\xi^{(1)}=0^{d_2}1^{n-d_2}, \xi^{(2)}=0^{d_1}(01)^{d_2-d_1}1^{n-d_2},$ and $\xi^{(3)}=0^{d_1}1^{n-d_1}$, respectively.

By work of Halacheva--Knutson--Zinn-Justin \cite[Thm.~2]{halacheva-knutson-zinn-justin}, $0/1/10$-puzzles with no $10$-$10$-$10$ pieces and left, right, and bottom boundary conditions $\xi^{(3)},\xi^{(1)},$ and $\xi^{(2)}$ (notice the rotated indices!) compute structure constants for multiplying inside $H^{\ast}(\Fl(d_1,d_2;n))$ classes pulled back (under the natural projections $\Fl(d_1,d_2;n)\to \gr(d_i;n)$ for $i=1,2$) from $\gr(d_1,n)$ and $\gr(d_2,n)$. Since the set of allowed puzzle pieces is closed under $120^{\circ}$-rotation, and since this maps our boundary conditions to those of \cite{halacheva-knutson-zinn-justin}, this rotation gives a bijection between their puzzles and ours. In \Cref{thm:2-step}, we are considering the coefficient of (the dual of) a class pulled back from $\gr(d_1,n)$ in the product of a class pulled back from $\gr(d_2,n)$ and a Schubert class from $\Fl(d_1,d_2;n)$. Both of these geometric quantities can be expressed as the same product of three classes in $H^{\ast}(\Fl(d_1,d_2;n))$ in different orders, so they agree. Thus we have that, for $m=3$, the cardinality of the set $\mc{S}_{3,n}(\bs{\lambda})$ computes the desired structure constant $c_{\bs{w}}$. This result implies in particular that if $\mc{S}_{3,n}(\bs{\lambda})$ is nonempty, then $\ell(w^{(1)})+\ell(w^{(2)})+\ell(w^{(3)}) = \ell(w_0^{\d}) = \dim_{\C} (\Fl(\d;n))$ (in terms of puzzles, this also follows from \cite[Lem.~2.3]{Knutson-Zinn-Justin-1}).

Now suppose $m \geq 4$. We will again apply \Cref{lem:splitting-interlacing}; the $\splt$ and $\merge$ maps are again easily seen to respect the top row conditions and the avoided patterns. We need to check that the corresponding recurrence holds on the geometric side.

Suppose that $c_{\bs{w}}$ is nonzero, so $\sum_i \ell(w^{(i)}) = \dim_{\C} \Fl(\d;n)$. Then we have
\begin{equation}
\label{eq:d-step-partial-expansion}
    \prod_{i=1}^{m} \sigma_{w^{(i)}} = \sum_{u,v} c_{w^{(1)},w^{(2)}}^u c_{w^{(3)},\ldots,w^{(m)}}^v \sigma_{u} \sigma_{v},     
\end{equation}
where the sum runs over $u \in S_n^{\d'}$ and $v \in S_n^{\d''}$ such that $\ell(u)+\ell(v)= \dim_{\C} \Fl(\d;n)$. The product of two such Schubert classes is zero unless $v=u^{\vee_{\d}}$, so suppose we are in this case. We claim then that $u^{\vee_{\d'}}$ and $v^{\vee_{\d''}}$ in fact lie in $S_n^{\{d_{m-2}\}}$. We prove the first claim, the second being similar. If $d_{m-2}=d_{m-1}$, then $S_n^{\d'} =S_n^{\{d_{m-2}\}}$, so we are done. So suppose $d_{m-2}<d_{m-1}$. We know that $u^{\vee_{\d}}=v \in S_n^{\d''}$ has no descent at $d_{m-1}$. But $u^{\vee_{\d}}=w_0 u w_0(\d)$ and $u^{\vee_{\d'}}=w_0 u w_0(\d')$ differ by a permutation on $1,2,\ldots,d_{m-2}$, so the same is true of $u^{\vee_{\d'}}$. Thus $u^{\vee_{\d'}} \in S_n^{\d' \setminus \{d_{m-1}\}}=S_n^{\{d_{m-2}\}}$. Therefore the nonzero summands in (\ref{eq:d-step-partial-expansion}) contain (duals of) classes pulled back from $\gr(d_{m-2},n)$; this corresponds exactly to the conditions on the top row $\mu$ of the arrays in (\ref{eq:splitting-union}), so the result follows by induction.
\end{proof}

\begin{proof}[Proof of \Cref{thm:ssm}]
The proof again takes the same form as those of \Cref{thm:k-theory,thm:dual-k}. Now, for the $m=3$ case we use a result \cite[Thm.~5.4]{Knutson-Zinn-Justin-2} of Knutson--Zinn-Justin which implies that $(-1)^{d(n-d)-|\bs{\xi}|} s_{\bs{\xi}}$ is the number of $0/1/10$-puzzles $\tilde{P}$ with boundary conditions $\bs{\xi}$, now allowing all of the puzzle pieces. These are in bijection with the desired interlacing triangular arrays by \Cref{prop:puzzle-cryptomorphism} and \Cref{thm:psi-is-bijection}. \Cref{lem:splitting-interlacing} can again be applied to prove the cases $m \geq 4$.
\end{proof}

\section*{Acknowledgements}
We are grateful to Richard Stanley and Alexei Borodin for introducing us to this problem. We also thank Allen Knutson and Paul Zinn-Justin for their insightful comments.

\bibliographystyle{plain}
\bibliography{arxiv-v2}
\end{document}